\documentclass[12pt, reqno]{amsart}

\newcommand*{\LatexDef}{.}
\numberwithin{equation}{section}
\usepackage{kpfonts}

\usepackage{amsthm}

\usepackage{amsmath}

\usepackage{amssymb}

\usepackage{dsfont} 

\usepackage{thmtools}

\usepackage{cancel}

\usepackage{verbatim}

\usepackage{lipsum} 

\usepackage[shortlabels]{enumitem}
\setlist[enumerate]{leftmargin=*,font=\upshape,align=parleft,label=(\alph*)}
\setlist[itemize]{leftmargin=*,labelwidth=*}

\usepackage{graphicx}
\usepackage{float}

\usepackage{color}

\usepackage{xspace}

\usepackage{calrsfs}

\DeclareSymbolFont{defaultmathcal}{OMS}{zplm}{m}{n}
\DeclareSymbolFontAlphabet{\mathcal}{defaultmathcal}

\DeclareSymbolFont{handwritten}{OMS}{rsfs}{m}{n}
\DeclareSymbolFontAlphabet{\handcal}{handwritten}

\usepackage[alphabetic, nobysame]{amsrefs}

\definecolor{darkgreen}{RGB}{54,124,50}
\usepackage{hyperref}
\hypersetup{
	pdfstartview={XYZ null null 1.00}, 
	pdfpagemode=UseNone, 
	colorlinks,
	breaklinks,
	linkcolor=blue,
	urlcolor=blue, 
	anchorcolor=blue,
	citecolor=darkgreen,
}

\usepackage[capitalise]{cleveref}

\usepackage{xparse} 

\usepackage{geometry}
\usepackage{microtype}

\ExplSyntaxOn
\NewDocumentCommand{\RN}{m}
{
	\textup{ \int_to_Roman:n { #1 } }
}
\NewDocumentCommand{\rn}{m}
{
	\textup{ \int_to_roman:n { #1 } }
}
\ExplSyntaxOff

\newcommand{\N}{\mathbb{N}}
\newcommand{\Z}{\mathbb{Z}}
\newcommand{\Q}{\mathbb{Q}}
\newcommand{\R}{\mathbb{R}}

\newcommand{\0}{\mathbb{\emptyset}}
\newcommand{\w}{\infty}

\newcommand{\F}{\mathbb{F}}

\newcommand{\1}{\mathds{1}} 

\newcommand\al{\alpha}
\newcommand\be{\beta}
\newcommand{\ga}{\gamma}
\newcommand\de{\delta}
\newcommand{\e}{\varepsilon}

\newcommand\la{\lambda}
\newcommand\om{\omega}
\newcommand\si{\sigma}

\newcommand{\De}{\Delta}

\newcommand{\Si}{\Sigma}

\renewcommand{\phi}{\varphi}

\DeclareRobustCommand{\rchi}{{\mathpalette\irchi\relax}}
\newcommand{\irchi}[2]{\raisebox{\depth}{$#1\cchi$}}
\let\cchi\chi
\let\chi\rchi

\newcommand{\DC}{\mathcal{D}}

\newcommand{\IC}{\mathcal{I}}

\newcommand{\KC}{\mathcal{K}}

\newcommand{\MC}{\mathcal{M}}

\newcommand{\Diag}{\operatorname{Diag}}

\newcommand{\dom}{\operatorname{dom}}

\newcommand{\id}{\operatorname{id}}
\newcommand{\Id}{\operatorname{Id}}
\newcommand{\Int}{\operatorname{Int}}

\newcommand{\osc}{\operatorname{osc}}

\newcommand{\Pow}{\handcal{P}}

\newcommand{\proj}{\operatorname{proj}}

\newcommand{\cl}{\overline}

\newcommand{\Fin}[1][\N]{{#1^{<\N}}}

\makeatletter
\def\moverlay{\mathpalette\mov@rlay}

\def\mov@rlay#1#2{\leavevmode\vtop{%
		\baselineskip\z@skip \lineskiplimit-\maxdimen
		\ialign{\hfil$\m@th#1##$\hfil\cr#2\crcr}}}

\newcommand{\charfusion}[3][\mathord]{
	#1{\ifx#1\mathop\vphantom{#2}\fi
		\mathpalette\mov@rlay{#2\cr#3}
	}
	\ifx#1\mathop\expandafter\displaylimits\fi}
\makeatother

\renewcommand{\le}{\leqslant}

\renewcommand{\ge}{\geqslant}

\newcommand{\symdiff}{\mathrel{\triangle}}

\newcommand{\conc}{{^\smallfrown}}

\newcommand{\partialto}{\rightharpoonup}

\newcommand{\bij}{\xrightarrow{\,\smash{\raisebox{-0.65ex}{\ensuremath{\scriptstyle\sim}}}\,}}

\newcommand{\imp}{\Rightarrow}

\newcommand{\shortiff}{\Leftrightarrow}

\newcommand*{\defeq}{\mathrel{\vcenter{\baselineskip0.5ex \lineskiplimit0pt \hbox{\scriptsize.}\hbox{\scriptsize.}}}=}

\newcommand*{\defequiv}{\mathrel{\vcenter{\baselineskip0.5ex \lineskiplimit0pt
			\hbox{\scriptsize.}\hbox{\scriptsize.}}}\shortiff}

\newcommand{\bigfreeprod}{\mathop{\hspace{-1pt}\scalebox{1.5}{$\ast$}}}

\newcommand{\disjU}{\sqcup}
\newcommand{\bigdisjU}{\bigsqcup}

\newcommand{\floor}[1]{{\lfloor #1 \rfloor}}

\DeclareMathSymbol{\lqm}{\mathord}{operators}{``}
\DeclareMathSymbol{\rqm}{\mathord}{operators}{`'}

\newcommand{\Godel}{G\"{o}del\xspace}

\newcommand{\set}[1]{\left\{ #1 \right\}}

\newcommand{\gen}[1]{\left\langle #1 \right\rangle}

\newcommand{\rest}[1]{\mathord{|_{#1}}}

\newcommand{\qts}[1]{``#1''}

\theoremstyle{plain}

\newtheorem{theorem}[equation]{Theorem}
\newtheorem*{theorem*}{Theorem}

\makeatletter
\def\@empty{}
\def\ifemptycredit#1{%
	\def\tmp{#1}%
	\ifx\tmp\@empty%
	\else%
	{~(#1)}%
	\fi%
}
\makeatother

\newenvironment{namedthm*}[2][]{
	\par\medskip\noindent \textbf{#2}\ifemptycredit{#1}\textbf{.}\itshape\xspace
}{}

\crefname{prop}{Proposition}{Propositions}
\newtheorem{prop}[equation]{Proposition}
\newtheorem*{propo*}{Proposition}

\crefname{property}{Property}{Properties}

\newtheorem*{property*}{Property}

\newtheorem{lemma}[equation]{Lemma}
\newtheorem*{lemma*}{Lemma}

\crefname{claimlemma}{Claim}{Claims}

\crefname{cor}{Corollary}{Corollaries}
\newtheorem{cor}[equation]{Corollary}
\newtheorem*{cor*}{Corollary}

\crefname{obs}{Observation}{Observations}
\newtheorem{obs}[equation]{Observation}
\newtheorem*{obs*}{Observation}

\crefname{obss}{Observations}{Observations}
\newtheorem{obss}[equation]{Observations}
\newtheorem{obss*}{Observations}

\crefname{fact}{Fact}{Facts}

\newtheorem*{fact*}{Fact}

\theoremstyle{definition}

\crefname{defn}{Definition}{Definitions}
\newtheorem{defn}[equation]{Definition}
\newtheorem*{defn*}{Definition}

\newenvironment{defn**}[1][]{\par\medskip\noindent \textbf{Definition\xspace#1.}\xspace}{}

\crefname{question}{Question}{Questions}
\newtheorem{question}[equation]{Question}
\newtheorem*{question*}{Question}

\crefname{conj}{Conjecture}{Conjectures}

\newtheorem*{conj*}{Conjecture}

\crefname{example}{Example}{Examples}
\newtheorem{example}[equation]{Example}
\newtheorem*{example*}{Example}

\crefname{examples.plain}{Examples}{Examples}
\newtheorem{examples.plain}[equation]{Examples}
\newtheorem*{examples.plain*}{Examples}

\theoremstyle{remark}

\crefname{remark}{Remark}{Remarks}
\newtheorem{remark}[equation]{Remark}
\newtheorem*{remark*}{Remark}

\newenvironment{remarklike*}[2][]{\par\medskip\noindent \textit{#2}#1\textbf{.}\rmfamily\xspace}{\smallskip}

\crefname{claim+}{Claim}{Claims}
\newtheorem{claim+}[equation]{Claim}

\crefname{claim}{Claim}{Claims}

\newtheorem*{claim*}{Claim}

\crefname{subclaim}{Subclaim}{Subclaims}

\newtheorem*{subclaim*}{Subclaim}

\newenvironment{case*}[1]{\smallskip\par\noindent \textit{Case}:~#1.\rmfamily}{}

\crefname{notation}{Notation}{Notations}
\newtheorem{notation}[equation]{Notation}
\newtheorem*{notation*}{Notation}

\newtheorem*{terminology*}{Terminology}

\crefname{convention}{Convention}{Conventions}

\newtheorem*{convention*}{Convention}
\newtheorem*{conventions*}{Conventions}

\crefname{spec}{Speculation}{Speculations}

\newtheorem*{spec*}{Speculation}

\crefname{caution}{Caution}{Cautions}
\newtheorem{caution}[equation]{Caution}
\newtheorem*{caution*}{Caution}

\crefname{hypothesis}{Hypothesis}{Hypotheses}

\newtheorem*{hypothesis*}{Hypothesis}

\crefname{assumption}{Assumption}{Assumptions}

\newtheorem*{assumption*}{Assumption}

\newcommand{\fntsz}[1][11]{\fontsize{#1}{#1}\selectfont}

\newenvironment{acknowledgements}[1][11]{\medskip \fntsz[#1]\begin{trivlist}
		\item[\hskip \labelsep {\textit{Acknowledgements}.}]}{\end{trivlist}\smallskip}

\newenvironment{enumenv}[1][i]
{
	\refstepcounter{equation}
	\begin{enumerate}[\upshape(\theequation.#1)]
}
{
	\end{enumerate}\smallskip
}

\newenvironment{enumref}[2]
{
	\begin{enumerate}[\upshape(\ref*{#2}.#1), series=#2]
}
{
\end{enumerate}
}

\crefname{examples}{Examples}{Examples}

\newenvironment{examples*}[1][\alph*]
{
	\refstepcounter{equation}
	\medskip
	\noindent\textbf{Examples.}
	\medskip
	\begin{enumerate}[\bfseries(\theequation.#1),ref=(\theequation.#1),itemsep=5pt]
}
{
	\end{enumerate}
	\smallskip
}

\theoremstyle{remark}

\declaretheoremstyle[
spaceabove=\topsep, 
spacebelow=6pt,
headfont=\normalfont\itshape,
notefont=\normalfont, notebraces={(}{)},
bodyfont=\normalfont,
postheadspace=4pt,
qed=\mbox{\smaller[4]$\boxtimes$}
]{claimproofstyle}
\declaretheorem[name={Proof of Claim}, style=claimproofstyle, unnumbered]{pf}

\crefname{subsection}{Subsection}{Subsections}

\crefformat{footnote}{#2\footnotemark[#1]#3}

\crefrangeformat{enumi}{#3#1#4--#5#2#6}

\theoremstyle{plain}

\usepackage{subfiles}

\usepackage{mdframed}
\newmdenv[
leftmargin = 1cm,
rightmargin = 0pt,
skipabove = 8pt,
skipbelow = 3pt,
innerleftmargin = 8pt,
innertopmargin = 0pt,
innerbottommargin = 0pt,
innerrightmargin = 0pt,
linewidth = 3pt,
topline = false,
rightline = false,
bottomline = false
]{leftbar}

\definecolor{gris}{RGB}{90,90,90}

\definecolor{vert}{RGB}{7,126,26}

\definecolor{rougefonce}{RGB}{136,0,21}

\definecolor{purple}{RGB}{116,0,159}

\makeatletter
\def\@settitle{\begin{center}%
		\baselineskip14\p@\relax
		\bfseries
		\uppercasenonmath\@title
		\@title
		\ifx\@subtitle\@empty\else
		\\[1ex]\uppercasenonmath\@subtitle
		\footnotesize\mdseries\@subtitle
		\fi
	\end{center}%
}
\def\subtitle#1{\gdef\@subtitle{#1}}
\def\@subtitle{}
\makeatother

\newcommand{\changehidden}[1]{}

\setlist[itemize]{itemsep=5pt,leftmargin=1\parindent}

\newcommand{\edge}[1]{(#1)}

\newcommand{\bsi}{\bar{\si}}
\newcommand{\tG}{{\tilde{G}}}
\newcommand{\bG}{{\bar{G}}}
\newcommand{\tH}{{\tilde{H}}}
\newcommand{\bF}{{\bar{F}}}
\newcommand{\bR}{{\bar{R}}}

\newcommand{\tf}{{\tilde{f}}}
\newcommand{\tX}{{\tilde{X}}}
\newcommand{\tF}{{\tilde{F}}}

\newcommand{\Finw}[1]{[#1]^{<\w}}
\newcommand{\FinX}[1][X]{\Finw{#1}}
\newcommand{\FinE}[1][X]{\Finw{#1}_E}
\newcommand{\FinG}[1][X]{\Finw{#1}_G}

\newcommand{\Fx}{\operatorname{Fx}}
\newcommand{\Mv}{\operatorname{Mv}}
\newcommand{\IMv}{\operatorname{ImMv}}

\newcommand{\submod}[2]{#1_{/ #2}}
\newcommand{\Xmod}[2][X]{\submod{#1}{#2}}
\newcommand{\Gmod}[2][G]{\submod{#1}{#2}}
\newcommand{\Emod}[2][E]{\submod{#1}{#2}}

\newcommand{\wmod}[2][w]{\submod{#1}{#2}}

\newcommand{\mean}[2][]{M^{#1}_{#2}}
\newcommand{\meanw}[1]{\mean[w]{#1}}
\newcommand{\meanf}[2][]{\mean[#1]{#2}(f)}
\newcommand{\meanwf}[1]{\meanf[w]{#1}}

\newcommand{\underf}{\underline{f}}
\newcommand{\overf}{\overline{f}}

\newcommand{\MS}{\MC}
\newcommand{\MSwGf}{\MS^w_G(f)}

\newcommand{\oscmu}[1][\mu]{\osc_{#1}}
\newcommand{\oscE}[1][E]{\osc_{#1}}
\newcommand{\infmu}[1][\mu]{\operatorname{inf}_{#1}}
\newcommand{\supmu}[1][\mu]{\operatorname{sup}_{#1}}

\newcommand{\dir}[1]{\vec{#1}}

\newcommand{\Linf}[1]{\|#1\|_{_\w}}
\newcommand{\Lone}[1]{\|#1\|_{_1}}
\newcommand{\Lonew}[1]{\|#1\|_{L^1(\mu_w)}}

\newcommand{\fsr}{fsr\xspace}
\newcommand{\fsrs}{fsrs\xspace}

\newcommand{\Classes}[2][]{\operatorname{Clss}_{#1}(#2)}

\newcommand{\SatG}{\operatorname{SatG}}
\newcommand{\PackG}{\operatorname{PackG}}

\newcommand{\Jue}{
	\renewcommand{\arraystretch}{1.5}
	\displaystyle \begin{array}{l|lllllllll}
		\text{\textnormal{Player 1}} & F_0 & & F_1 & & F_2 & & ... \\
		\hline
		\text{\textnormal{Player 2}} & & \Phi_0 & & \Phi_1 & & \Phi_2 & & ...
	\end{array}
	\renewcommand{\arraystretch}{1}	
}

\newcommand{\XE}{X \disjU E}
\newcommand{\XG}{X \disjU G}
\newcommand{\Vrts}{\operatorname{Vrt}}
\newcommand{\Edges}{\operatorname{Edg}}
\newcommand{\ext}[1]{{\hat{#1}}}
\newcommand{\extD}{{\hat{D}}}
\newcommand{\extE}{{\ext{E}}}
\newcommand{\extF}{{\ext{F}}}

\newcommand{\fvp}[1][\mu]{\operatorname{fvp}_{#1}}
\newcommand{\fep}[1][\mu]{\operatorname{fep}_{#1}}

\newcommand{\EdgsBtw}[3][G]{[#2,#3]_{#1}}

\newcommand{\GpS}[2][S]{#2^{\mathord{\wedge}#1}}
\newcommand{\GS}[1][S]{\GpS[#1]{G}}
\newcommand{\HS}[1][S]{\GpS[#1]{H}}

\newcommand{\bc}{{\mathbf{c}}}
\newcommand{\br}{{\mathbf{r}}}

\newcommand{\trans}{\pitchfork}
\newcommand{\transvee}{\mathop{\boxast}}
\newcommand{\bigtransvee}{\mathop{\hspace{-1pt}\scalebox{1.5}{$\boxast$}}}

\usepackage{stmaryrd} 

\renewcommand{\theequation}{\thesection.\arabic{equation}}

\makeatletter
\renewcommand\subsection{\@startsection{subsection}{2}%
	\z@{-1.5em}{.7em}%
	{\noindent\bfseries}}
\makeatother

\makeatletter
\def\l@section{\@tocline{1}{5pt}{0pc}{}{}}
\renewcommand{\tocsection}[3]{%
	\indentlabel{\@ifnotempty{#2}{\makebox[20pt][l]{%
				\ignorespaces#1 #2.\hfill}}}\sc #3\dotfill}

\newdimen{\tocsubsecmarg}
\def\l@subsection{\@tocline{2}{3pt}{0pc}{\tocsubsecmarg}{}}
\renewcommand{\tocsubsection}[3]{%
	\indentlabel{\@ifnotempty{#2}{\makebox[30pt][l]{%
				\ignorespaces#1 #2.\hfill}}}#3\dotfill}
\makeatother
\let\oldtocsubsection=\tocsubsection
\renewcommand{\tocsubsection}[2]{\hspace{3em} \oldtocsubsection{#1}{#2}}

\newcommand{\stoptocwriting}{%
	\addtocontents{toc}{\protect\setcounter{tocdepth}{-5}}}
\newcommand{\resumetocwriting}{%
	\addtocontents{toc}{\protect\setcounter{tocdepth}{\arabic{tocdepth}}}}

\DeclareMathSizes{12}{12}{7}{7}

\title[]{Edge sliding and ergodic hyperfinite decomposition}

\author[]{Benjamin D. Miller}
\address[Benjamin D. Miller]{Kurt \Godel Research Center for Mathematical Logic, W\"{a}hringer Stra{\ss}e 25, 1090 Wien, Austria}
\email{benjamin.miller@univie.ac.at}
\thanks{The first author was supported in part by FWF Grants P28153 and P29999.}

\author[]{Anush Tserunyan}
\address[Anush Tserunyan]{Department of Mathematics, University of Illinois at Urbana-Champaign, IL, 61801, USA}
\email{anush@illinois.edu}
\thanks{The second author's research was partially supported by NSF Grant DMS-1501036.}

\begin{document}

\begin{abstract}
	We use edge slidings and saturated disjoint Borel families to give a streamlined proof of Hjorth's theorem on cost attained: if a countable p.m.p. ergodic equivalence relation $E$ is treeable and has cost $n \in \N \cup \set{\w}$ then it is induced by an a.e. free p.m.p. action of the free group $\F_n$ on $n$ generators. More importantly, our techniques give a significant strengthening of this theorem: the action of $\F_n$ can be arranged so that each of the $n$ generators alone acts ergodically.
	
	The existence of an ergodic action for the first generator immediately follows from a powerful theorem of Tucker-Drob, whose proof however uses a recent substantial result in probability theory as a black box. We give a constructive and purely descriptive set theoretic proof of a weaker version of Tucker-Drob's theorem, which is enough for many of its applications, including our strengthening of Hjorth's theorem. Our proof uses new tools, such as asymptotic means on graphs, packed disjoint Borel families, and a cost threshold for finitizing the connected components of nonhyperfinite graphs.
\end{abstract}

\maketitle

\tableofcontents


\section{Introduction}

Let $(X,\mu)$ be a standard probability space. For a measure-preserving locally countable graph $G$ on $(X,\mu)$, define its \emph{$\mu$-cost} $C_\mu(G)$ analogous to counting the number of edges of a finite graph by halving the sum of the degrees:
\[
C_\mu(G) \defeq \frac{1}{2} \int_X \deg_G(x) d\mu(x).
\]
For a measure-preserving countable Borel equivalence relation $E$ on $(X,\mu)$, define its \emph{$\mu$-cost}
\[
c_\mu(E) \defeq \inf \set{C_\mu(G) : \text{$G$ is a Borel graphing of $E$}}.
\]
Introduced by Levitt \cite{Levitt} and extensively developed by Gaboriau \cites{Gaboriau:mercuriale,Gab2,Gab3}, $\mu$-cost is a powerful isomorphism invariant for p.m.p. countable Borel equivalence relations; see also \cite{Kechris-Miller}*{Section 18}. Analogous to the fact that the free group $\F_n$ on $n$ generators has rank $n$, Gaboriau's fundamental theorem of the theory of cost \cite{Gaboriau:mercuriale} states that when $c_\mu(E) < \w$, any Borel treeing $G$ of $E$ achieves its cost, i.e. $c_\mu(E) = C_\mu(G)$. This, implies in particular that the orbit equivalence relation induced by an a.e. free p.m.p. action of $\F_n$ has cost $n$, which gives the following rigidity result in orbit equivalence: if a.e. free p.m.p. actions of $\F_n$ and $\F_m$ are orbit equivalent, then $n = m$ \cite{Gaboriau:mercuriale}*{Corollaire 1}.

A converse to this was later obtained by Hjorth in \cite{Hjorth:cost_lemma} (see also \cite{Kechris-Miller}*{Theorems 28.2 and 28.3}): 

\begin{theorem}[Hjorth 2013]\label{intro:Gregs_lemma:integer_cost}
	If a countable Borel ergodic p.m.p. equivalence relation $E$ is treeable and has cost $n \in \N \cup \set{\w}$, then it is induced by an a.e. free p.m.p. action of $\F_n$.
\end{theorem}

Our first result is a streamlined proof of Hjorth's original theorem in its full generality (\cref{Gregs_lemma} below), which we present in \cref{subsec:aperiodic-hyperfinite_factor,subsec:aperiodic-hyperfinite_decomposition}. As in Hjorth's original argument, \cref{intro:Gregs_lemma:integer_cost} reduces to accomplishing the following task $\om$-many times: given a graphing $G$ of $E$ and a subgraph $G_0 \subseteq G$, build a nontrivial finite subequivalence relation $F \subseteq E$ \emph{transverse} to $E_{G_0}$ (i.e. $F \cap E_{G_0} = \Id_X$) such that enough edges of $G$ are ``allocated'' to be ``moved'' into a graphing of $F$.

One of the two factors that make our proof conceptually clear is the isolation of a class of maps called \emph{edge slidings} (more generally, \emph{well-iterated edge slidings}) that implement the ``moving'' of edges of $G$ without affecting its connectivity (i.e. preserving $E_G$) or introducing new cycles; this is developed in \cref{sec:edge_sliding}.

The other factor, also responsible for brevity, is the use of what we call a \emph{saturated \fsr}\footnote{The abbreviation \emph{\fsr} stands for \emph{finite partial subequivalence relation} and we use it for historical reasons, even though it does not exactly match the phrase it stands for.}, that is, a Borel maximal disjoint subfamily $\Psi$ of a given family $\Phi$ of finite subsets of $X$ such that no $A \in \Psi$ can be properly extended to $A' \in \Phi$ while remaining disjoint from all other sets in $\Psi$. In \cref{subsec:saturated-fsr}, we prove the existence of such \fsrs for Borel $\Phi \subseteq \FinE$ modulo an $E$-compressible set.

Furthermore, what our proof of \cref{intro:Gregs_lemma:integer_cost} actually gives is an aperiodic hyperfinite decomposition theorem for (not necessarily acyclic) locally countable Borel graphs on standard Borel spaces up to a well-iterated edge sliding, and we roughly state it here:

\begin{theorem}[Aperiodic hyperfinite decomposition]\label{intro:decomp:aper:graphs}
	For any locally countable Borel graph $G$ on a standard Borel space $X$, up to replacing $G$ with a spanning subgraph of a well-iterated edge slide of $G$ and ignoring an $E_G$-compressible set, there are Borel partitions 
	\[
	G = \bigdisjU_{n \in \N} G_n \quad \text{ and } \quad X = \bigdisjU_{N \in \N \cup \set{\w}} X_N,
	\]
	where each $G_n$ is a hyperfinite Borel graph, each $X_N$ is Borel $E_G$-invariant (possibly empty), and, for each $N \in \N^+ \cup \set{\w}$ and $n \in \N^+$,
	
	\medskip
	\begin{tabular}{rcl}
		$n < N$ &$\imp$& $G_n \rest{X_N}$ is component-infinite,
		\\
		$n = N$ &$\imp$& $G_n \rest{X_N}$ is component-finite,
		\\
		$n > N$ &$\imp$& $G_n \rest{X_N} = \0$.
	\end{tabular}
\end{theorem}

The definitions of the terms involved are given in \cref{sec:prelims}, and more precise versions of this theorem are stated in \cref{subsec:aperiodic-hyperfinite_decomposition} as \cref{decomp:aper:eq_rel} and \cref{decomp:aper:graphs}.

\bigskip

Our second and new result is the strengthening of \cref{intro:Gregs_lemma:integer_cost} that guarantees, in addition, that every generator of $\F_n$ acts ergodically. More precisely:

\begin{theorem}[Ergodic generators for cost attained]\label{intro:ergodic_generators_Gregs_lemma:integer_cost}
	If a countable Borel ergodic p.m.p. equivalence relation $E$ is treeable and has cost $n \in \N \cup \set{\w}$, then it is induced by an a.e. free p.m.p. action of $\F_n$ such that each of the $n$ standard generators of $\F_n$ acts ergodically.
\end{theorem}

Again, what we actually prove is an ergodic hyperfinite decomposition theorem for graphs up to a well-iterated edge sliding, whose statement is roughly as follows (see also \cref{decomp:erg}):

\begin{theorem}[Ergodic hyperfinite decomposition]\label{intro:decomp:erg:graphs}
	Let $G$ be a locally countable p.m.p. ergodic Borel graph. Up to replacing $G$ with a spanning subgraph of a well-iterated edge slide of $G$, there is $N \in \N \cup \set{\w}$ with $N \le C_\mu(G)$ and a Borel partition
	\[
	G = \bigdisjU_{n = 0}^N G_n \text{ a.e.},
	\] 
	where, for each $0 \le n < N$, $G_n$ is an ergodic hyperfinite Borel graph and, if $N < \w$, $G_N$ is a component-finite Borel graph (possibly empty).
\end{theorem}

This theorem is a result of the following two theorems put together (see \labelcref{decomp:erg_over_base} and \labelcref{ergodic_hyp_slid-subgraph} for more precise statements).

\begin{theorem}[Ergodic hyperfinite decomposition over an ergodic base]\label{intro:decomp:erg_over_base:graphs}
	Let $G_0 \subseteq G$ be locally countable p.m.p. ergodic Borel graphs. Up to replacing $G$ with a well-iterated $G_0$-based\footnote{This means that in the process of sliding the edges $G_0$ remains fixed.} edge slide of $G$ and further replacing it with a spanning Borel subgraph containing $G_0$, there is $N \in \N \cup \set{\w}$ with $N \le C_\mu(G)$ and a Borel partition
	\[
	G = \bigdisjU_{n = 0}^N G_n \text{ a.e.},
	\] 
	where, for each $1 \le n < N$, $G_n$ is an ergodic hyperfinite Borel graph and, if $N < \w$, $G_N$ is a component-finite Borel graph.
\end{theorem}

\begin{theorem}[Ergodic hyperfinite edge slid subgraph]\label{intro:ergodic_hyp_slid-subgraph}
	For any locally countable p.m.p. ergodic Borel graph $G$, up to replacing $G$ with a well-iterated edge slide, there is an ergodic hyperfinite Borel subgraph $H \subseteq G$.
\end{theorem}

Just like for \cref{intro:decomp:aper:graphs}, the proof of \cref{intro:decomp:erg_over_base:graphs} is an $\om$-iteration of the following task: Build a finite subequivalence relation $F \subseteq E_G$ transverse to $E_{G_0}$, while allocating enough edges of $G$ to be moved into a graphing of $F$. However, this $F$ needs to satisfy an additional property to ensure ergodicity: for an a priori given bounded function $f : X \to \R$, make sure that the difference of averages of $f$ over any two $F$-classes is a fixed proportion (say $\frac{2}{3}$) of the oscillation of $f$; we refer to this below as the \emph{$\frac{2}{3}$-requirement}. The ergodicity of $G_0$ guarantees that, for each point $x \in X$, its $F$-class can be formed within the $G$-neighborhood of the $G_0$-connected component of $x$, without having to look farther as one has to do in \cref{intro:decomp:aper:graphs}. Hence, even though there are more conditions on $F$, the proof of \cref{intro:decomp:erg_over_base:graphs} is even simpler than that of \cref{intro:decomp:aper:graphs}; in particular, it doesn't require the use of saturated \fsrs as any maximal \fsr already does the job.

\medskip

As for \cref{intro:ergodic_hyp_slid-subgraph}, it is a weaker version of the following very powerful theorem due to Tucker-Drob \cite{TuckerDrob:erg_hyp_subgraph} proven prior to our proof of \cref{intro:ergodic_hyp_slid-subgraph}:

\begin{theorem}[Tucker-Drob 2016]\label{Robins_theorem}
	Any locally countable p.m.p. ergodic Borel graph admits an ergodic hyperfinite Borel subgraph.
\end{theorem}

However, Tucker-Drob intricately derives \cref{Robins_theorem} from a difficult recent result in probability theory by Hutchcroft and Nachmias regarding indistinguishability of the Wired Uniform Spanning Forest (WUFS) \cite{Hutchcroft-Nachmias:indistinguishability}*{Theorem 1.1}, and the derivation itself makes use of further probabilistic techniques. This motivated the present authors to find a more constructive direct argument that uses only descriptive set theoretic methods and yields \cref{intro:ergodic_hyp_slid-subgraph}, which, although weaker than \cref{Robins_theorem}, suffices for many applications, in particular those in the present paper.

This being said, our proof of \cref{intro:ergodic_hyp_slid-subgraph} is also rather complicated and required further new tools, which are interesting in their own right. Before proceeding with a further discussion, we pose a natural (given \cref{intro:ergodic_generators_Gregs_lemma:integer_cost}) question, whose answer, however, we do not know.

\begin{question}
	If a countable Borel ergodic p.m.p. equivalence relation $E$ is treeable and has cost $n \in \N \cup \set{\w}$, is it induced by an a.e. free p.m.p. action of $\F_n$ such that each group element $\ga \in \F_n \setminus \set{1}$ acts ergodically?
\end{question}

\stoptocwriting
\subsection*{Discussion of the proof of \cref{intro:ergodic_hyp_slid-subgraph}}
\resumetocwriting

Suppose $G$ is not $\mu$-hyperfinite and the underlying probability space is $(X,\mu)$. As in the proof of \cref{intro:decomp:erg_over_base:graphs}, we perform $\om$-iterations of the task of building a finite subequivalence relation $F \subseteq E_G$, allocating edges of $G$ to be moved to graph $F$, and fulfilling the $\frac{2}{3}$-requirement mentioned above for \cref{intro:decomp:erg_over_base:graphs}.

However, unlike in \cref{intro:decomp:erg_over_base:graphs}, the $F$-classes might have to be built out of points that are far apart in $G$-distance, which makes the process of allocation and moving of the edges of $G$ harder, namely: the allocated edges travel through a long $G$-path (referred to as \emph{railway} below), which then has to stay fixed throughout the future iterations, making performing the latter even harder or actually impossible. Thus, we need to enforce a bound on the cost of edges lying on these railways. This is where the non-$\mu$-hyperfiniteness of $G$ is used and we exploit it via the following invariant, which we call \emph{finitizing edge-cut price}:
\[
\fep(G) 
\defeq 
\inf \set{C_\mu(H) : H \subseteq G \text{ Borel}, G \setminus H \text{ component-finite}}.
\]
It is easy to show that if $G$ is non-$\mu$-hyperfinite, then $\fep(G) > 0$, see \cref{fep=0->fvp=0->hyperfinite}. In fact, this characterizes non-$\mu$-hyperfiniteness for locally finite graphs, see \cref{char_of_mu_hyperfiniteness_via_cuts}. 

Throughout our $\om$-iteration, we ensure that the total cost of the edges lying on the railways is less than $\fep(G)$. This is done by allowing edge sliding only within a hyperfinite Borel subgraph $H \subseteq G$ (different for each iteration) and we refer to this process as \emph{shortcutting}. But now a new task arises: prove that allowing edge sliding only within a hyperfinite Borel subgraph of $G$ is still enough to build a required subequivalence relation $F$. This is done using a new technique, called \emph{packing}, of building \fsrs with a maximality property stronger than being saturated. This is developed in \cref{subsec:packed-fsr}.

In our first attempt to build $F$, we try to make the $f$-averages on $F$-classes fall in the middle third interval of $[\inf f, \sup f]$. What packing guarantees is that either (Case 1) no more $F$-classes can be formed with this property or (Case 2) $H \rest{\dom(F)}$ is component-finite. The latter is the success case because this means that the domain of $F$ is most of $X$.

In Case 1, we make heavy use of the convexity of averages: for disjoint finite sets $U,V$, the average of $f$ over $U \disjU V$ is a convex combination of that over $U$ and $V$. This phenomenon is exploited via another invariant we introduce: the \emph{set of asymptotic means of $f$ along $G$} (see \cref{defn:asymptotic_means}), which, by convexity, is a closed interval. We develop the theory of this invariant in \cref{sec:asymptotic_means}. 

It is now (at least intuitively) clear that in Case 1, we have shrunk the set of asymptotic means by a factor of $\frac{1}{3}$ because it is convex and doesn't intersect the middle third interval of $[\inf f, \sup f]$. This is enough to finish the construction of $F$ so that it fulfills the $\frac{2}{3}$-requirement, completing the proof.

\smallskip

Lastly, to keep the paper self-contained, we give a direct proof of a pointwise ergodic theorem for hyperfinite Borel equivalence relations in \cref{subsec:ptwise_erg_for_hyperfinite}; this, of course, follows from the pointwise ergodic theorem for $\Z$-actions, but it only takes a small modification of our argument to yield a short proof of the latter theorem as well. Furthermore, in \cref{subsec:Cauchy,subsec:char_of_rel_erg}, we state and prove a characterization of (\emph{relative}) \emph{ergodicity} for hyperfinite Borel equivalence relations in terms of a \emph{Cauchy condition}, whose rough statement appeared above in the $\frac{2}{3}$-requirement.

\stoptocwriting
\subsection*{Organization}
\resumetocwriting

\cref{sec:prelims} contains most of the terminology and notation we use, as well as some basic lemmas. In \cref{sec:edge_sliding}, we develop the theory of edge sliding, which is used throughout the paper. \cref{sec:fsr} is where saturated and packed \fsrs are defined and their existence is proved. In \cref{sec:hyperfinite_basics}, we quickly survey the basics of hyperfinite Borel equivalence relations and graphs, including a warm up application of edge sliding to turning an acyclic hyperfinite Borel graph into a Borel forest of (directed) lines. \cref{sec:decomps} is where all of the decomposition theorems (aperiodic and ergodic hyperfinite) are proved, except that \cref{intro:ergodic_hyp_slid-subgraph} is used as a black-box in the proof of the ergodic hyperfinite decomposition (\cref{decomp:erg}).

The rest of the paper is dedicated to \cref{intro:ergodic_hyp_slid-subgraph}. \cref{sec:hyp_asymp_means} contains the pointwise ergodic theorem for hyperfinite Borel equivalence relations, as well as its equivalence to the aforementioned Cauchy condition. In \cref{sec:asymptotic_means}, we define asymptotic means along graphs and discuss the properties of the set of these means as an invariant of the induced equivalence relation. The definition of finitizing edge-cut price $\fep(G)$ and its relation to hyperfiniteness are given in \cref{sec:finitizing_cuts}. The process referred to above as shortcutting is described in \cref{sec:shortcutting}. Finally, the proof of \cref{intro:ergodic_hyp_slid-subgraph} is given in \cref{sec:erg_hyp_slid-subgraphs}.

\begin{acknowledgements}
	We would like to thank Anton Bernshteyn and Andrew Marks for helpful discussions and suggestions.
\end{acknowledgements}


\section{Preliminaries}\label{sec:prelims}

Throughout, let $X$ be a standard Borel space; it will often be equipped with a Borel measure $\mu$. We let $\FinX$ denote the space of finite nonempty subsets of $X$, which derives its standard Borel structure from that of $X$.

The set $\N$ of natural numbers includes $0$, of course, and we put $\N^+ \defeq \N \setminus \set{0}$.

We use standard descriptive set theoretic terminology and notation, for which we refer the reader to \cite{Kechris-Miller}. Below we set up notation and terminology that is either not standard or requires emphasis.

\subsection{Equivalence relations}

In this paper we only consider countable Borel equivalence relations.

We denote by $\Id_X$ the identity relation on $X$. 

Let $E$ denote such a relation on $X$.

For a set $A \subseteq X$, denote by $[A]_E$ the \emph{$E$-saturation of $A$}, i.e. $[A]_E = \set{x \in X : \exists y \in A \; x E y}$, and by $(A)_E$ the \emph{$E$-hull of $A$}, i.e. $(A)_E \defeq X \setminus (X \setminus [A]_E)$.

Put $\Linf{E} \defeq \max_{x \in X} |[x]_E|$ and call $E$ \emph{bounded} if $\Linf{E} < \w$.

Call a set $A \subseteq X$ \emph{$E$-related} if it is contained in a single $E$-class. Denote the collection of all $E$-related finite nonempty subsets of $X$ by $\FinE$.

For an ideal $\IC \subseteq \Pow(X)$, we write ``a statement $P$ holds modulo $\IC$'' to mean that there is a Borel set $Z \in \IC$ such that $P$ holds on $X \setminus Z$, i.e. $P$ holds once all of the objects in $P$ are restricted to $X \setminus Z$. Most often, the ideal $\IC$ is $E$-invariant, i.e. $A \in \IC \imp [A]_E \in \IC$; examples include $E$-smooth, $E$-compressible, and $\mu$-null ideals, where $\mu$ is an $E$-quasi-invariant Borel measure.

Call $C \subseteq X$ an \emph{$E$-complete set} (also called \emph{$E$-sweeping out set} and \emph{$E$-complete section}) if its intersection with every $E$-class is nonempty; in other words, $[C]_E = X$.


Throughout, we use the following instance of the Luzin--Novikov theorem \cite{bible}*{18.10}.

\begin{lemma}[Uniform enumeration relative to a point]\label{relative_enum_of_each_class}
	For any countable Borel equivalence relation $E$ on $X$, there is a sequence $(\ga_n)_n$ of Borel functions $X \to X$ such that, for each $x \in X$, $[x]_E = \set{\ga_n(x) : n \in \N}$.
\end{lemma}

The following straightforward, yet useful, characterization of compressibility will be used below without mention, see \cite{DJK}*{2.5}.

\begin{lemma}[Characterization of compressibility via smoothness]\label{compressible=supset_aperiodic+smooth}
	A countable Borel equivalence relation $E$ is compressible if and only if it contains an aperiodic smooth Borel subequivalence relation.
\end{lemma}

\subsection{Graphs}

\subsubsection{Edges and graphs}

By an \emph{edge} we mean an element $e \defeq (x,y) \in X^2$ with $x \ne y$; put $-e \defeq (y,x)$ and call it the \emph{inverse} of $e$. Call $x$ the \emph{origin} and $y$ the \emph{terminus} of $e$ and denote them by $o(e)$ and $t(e)$, respectively.

By a \emph{graph} $G$ on $X$, we simply mean any symmetric subset of $X^2$, in other words, our graphs are undirected and have no parallel edges, but may have loops\footnote{It is more common, at least in descriptive set theory, to require graphs to be irreflexive (no loops), but for our purposes we find it convenient to have equivalence relations also be graphs.}.

In this paper, we only consider \emph{locally countable} graphs on $X$, i.e. graphs whose degree is countable.

Call the set $\dom(G) \defeq \proj_X(G \setminus \Id_X)$ the \emph{domain} of $G$. For sets $A, B \subseteq X$, denote
\begin{align*}
	G \rest{A} &\defeq G \cap [A]^2
	\\
	\EdgsBtw{A}{B} &\defeq \set{e \in G : \text{one endpoint of $e$ is in $A$ and the other in $B$}}.
\end{align*}
Say that $A$,$B$ are \emph{$G$-adjacent} if $\EdgsBtw{A}{B} \ne \0$.

\subsubsection{Connectedness}

Call a graph \emph{component-finite} (resp. \emph{component-infinite}) if each of its connected components is finite (resp. infinite).

For a subset $A \subseteq X$, we say that $A$ is \emph{$G$-connected} or $G$ \emph{connects} $A$ if the graph $G \rest{A}$ is connected. 

\begin{caution}
	``$A$ is $G$-connected'' is stronger than ``$A$ is contained in a $G$-connected component''.
\end{caution}
Let $\FinG$ denote the set of finite $G$-connected nonempty subsets of $X$.

Denote by $E_G$ the equivalence relation on $X$ \emph{induced} by $G$, i.e. of being in the same $G$-connected component. However, for a set $A \subseteq X$ and a point $x \in X$, we write $[A]_G$ and $[x]_G$ instead of $[A]_{E_G}$ and $[x]_{E_G}$.

For a graph $G$ and an equivalence relation $E$ on $X$, we say that $G$ is a \emph{graphing} of $E$ or $G$ \emph{graphs} $E$ if $E_G = E$. Furthermore, say that $G$ is a \emph{supergraphing} of $E$ or $G$ \emph{supergraphs} $E$ if $G \cap E$ is a graphing of $E$, equivalently, $G$ connects every $E$-class.

Call a subgraph $H \subseteq G$ \emph{spanning} if $E_H = E_G$.

\subsubsection{Directed graphs}

A \emph{directed graph} is any subset of $X^2$. For an (undirected) graph $G \subseteq [X]^2$, call a graph $\dir{G}$ a \emph{directing} of $G$ if $\dir{G} \subseteq G$ and for each edge $(x,y) \in G$ exactly one of $(x,y), -(x,y)$ is in $\dir{G}$. Conversely, the \emph{undirecting} of a directed graph $\dir{G}$ is simply its symmetrization $-\dir{G} \cup \dir{G}$.

\subsubsection{Walks, paths, cycles}

By a \emph{walk} $W$ in a graph $G$ we mean a sequence $e_0,e_1,...,e_n$ of edges such that, for each $i < n$, $t(e_i) = o(e_{i+1})$. Call $W$ \emph{cyclic} if $o(e_0) = t(e_n)$. A \emph{backtracking} in $W$ is an index $i \le n$ such that $e_i = - e_{i+1}$, where $e_{n+1} \defeq e_0$.

A \emph{path} $P \defeq e_0,e_1,...,e_n$ is a walk without any backtracking. An endpoint of any edge of $P$ is referred to a \emph{vertex of $P$}. Call $P$ \emph{simple} if no vertex appears on it more than once. For $x,y \in X$, say that $P$ is \emph{from $x$ to $y$} or $P$ \emph{connects $x$ and $y$} if $o(e_0) = x$ and $t(e_n) = y$. A \emph{cycle} is a cyclic path.

\subsection{Transversality}

\begin{defn}[Transversality]\label{defn:transversality}
	Let $E_0, E_1, F$ be equivalence relations on $X$. We say that $E_0$ and $E_1$ are \emph{transverse}, and write $E_0 \trans E_1$, if $E_0 \cap E_1 = \Id_X$. More generally, we say that $E_0$ and $E_1$ are \emph{transverse over $F$}, and write $E_0 \trans_F E_1$, if $E_0 \cap E_1 = F$.
\end{defn}

\begin{defn}[Increasing transversality]\label{increasing_transversality}
	Let $N \le \om$ and let $(E_n)_{n < N}$ be a sequence of equivalence relations on $X$. Say that $(E_n)_{n < N}$ is \emph{increasingly transverse} if, for each $n < N$,
	\[
	E_n \trans \bigvee_{k < n} E_k.
	\]
	We denote the join of a sequence $(E_n)_{n < N}$ by $\bigtransvee_{n < N} E_n$ if it is increasingly transverse.
\end{defn}

\begin{defn}[Transversality for graphs]
	Call graphs $G_0, G_1$ \emph{transverse}, and write $G_0 \trans G_1$, if $E_{G_0}$ and $E_{G_1}$ are transverse. For $N \le \om$, call a sequence $(G_n)_{n < N}$ of graphs on $X$ \emph{increasingly transverse} if the sequence $(E_{G_n})_{n < N}$ is increasingly transverse. Call $(G_n)_{n < N}$ an \emph{increasingly transverse spanning partition} of a graph $G$ if it is increasingly transverse and $\bigdisjU_{n \in N} G_n$ is a spanning subgraph of $G$.
\end{defn}

\subsection{Quotients by smooth equivalence relations}

Let $X$ be a standard Borel space and $F$ a smooth Borel countable equivalence relation on $X$. In this paper, all quotients by smooth Borel equivalence relations are \emph{concrete}, i.e. we fix a Borel selector $s_F : X \to X$ for $F$ and identify the quotient space $\Xmod{F}$ with $s_F(X)$; in particular, $\Xmod{F}$ is a subset of $X$.

When $F$ is finite and $X$ is equipped with an $F$-invariant Borel measure $\mu$, instead of the usual quotient measure $\mu / F$ on $\Xmod{F}$ (i.e. the push-forward of $\mu$ under the factor map), we use the restriction $\mu \rest{\Xmod{F}}$ on $\Xmod{F}$. It is clear that the measures $\mu \rest{\Xmod{F}}$ and $\mu / F$ are different, and in fact, the $F$-invariance of $\mu$ implies that $d (\mu / F) = |[x]_F| \, d (\mu \rest{\Xmod{F}})$.

For an equivalence relation $E \supseteq F$, the \emph{concrete quotient equivalence relation $\Emod{F}$} is simply the restriction of $E$ to $\Xmod{F} \subseteq X$.

For a graph $G \subseteq X^2$, define its \emph{concrete quotient} (or a \emph{graph minor}) $\Gmod{F}$ as the pushforward of $G$ via the map $s_F^{(2)} : (x,y) \mapsto \big(s_F(x), s_F(y)\big)$, i.e., for $u,v \in \Xmod{F}$,
$$
u (\Gmod{F}) v \defequiv \exists x \in s_F^{-1}(u), y \in s_F^{-1}(v) \; x G y.
$$
When $G$ is an equivalence relation, this definition of $\Gmod{F}$ coincides with the one above.

\begin{obs}
	If $G$ supergraphs $F$, then every $\Gmod{F}$-connected set $A \subseteq \Xmod{F}$ pulls back to a $G$-connected set, namely, $[A]_F$. In particular, if $G$ is acyclic then $\Gmod{F}$ is also acyclic.
\end{obs}

\begin{lemma}\label{smooth_quotient_has_inverse}
	The quotient map $s_F^{(2)}$ restricted to $G$ has a Borel right-inverse $i : \Gmod{F} \to G$. In particular, if $\mu$ is an $E_G$-invariant Borel measure on $X$, then any Borel subgraph $H' \subseteq \Gmod{F}$ lifts to a Borel subgraph $H \subseteq G$ of equal $\mu$-cost, i.e. $C_\mu(H) = C_\mu(H') = C_{\mu \rest{\Xmod{F}}}(H')$.
\end{lemma}
\begin{proof}
	Let $(\ga_n)_n$ be as in \cref{relative_enum_of_each_class} when applied to $F$. It is enough to define $i$ on a Borel directing $\dir{\Gmod{F}}$ of $G$ and extend it to $G$ symmetrically. For each edge $(u,v) \in \dir{\Gmod{F}}$, define $i(u,v) \defeq (\ga_n(u), \ga_m(v))$, where $(n,m) \in \N^2$ is the lexicographically least pair for which $(\ga_n(u), \ga_m(v)) \in G$.
\end{proof}

\subsection{Weight functions}

We refer to any real-valued non-negative function $w : X \to [0,\w)$ as a \emph{weight function} and by a \emph{$w$-weight} of a countable set $A \subseteq X$ we mean
$
|A|_w \defeq \sum_{x \in A} w(x).
$

Let $(X,\mu)$ be a measure space. For any measurable weight function $w : X \to [0,\w)$, define a measure $\mu_w$ by setting $d \mu_w \defeq w d \mu$, i.e. for every measurable set $A \subseteq X$, $\mu_w(A) \defeq \int_A w d \mu$.

Lastly, for a finite Borel equivalence relation $F$ on $X$, define the \emph{concrete quotient of $w$} as the function $\wmod{F} : \Xmod{F} \to [0, \w)$ given by $\wmod{F}[x] \defeq |[x]_F|_w$ for $x \in \Xmod{F} \subseteq X$.

\subsection{Miscellaneous}

\subsubsection{$\e$-equality}

For reals $a,b$ and $\e \ge 0$, write $a \approx_\e b$ to mean $|a - b| \le \e$. 

\subsubsection{$\mu$-$\e$ sets}

For a standard measure space $(X,\mu)$ and $\e > 0$, we say that a measurable set $A \subseteq X$ is \emph{$\mu$-$\e$} if $\mu(A) \le \e$; consequently, we say that $A$ is \emph{$\mu$-co-$\e$} if $A^c$ is $\mu$-$\e$.

\subsubsection{Functions}

For a function $f : X \to Y$ and $A \subseteq X$, we denote by $f \rest{A}$ its restriction to the domain $A$. By a partial function $g : X \partialto Y$ we simply mean a function $g : X' \to Y$ for some $X' \subseteq X$ and refer to $X'$ as the domain of $g$, denoted by $\dom(g)$. Call $g$ \emph{entire} if $\dom(g) = X$.

\subsubsection{Set operations and relations}

For sets $A,B$, we say that $A$ \emph{meets} or \emph{intersects} $B$ if $A \cap B \ne \0$. We write $C = A \disjU B$ for a \emph{disjoint union}, i.e. to mean that $A \cap B = \0$ and $C = A \cup B$.

\subsubsection{Intervals}\label{subsubsec:intervals}

By an \emph{interval}, we mean any convex subset of $\R$. For a nonempty interval $I$, put $|I| \defeq \sup I - \inf I$ and call it its \emph{length}.

For intervals $I,J$ and a real $a$, write $I \le a$ (resp. $I < a$) if $\sup I \le a$ (resp. $\sup I < a$); the notation $a \le I$ and $a < I$ is defined analogously. Write $I \le J$ (resp. $I < J$) if $\sup I \le \inf J$ (resp. $\sup I < \inf J$).

Finally, for an interval $I$, let $I_-$ and $I_+$ denote the (possibly empty) left and right connected components of $\R \setminus I$. We write $I_\pm$ to denote one of the connected components of $\R \setminus I$, without specifying which one. In other words, every instance of $I_\pm$ is equal to either $I_-$ or $I_+$, but not their union.


\section{Edge sliding}\label{sec:edge_sliding}

Throughout this section, let $X$ be a nonempty standard Borel space.

\subsection{Basic edge sliding}

For edges $e \defeq \edge{u_0,u_1}, e' \defeq \edge{v_0,v_1} \in X^2$, say that paths $P_0, P_1$ \emph{connect} the endpoints of $e$ and $e'$ if $P_i$ connects $u_i$ to $v_i$, for each $i \in \set{0,1}$.

\begin{defn}
	For edges $e,e' \in X^2$ and a graph $R \subseteq X^2$, we say that \emph{$e$ slides into $e'$ along $R$} if there are paths $P_0, P_1$ in $R$ that connect the endpoints of $e$ and $e'$.
\end{defn}

\begin{obss}\label{sliding_is_an_eq_rel}
	Sliding is reflexive, symmetric, and transitive. More precisely, for edges $e_0,e_1,e_2$ and graphs $R,R' \subseteq X^2$,
	\begin{enumref}{a}{sliding_is_an_eq_rel}
		\item \textup{(Reflexivity)} $e_0$ slides into $e_0$ along $R$;
		
		\item \label{item:sliding_is:symmatric} \textup{(Symmetry)} if $e_0$ slides into $e_1$ along $R$, then $e_1$ slides into $e_0$ along $R$;
		
		\item \label{item:sliding_is:transitive} \textup{(Transitivity)} if $e_0$ slides into $e_1$ along $R$ and $e_1$ slides into $e_2$ along $R'$, then $e_0$ slides into $e_2$ along $R \cup R'$.
	\end{enumref}
\end{obss}

\begin{obss}\label{sliding_preserves}
	Let $e_0 \in X^2$ slide into $e_1 \in X^2$ along a graph $H$.
	\begin{enumref}{a}{sliding_preserves}
		\item \label{item:sliding_preserves:connectivity} $E_{H \cup \set{e_0}} = E_{H \cup \set{e_1}}$.
		
		\item \label{item:sliding_preserves:transversality} $e_0 \notin E_H$ if and only if $e_1 \notin E_H$.
	\end{enumref}
\end{obss}

For $\si : X^2 \to X^2$, put 
\begin{align*}
	\Fx(\si) &\defeq \set{e \in X^2 : \si(e) = e},
	\\
	\Mv(\si) &\defeq X^2 \setminus \Fx(\si),
	\\
	\IMv(\si) &\defeq \si(\Mv(\si)).
\end{align*}

\begin{obs}\label{Mv_cap_im_subset_IMv}
	For $\si : X^2 \to X^2$, $\Mv(\si) \cap \si(X^2) \subseteq \IMv(\si)$.
\end{obs}

\begin{defn}
	Call $\si : X^2 \to X^2$ an \emph{edge-operator} if it is symmetric, i.e. $\si(u,v) = -\si(v,u)$, and $\Fx(\si) \supseteq \Diag(X^2)$. For such $\si$, we say that
	\begin{itemize}
		\item $\si$ \emph{moves} a graph $G_0 \subseteq X^2$ \emph{into} a graph $G_1 \subseteq X^2$ if $\Mv(\si) \subseteq G_0$ and $\si(G_0) \subseteq G_1$.
		
		\item $\si$ is \emph{connectivity preserving} for a graph $G$ if $E_{\si(G)} = E_G$.
	\end{itemize}
\end{defn}

\begin{defn}\label{defn:edge_sliding}
	An edge-operator $\si$ on $X$ is called an \emph{edge sliding} \emph{along a graph $R \subseteq X^2$} if $\Fx(\si) \supseteq R$ and every edge $e \in X^2$ slides into $\si(e)$ along $R$.
	
	\begin{itemize}
		\item Call $\si$ an \emph{edge sliding} if it is an edge sliding along $R$ for some graph $R \subseteq X^2$, to which we refer as a \emph{railway} for $\si$.
		
		\item For graphs $H, G$, we say that $\si$ is an \emph{$H$-based edge sliding of $G$} if $\Mv(\si) \subseteq G$, $H \subseteq \Fx(\si)$, and $H \cup G$ (equivalently, $H \cup \si(G)$) contains a railway for $\si$. When ``$H$-based'' is omitted, we mean that $H = \0$, so $G$ contains a railway for $\si$.
	\end{itemize}
\end{defn}

\begin{obss}\label{edge_sliding_prop}
	Let $G,H$ be graphs on $X$ and let $\si$ be an $H$-based edge sliding of $G$.
	
	\begin{enumref}{a}{edge_sliding_prop}
		\item \label{item:edge_sliding:preserves_connectivity} $\si$ is connectivity preserving for $G \cup H$.
		
		\item \label{item:railway_any_spanning_subgraph} If $R$ is a railway for $\si$, then any spanning subgraph of $R$ is also a railway for $\si$.
		
		\item \label{item:reversing_edge_sliding} If $\si \rest{G}$ is one-to-one, then there is an $H$-based edge sliding $\tau$ of $\si(G)$ reversing the action of $\si$ on $G$, namely,
		\[
		\tau(e) \defeq 
		\begin{cases}
		\textup{the unique $e' \in G \cap \si^{-1}(e)$} 
		& 
		\textup{if $e \in \si(G)$}
		\\
		e 
		& 
		\textup{otherwise}.
		\end{cases}
		\]
	\end{enumref}
\end{obss}

\begin{lemma}\label{edge_sliding_preserves_acyclicity}
	For any edge sliding of a graph $G$ on $X$, if $G$ is acyclic, then $\si(G)$ is also acyclic. If $\si \rest{G}$ is one-to-one, then the converse also holds.
\end{lemma}
\begin{proof}
	The second statement follows from the first due to \labelcref{item:reversing_edge_sliding}.
	
	Towards the contrapositive of the first statement, let $C$ be a simple cycle in $\si(G)$. If $C \subseteq R$ then $G$ contains $C$ and we are done, so suppose $C \nsubseteq R$. For each edge $e' \in C$, choose a $\si$-preimage $e \in G$ and paths $P_0(e),P_1(e) \subseteq R$ that connect the endpoints of $e$ and $e'$. Replacing each $e' \in C$ with the path $P_0(e) \conc e \conc P_1(e)$, we obtain a cyclic walk $\tilde{C}$ in $G$. For each $e' \in C \setminus R$, $e$ occurs in $\tilde{C}$ exactly once, so after deleting all backtrackings from $\tilde{C}$, we obtain a cycle in $G$ that contains $\set{e : e' \in C \setminus R}$, and is thus nontrivial.
\end{proof}

\begin{remark}
	In the setting of multigraphs, the second statement of \cref{edge_sliding_preserves_acyclicity} holds without the assumption of injectivity.
\end{remark}

\begin{lemma}[Edge sliding into an $F$-complete set]\label{edge_sliding_into_complete-sect}
	Let $F \subseteq E$ be countable Borel equivalence relations and let $G \subseteq E \setminus F$ be a Borel graph. For any Borel $F$-complete set $Y \subseteq X$, there is a Borel edge sliding along $F$ that moves $G$ into $Y^2$.
\end{lemma}
\begin{proof}
	Using \cref{relative_enum_of_each_class}, we get a Borel function $\pi : X \to Y$ that is the identity on $Y$ and $\pi(x) F x$ for each $x \in X$. Define an edge-operator $\si : X^2 \to X^2$ by mapping each $(x,y) \in G$ to $\big(\pi(x), \pi(y)\big)$ and setting $\si$ to be the identity outside of $G$. This does it.
\end{proof}

\subsection{Iterated edge sliding}

For a set $Y$, an ordinal $\la$, and a sequence $(y_\al)_{\al < \la} \subseteq Y$, we say that \emph{$\displaystyle\lim_{\al \to \la} y_\al$ exists}, and write $\displaystyle\lim_{\al \to \la} y_\al = y$, if there is $\be < \la$ such that $y_\be = y_\ga = y$ for all $\ga \in [\be, \la)$. For a sequence $(\si_\al)_{\al < \la}$ of maps $Y \to Y$, define its \emph{composition up to $\al \le \la$} to be the partial function $\bsi_\al : Y \partialto Y$ defined by
\[
\bsi_\al(y) \defeq 
\begin{cases}
y &\text{if } \al = 0
\\
\si_\be(\bsi_\be(y)) &\text{if } \al = \be + 1
\\
\displaystyle\lim_{\be \to \al} \bsi_\be(y) &\text{if $\al$ is a limit ordinal and $\displaystyle\lim_{\al \to \la} y_\al$ exists}
\\
\text{undefined} &\text{otherwise}.
\end{cases}
\]
Call the sequence $(\si_\al)_{\al < \la}$ \emph{composable} if $\bsi_\al$ is entire for each $\al \le \la$, and refer to $\bsi_\la$ as its composition.

\begin{defn}\label{defn:iterated_edge_sliding}
	For graphs $H, G$ on $X$, an \emph{$H$-based iterated edge sliding of $G$} is an edge-operator $\si$ on $X$ that is the composition of some composable sequence $(\si_\al)_{\al < \la}$, where, for each $\al < \la$, $\si_\al$ is an $H$-based edge sliding of $\bsi_\al(G)$.
	
	\begin{itemize}
		\item We refer to the sequence $(\si_\al)_{\al < \la}$ as a \emph{witnessing iteration for $\si$} and to $\la$ as its \emph{length}. The minimum over the lengths of witnessing iterations for $\si$ is called the \emph{rank} of $\si$.
		
		\item Call $\tG$ an \emph{$H$-based iterated edge slide of $G$} if it is the image of $G$ under some $H$-based iterated edge sliding of $G$.
		
		\item If ``$H$-based'' is omitted, we mean that $H = \0$.
	\end{itemize}
\end{defn}

\begin{remark}
	Anton Bernshteyn pointed out that, in the setting of multigraphs, one could show that any iterated edge sliding is of rank at most $\om$.
\end{remark}

\begin{prop}[Properties preserved by iterated edge slidings]\label{iterated_edge_sliding_preserves}
	Iterated edge slidings preserve acyclicity, and can only reduce connectivity and cost. More precisely, for any iterated edge slide $\tG$ of a graph $G$ on $X$, we have:
	\begin{enumref}{a}{iterated_edge_sliding_preserves}
		\item \label{item:iterated_edge_sliding_preserves:subconnectivity} $E_\tG \subseteq E_G$.
		
		\item \label{item:iterated_edge_sliding_preserves:acyclicity} If $G$ is acyclic, then $\tG$ is also acyclic.
		
		\item \label{item:iterated_edge_sliding_preserves:cost} If $(X,\mu)$ is a standard measure space and $\tG$ a Borel iterated edge slide of a Borel measure-preserving graph $G$ on $X$, then $C_\mu(\tG) \le C_\mu(G)$.
	\end{enumref}
\end{prop}
\begin{proof}
	\labelcref{item:iterated_edge_sliding_preserves:subconnectivity,item:iterated_edge_sliding_preserves:acyclicity} follow by a straightforward induction on the rank of the iterated edge sliding, using \cref{item:edge_sliding:preserves_connectivity,edge_sliding_preserves_acyclicity}, respectively. \labelcref{item:iterated_edge_sliding_preserves:cost} is due to $E_G$ being measure-preserving and the iterated edge sliding being a Borel transformation of $E_G$; see, for example, \cite{Kechris-Miller}*{16.1 and 16.2}.
\end{proof}

\begin{remark}
	In the setting of multigraphs, iterated edge slidings preserve acyclicity and cost even without the assumption of injectivity.
\end{remark}

\subsection{Well-iterated edge sliding}

\begin{obs}\label{finite_itaration_connectivity_preserving}
	An iterated edge sliding of a graph $G$ of finite rank is connectivity-preserving for $G$.
\end{obs}

However, the iterated edge slidings of $G$ of infinite rank may not preserve the connectivity of $G$ and here is how it may happen: $\si_0$ slides an edge $e \in G$ along a railway $R_0$, then $\si_1$ slides some edges of $R_0$ along a railway $R_1$, then $\si_2$ slides some edges of $R_1$ along a railway $R_2$, and so on, so the distance between the endpoints of $e$ gets larger and larger, becoming infinite after $\om$-iterations. Thus, we restrict to only the following kinds of iterations that guarantee preservation of connectivity.

\begin{defn}\label{defn:conserv_comp}
	For a countable ordinal $\la$ and graphs $G, H$, a sequence $(\si_\al)_{\al < \la}$ of edge-operators is called \emph{$H$-based $G$-conservative} if there is a pairwise disjoint sequence $(G_\al)_{\al < \la}$ of subgraphs of $G$ such that, for each $\al < \la$, putting $\bG_\al \defeq \bigdisjU_{\be < \al} G_\al$, $\si_\al$ is an $\big(H \cup \bsi_\al(\bG_\al)\big)$-based iterated edge sliding of $G_\al$ that is connectivity preserving for $H \cup \bsi_\al(\bG_\al) \cup G_\al$.
\end{defn}

Simple observations are in order, which may be used below without mention.

\begin{obss}\label{conserv_comp_properties}
	In the notation of \cref{defn:conserv_comp}, for any $\al \le \be$,
	\begin{enumref}{a}{conserv_comp_properties}
		\item $\bsi_\al(G_\be) = G_\be$; whence, $\bsi_\al(\bG_{\al + 1}) = \bsi_\al(\bG_\al) \cup G_\al$.
		
		\item \label{item:conserv_comp_properties:fixes_railways} $\Fx(\si_\be) \supseteq \bsi_\al(\bG_\al)$; whence, $\bsi_\be(\bG_\al) = \bsi_\al(\bG_\al)$ and $(\si_\al)_{\al < \la}$ is a composable sequence.
		
		\item The sequence $(\si_\al)_{\al < \la}$ moves each edge $e \in X^2$ at most once, i.e. there is $\be < \la$ such that $e \in \Fx(\si_\al)$ and $\si_\be(e) \in \Fx(\si_\ga)$ for all $\al \in[0, \be)$ and $\ga \in (\be, \la)$.
		
		\item \label{item:conserv_comp_properties:Mv=cupMV} $\Mv(\si) = \bigcup_{\al < \la} \Mv(\si_\al)$.
	\end{enumref}
\end{obss}

\begin{prop}\label{conserv_connectivity_preserving}
	A composition $\si$ of a $G$-conservative sequence $(\si_\al)_{\al < \la}$ of iterated edge slidings is connectivity preserving for $G$.
\end{prop}
\begin{proof}
	In the notation of \cref{defn:conserv_comp}, for each $\al < \la$, $\si_\al$ is connectivity preserving for $\bsi_\al(\bG_\al) \cup G_\al = \bsi_\al(\bG_{\al + 1})$, so each edge $e \in \Mv(\si_\al)$ slides into $\si_\al(e)$ along a path $P \subseteq \bsi_{\al+1}(\bG_{\al + 1})$.
	\labelcref{item:conserv_comp_properties:fixes_railways} ensures that $P$ is pointwise fixed by all $\si_\be$ with $\be > \al$, so the endpoints of $e$ remain connected for the rest of the iteration.
\end{proof}

\begin{defn}
	For graphs $G,H$ on $X$, call an edge-operator $\si$ an \emph{$H$-based well-iterated edge sliding of $G$} if it is a composition of an $H$-based $G$-conservative sequence of edge slidings.
	
	\begin{itemize}
		\item Call a graph $\tG$ an \emph{$H$-based well-iterated edge slide of $G$} if $\tG = \si(G)$ for some $H$-based well-iterated edge sliding $\si$ of $G$.
		
		\item As above, if ``$H$-based'' is omitted, then $H = \0$.
	\end{itemize}
\end{defn}

\begin{prop}\label{well-iterated_properties}
	Let $G,H$ be graphs on $X$.
	\begin{enumref}{a}{well-iterated_properties}
		\item \label{item:well-iterated_properties:closed_under_conserv_comp} The collection of $H$-based well-iterated edge slidings of $G$ is closed under $H$-based $G$-conservative compositions.
		
		\item \label{item:well-iterated_properties:connectivity_preserving} Any $H$-based well-iterated edge sliding of $G$ is connectivity preserving for $G \cup H$.
	\end{enumref}
\end{prop}
\begin{proof}
	\labelcref{item:well-iterated_properties:closed_under_conserv_comp} is by definition, and
	\labelcref{item:well-iterated_properties:connectivity_preserving} follows from \labelcref{item:edge_sliding:preserves_connectivity} and \labelcref{conserv_connectivity_preserving}.
\end{proof}

\begin{lemma}\label{replace_base_any_spanning_subgraph}
	Let $H, G$ be graphs on $X$. If $\si$ is an $H$-based well-iterated edge sliding of $G$, then $\si$ is also an $H'$-based well-iterated edge sliding of $G$ for any spanning subgraph $H'$ of $H$.
\end{lemma}
\begin{proof}
	Follows by induction from \labelcref{item:railway_any_spanning_subgraph}.
\end{proof}

\subsection{Graphing equivalence relations}

\begin{defn}\label{defn:into_and_graphs}
	Let $H,G, \tG$ be graphs and $F' \subseteq F$ equivalence relations on $X$. We say that an $H$-based (well-)iterated edge sliding $\si$ of $G$
	\begin{itemize}
		\item is \emph{into} $\tG$ if $\si$ is actually an $H$-based (well-) iterated edge sliding $\si$ of $G \cap \si^{-1}(\tG)$; in particular, $\IMv(\si) \subseteq \tG$.
		
		\item \emph{graphs $F$ over $F'$} if $H \cup \si(G) \cup F'$ is a supergraphing of $F$. We omit ``over $F'$'' if $F'$ is the identity relation.
	\end{itemize}
\end{defn}

\begin{prop}[Connecting increasing unions using disjoint graphs]\label{connecting_unions_with_disjoint}
	Let $H, G$ be locally countable Borel graphs on $X$. Let
	\begin{itemize}[\scriptsize$\triangleright$]
		\item $(F_n)_{n \in \N}$ be an increasing sequence of Borel equivalence relations on $X$, where $F_0 = \Id_X$,
		
		\item $(G_n)_{n \in \N}$ be a pairwise disjoint sequence of Borel subgraphs of $G$, and
		
		\item $(\si_n)_{n \in \N}$ be a sequence of Borel edge-operators, where $\si_n$ is an $(H \cup F_n)$-based well-iterated edge sliding of $G_n$ into $F_{n+1}$ that graphs $F_{n+1}$ over $F_n$.
	\end{itemize}
	Then $(\si_n)_{n \in \N}$ is an $H$-based $G$-conservative sequence, whose composition $\bsi_\om$ is a Borel $H$-based well-iterated edge sliding of $G$ into $F \defeq \bigcup_{n \in \N} F_n$ that graphs every $F_n$ (hence, also $F$).
\end{prop}
\begin{proof}
	By replacing $G_n$ with $G_n \cap \si_n^{-1}(F_{n+1})$, we may assume without loss of generality that $\si_n(G_n) \subseteq F_{n+1}$. We aim to show that the sequence $(G_n)_{n \in \N}$ witnesses the $G$-conservativeness of $(\si_n)_{n \ge 0}$. Put $\bG_n \defeq \bigcup_{k < n} G_k$.
	
	\begin{claim+}\label{claim:connecting_unions_with_disjoint:comp_calc}
		For each $n > m$, $\bsi_n(G_m) = \si_m(G_m)$; in particular, $\bsi_n(\bG_m) = \bigdisjU_{k < m} \si_k(G_k) \subseteq F_{m+1}$.
	\end{claim+}
	\begin{pf}
		Because $\Mv(\si_n) \subseteq G_k$ and the graphs $G_k$ are pairwise disjoint, $G_m \subseteq \Fx(\bsi_m)$, so $\bsi_m(G_m) = G_m$, and hence, $\bsi_{m+1}(G_m) = \si_m(G_m)$. By our assumption, $\si_m(G_m) \subseteq F_{m+1} \subseteq \Fx(\si_k)$ for each $k > m$, so $\bsi_n(G_m) = \si_{n-1} \circ ... \circ \si_{m+1} \circ \si_m(G_m) = \si_m(G_m)$.
	\end{pf} 
	
	It now follows that $\bsi_n(\bG_n)$ graphs $F_n$, so, by \cref{replace_base_any_spanning_subgraph}, $\si_n$ is an $\big(H \cup \bsi_n(\bG_n)\big)$-based well-iterated edge sliding of $G_n$. Hence, the sequence $(G_n)_{n \in \N}$ indeed witnesses the fact that $(\si_n)_{n \in \N}$ is $H$-based $G$-conservative. \cref{claim:connecting_unions_with_disjoint:comp_calc} also implies that $\bsi_{\om}$ graphs every $F_n$.
\end{proof}

\begin{prop}[Connecting increasing unions using image graphs]\label{connecting_unions_using_images}
	Let $H, G$ be locally countable Borel graphs on $X$. Let
	\begin{itemize}[\scriptsize$\triangleright$]
		\item $(F_n)_{n \ge 0}$ be an increasing sequence of Borel equivalence relations on $X$, where $F_0 = \Id_X$, and
		
		\item $(\si_n)_{n \ge 0}$ be a sequence of Borel edge-operators, where $\si_n$ is an $(H \cup F_n)$-based well-iterated edge sliding of $\bsi_n(G)$ into $F_{n+1}$ that graphs $F_{n+1}$ over $F_n$.
	\end{itemize}
	Then $(\si_n)_{n \in \N}$ is an $H$-based $G$-conservative sequence, whose composition $\si$ is a Borel $H$-based well-iterated edge sliding of $G$ into $F \defeq \bigcup_{n \in \N} F_n$ that graphs every $F_n$ (hence, also $F$).
\end{prop}
\begin{proof}	
	By \cref{defn:into_and_graphs}, $\si_n$ is an $(H \cup F_n)$-based well-iterated edge sliding of
	\[
	G_n \defeq \big(\bsi_n(G) \cap \si_n^{-1}(F_{n+1})\big) \setminus F_n
	\]
	into $F_{n+1}$, so it is enough to show that these $G_n$ are pairwise disjoint subgraphs of $G$ because then \cref{connecting_unions_with_disjoint} applies. By \cref{Mv_cap_im_subset_IMv}, 
	\[
	\Mv(\bsi_n) \cap \bsi_n(X^2) \subseteq \IMv(\bsi_n) \subseteq F_n,
	\]
	so $G_n \cap \Mv(\bsi_n) = \0$ because $G_n \subseteq \bsi_n(G) \setminus F_n$; in particular $G_n \subseteq \bsi_n(G) \cap \Fx(\bsi_n) \subseteq G$. Also, because $\Mv(\si_n) \subseteq G_n$, it follows that $\Mv(\si_n) \cap \Mv(\bsi_n) = \0$, so induction on $n$ gives
	\[
	\Mv(\bsi_n) = \bigdisjU_{m < n} \Mv(\si_m).
	\]
	For each $n > m \ge 0$, on one hand we have
	\[
	G_n \subseteq \big(G \setminus \Mv(\bsi_n)\big) \setminus F_n = G \setminus \big(\bigdisjU_{k < n} \Mv(\si_k) \cup F_n\big) \subseteq G \setminus \big(\Mv(\si_m) \cup F_m\big)
	\]
	and on the other
	\[
	G_m \subseteq \big(\Mv(\si_m) \cup \Fx(\si_m)\big) \cap \si_m^{-1}(F_m) \subseteq \Mv(\si_m) \cup F_m,
	\]
	so $G_n \cap G_m = \0$, finishing the proof.
\end{proof}

\subsection{Edge sliding over quotients}

Let $F$ be a smooth Borel equivalence relation on $X$ and let $s_F$ be a Borel selector for $G$. Any Borel edge-operator $\si'$ on $\Xmod{F}$ admits a lift to a Borel edge-operator $\si$ on $X$, namely:
\begin{equation}\label{natural_lift}
	\si'(x,y) \defeq
	\begin{cases}
		\si\big(s(x), s(y)\big)
		&
		\text{if } \big(s(x), s(y)\big) \in \Mv(\si)
		\\
		(x,y)
		&
		\text{otherwise}.
	\end{cases}
\end{equation}

Assuming that Borel selectors are fixed for all smooth Borel equivalence relations that appear below, we refer to the lift defined in \labelcref{natural_lift} as the \emph{natural lift}.

\begin{lemma}\label{lift_of_sliding}
	Let $G, H$ be Borel graphs and $F$ a smooth Borel equivalence relation on $X$. The natural lift $\si$ of any Borel $H$-based (resp. well-iterated) edge sliding $\si'$ of $\Gmod{F}$ is a Borel $(H \cup F)$-based (resp. well-iterated) edge sliding of $G$.
\end{lemma}
\begin{proof}
	For an edge sliding $\si'$, it suffices to observe that $\si$ fixes $F$ pointwise and, if $R \subseteq (\Xmod{F})^2$ is a railway for $\si'$, then $R \cup F$ is a railway for $\si$ because any $(x,y) \in F$ slides into $\big(s(x), s(y)\big)$ along $F$. The well-iterated case now follows by induction on the length of the iteration.
\end{proof}


\section{Strongly maximal Borel \fsrs}\label{sec:fsr}

Throughout this section, let $X$ be a standard Borel space and let $E$ be a countable Borel equivalence relation on $X$.

\subsection{Finite partial equivalence relations}\label{subsec:fsr}

A \emph{partial equivalence relation} $F$ on $X$ is an equivalence relation defined on a subset of $X$, which we refer to as the \emph{domain} of $F$ and denote by $\dom(F)$; thus, 
$$
\dom(F) = \set{x \in X : (x,x) \in F}.
$$
The term \emph{$F$-class} refers to a subset of $\dom(F)$ that is an $F$-class in the usual sense. Say that $F$ is \emph{entire} if $\dom(F) = X$.

We refer to $\cl{F} \defeq F \cup \Diag(X)$ as the \emph{completion} of $F$. Call a set $A \subseteq X$ \emph{$F$-invariant} if it is $\cl{F}$-invariant. 

A partial equivalence relation $F'$ is \emph{$F$-invariant} if every $F'$-class is $F$-invariant. In this case, $\dom(F')$ is $F$-invariant and the union $F' \cup F$ is still a partial equivalence relation that coincides with $F'$ on $\dom(F')$ and with $F$ on $X \setminus \dom(F')$.

Below, we will only be dealing with finite partial subequivalence relations of some ambient countable equivalence relation. Therefore, for convenience and for historical reasons, we refer to finite partial equivalence relations as \emph{\fsr} (stands for \emph{finite partial subequivalence relation}), which is the standard term used in \cite{Kechris-Miller} and earlier, by now classical, papers.

For a pairwise disjoint collection $\Psi \subseteq \FinX$, we let $E_\Psi$ denote the \emph{induced \fsr}, namely, the $E_\Psi$-classes are precisely the sets in $\Psi$. Conversely, for an \fsr $F$ and $Y \subseteq X$, let $\Classes[Y]{F}$ denote the set of $F$-classes contained in $Y$, where we omit the subscript $Y$ if $Y = X$. For $\Phi \subseteq \FinX$, say that $F$ is \emph{within $\Phi$} if $\Classes{F} \subseteq \Phi$.

\subsection{Maximal \fsrs}

For $\Phi \subseteq \FinX$, call an \fsr $F$ \emph{$\Phi$-maximal} (or \emph{maximal within $\Phi$}) if it is within $\Phi$ and there is no $U \in \Phi$ disjoint from $\dom(F)$.

In general, existence of a $\Phi$-maximal \fsr follows from Zorn's lemma, but when $\Phi \subseteq \FinE$, for a countable Borel equivalence relation $E$, a finer statement is true \cite{Kechris-Miller}*{Lemma 7.3}:

\begin{prop}[Kechris--Miller]\label{existence-of-maximal-fsr}
	For a countable Borel equivalence relation $E$ on $X$, any Borel $\Phi \subseteq \FinE$ admits a $\Phi$-maximal Borel \fsr $F \subseteq E$.
\end{prop}

The latter is mainly based on the following lemma, which, in turn, follows from the Feldman--Moore theorem.

\begin{lemma}[Kechris--Miller]\label{coloring_intersection_graph}
	For a countable Borel equivalence relation $E$ on $X$, the intersection graph on $\FinE$ admits a countable Borel coloring.
\end{lemma}
\begin{proof}
	See \cite{Kechris-Miller}*{Proof of Lemma 7.3}.
\end{proof}

Below we formulate and prove enhancements of \cref{existence-of-maximal-fsr} with stronger notions of maximality.

\subsection{Injective extensions and saturated \fsrs}\label{subsec:saturated-fsr}

\begin{defn}\
	\begin{itemize}
		\item We say that a set $U \in \FinX$ is \emph{injective over an \fsr $F$} if it is $F$-invariant and contains at most one $F$-class.
		
		\item For \fsrs $F_0,F_1$, say that $F_1$ \emph{injectively extends} $F_0$ if $F_1 \supseteq F_0$ and each $F_1$-class contains at most one $F_0$-class. 
		
		\item Call a sequence $(F_n)_{n \in \N}$ of \fsrs \emph{injectively increasing} if each $F_{n+1}$ injectively extends $F_n$.
	\end{itemize}
\end{defn}

\begin{lemma}\label{injectively_increasing_is_smooth}
	For any injectively increasing sequence $(F_n)_n$ of Borel \fsrs, $F \defeq \bigcup_n F_n$ is smooth.
\end{lemma}
\begin{proof}
	The smoothness of $F$ is witnessed by the map 
	$$
	x \mapsto [x]_{F_{n_x}} : \dom(F) \to \FinX,
	$$ 
	where $n_x$ is the smallest number such that $x \in \dom(F_{n_x})$.
\end{proof}

\begin{defn}
	For a collection $\Phi \subseteq \FinX$, call an \fsr $F$ \emph{$\Phi$-saturated} (or \emph{saturated within $\Phi$}) if it is within $\Phi$ and there is no $U \in \Phi \setminus \Classes{F}$ injective over $F$.
\end{defn}

\begin{cor}[Miller]\label{existence_of_saturated_fsr}
	For any countable Borel equivalence relation $E$ on $X$ and any Borel $\Phi \subseteq \FinE$, there is a Borel \fsr $F \subseteq E$ that is $\Phi$-saturated modulo $E$-compressible.
\end{cor}
\begin{proof}
	Follows from \cref{winning_SatG} below, where we have Player 1 play $\Phi$ all the time.
\end{proof}

We formulate a slightly stronger version of the last statement in the language of games.

\begin{defn}[Saturation game]\label{defn:SatG}
	Let $E$ be a countable Borel equivalence relation on $X$. The \emph{saturation game $\SatG(E)$} is as follows:
	
	\medskip
	$
	\Jue
	$
	\medskip
	
	\noindent where $\Phi_n \subseteq \FinE$ is Borel for each $n$ and the $F_n$ are injectively increasing Borel \fsrs with each $F_n$ being within $\bigcup_{k < n} \Phi_k$ (so $F_0 = \0$). We say that \emph{Player 1 wins modulo $E$-compressible} if, modulo $E$-compressible, $F \defeq \bigcup_{n \in \N} F_n$ is finite and saturated within $\Phi \defeq \bigcup_{n \in \N} \Phi_n$.
\end{defn}

\begin{theorem}\label{winning_SatG}
	For any countable Borel equivalence relation $E$, Player 1 has a strategy to win $\SatG(E)$ modulo $E$-compressible.
\end{theorem}
\begin{proof}
	By \cref{coloring_intersection_graph}, fix a countable coloring of the intersection graph on $\FinE$ and a sequence $(k_n)_n$ of natural numbers such that each $k \in \N$ appears infinitely many times.
	
	Having Player 1 play $F_0 \defeq \0$ in her $0^\text{th}$ move, and we describe the $n^\text{th}$ move of Player 1, for $n \ge 1$, assuming Player 2 has made his $(n-1)^\text{th}$ move. Let $\Psi_n$ be the collection of all sets in $\bigcup_{i < n} \Phi_i$ of color $k_n$ that are injective over $F_{n-1}$. Because $E_{\Psi_n}$ is $F_{n-1}$-invariant, $F_n \defeq F_{n-1} \cup E_{\Psi_n}$ is an \fsr that injectively extends $F_{n-1}$, so we have Player 2 play $F_n$.
	
	By \cref{injectively_increasing_is_smooth,compressible=supset_aperiodic+smooth}, $F \defeq \bigcup_{n \in \N} F_n$ is finite modulo $E$-compressible. Thus, throwing out this $E$-compressible set, we may assume that $F$ is finite and we show that it is saturated within $\Phi \defeq \bigcup_{n \in \N} \Phi_n$.
	
	Let $U \in \Phi$ be injective over $F$. We will show that $U \subseteq \dom(F)$. Let $N$ be large enough so that, for all $n \ge N$, $U \in \bigcup_{i < n} \Phi_i$ and $F \rest{U} = F_n \rest{U}$, so $U$ is injective over $F_n$. Letting $k$ be the color of $U$, there are arbitrarily large $n$ with $k_n = k$, so there must be $n > N$ for which $U$ is in $\Psi_n$, and hence is contained in $\dom(F_n)$.
\end{proof}

\begin{defn}
	Call a collection $\Phi \subseteq \FinX$ \emph{rich}, if for every $\Psi \subseteq \Phi$ with $\bigcup \Psi \ne X$, there are a nonempty $U \subseteq \big(\bigcup \Psi\big)^c$ and $V \in \Psi \cup \set{\0}$ such that $U \cup V \in \Phi$.
\end{defn}

\begin{lemma}\label{rich-saturated_is_entire}
	For any rich $\Phi \subseteq \FinX$, any $\Phi$-saturated \fsr is entire.
\end{lemma}
\begin{proof}
	Immediate from the definition of richness.
\end{proof}

As a quick application, we prove the following folklore lemma, which we will use below in the proof of \cref{ergodic_hyp_slid-subgraph}.

\begin{lemma}\label{connecting_with_compl_sect}
	Let $G$ be a Borel locally countable graph on a standard Borel space $X$. For any Borel $E_G$-complete set $S \subseteq X$, there is a Borel \fsr $F \subseteq E_G$ with $E_G$-cocompressible domain such that $G$ connects $F$ and each $F$-class intersects $S$.
\end{lemma}
\begin{proof}
	Let $\Phi$ be the collection of all $U \in \FinG$ that are $G$-connected and intersect $S$, and observe that $\Phi$ is rich. Applying \cref{existence_of_saturated_fsr}, we get a Borel \fsr $F \subseteq E_G$ that is saturated within $\Phi$ modulo an $E_G$-compressible set. Throwing this set out, \cref{rich-saturated_is_entire} implies that $F$ is entire.
\end{proof}

\subsection{Packed \fsrs}\label{subsec:packed-fsr}

Throughout this subsection, fix a Borel \emph{superadditive} function $p : \FinX \to [0, \w)$, i.e. 
$$
p(U \disjU V) \ge p(U) + p(V)
$$ 
for all disjoint $U,V \in \FinX$.

We will build a ``packed'' \fsr, where in the corresponding game, we will allow Player 2 to merge different classes of its previous move (an \fsr) as long as a certain ``control'' condition imposed by $p$ is satisfied.

\smallskip

For an \fsr $F$, let $F \rest{p}$ denote the restriction of $F$ to the $F$-invariant set 
$$
\set{x \in \dom(F) : p([x]_F) \ge 1}.
$$

\begin{defn}\label{defn:admissible}
	For an \fsr $F$, call $V \in \FinX$ \emph{$p$-admissible for $F$} if, firstly, it is $F$-invariant, and secondly, at least one of the following holds: 
	\begin{enumref}{a}{defn:admissible}
		\item\label{item:defn:admissible:none} $|\Classes[V]{F}| \le 1$;
		
		\item\label{item:defn:admissible:at_most_one} $p(V) > 0$ and $|\Classes[V]{F \rest{p}}| \le 1$;
		
		\item\label{item:defn:admissible:many} $p(V) > 0$ and 
		$
		\displaystyle 
		\sum_{U \in \Classes[V]{F \rest{p}}} \floor{p(U)} 
		\le
		|V \setminus \dom(F \rest{p})|$.
	\end{enumref}
\end{defn}

Here is an intuitive interpretation of the last condition: when trying to pack together some $F$-classes of positive $p$-weight, one needs to also include a padding of at least as many points as the sum of the $p$-weights of these $F$-classes. Each point in the padding is either a new (i.e. outside of $\dom(F)$) point or an old one contained in a $p$-null $F$-class.

\begin{defn}\label{defn:admissible_sequences}
	For \fsrs $F,F'$, call $F'$ a \emph{$p$-admissible extension of $F$} if $F' \supseteq F$ and each $F'$-class is $p$-admissible for $F$. Say that a sequence $(F_n)_n$ of \fsrs is $p$-admissibly increasing if each $\Psi_{n+1}$ is a $p$-admissible extension of $\Psi_n$.
\end{defn}

\begin{lemma}\label{admissible_sequence_stabilizes_mod-compressible}
	If $p : \FinX \to [0, \w)$ is a superadditive Borel function, then, for any $p$-admissibly increasing sequence $(F_n)_n$ of Borel \fsrs, the aperiodic part $Z$ of $F \defeq \bigcup_n F_n$ is $F$-compressible.
\end{lemma}
\begin{proof}
	$Z$ is covered by the following $F$-invariant sets:
	\begin{align*}
		Z_0 
		&\defeq 
		\bigcup \set{V \in \Classes[Z]{F} : \forall n \; \sum_{U \in \Classes[V]{F_n}} p(U) = 0}
		\\
		Z_1 
		&\defeq 
		\bigcup \set{V \in \Classes[Z]{F} : \exists n \; |\Classes[V]{F_n \rest{p}}| = 1}
		\\	
		Z_\Si
		&\defeq
		\bigcup \set{V \in \Classes[Z]{F} : \exists n \; 0 < \sum_{U \in \Classes[V]{F_n}} p(U) < \w}
		\\
		Z_2
		&\defeq
		\bigcup \set{V \in \Classes[Z]{F} : \forall^\w n \; |\Classes[V]{F_n \rest{p}}| \ge 2}.
	\end{align*}
	
	The fact that $(F_n \rest{Z_0})_n$ is $p$-admissibly increasing is solely witnessed by condition \labelcref{item:defn:admissible:none}. Therefore, $(F_n \rest{Z_0})_n$ is injectively increasing, so, by \cref{injectively_increasing_is_smooth}, $Z_0$ is $F$-smooth, hence $F$-compressible by \cref{compressible=supset_aperiodic+smooth}.
	
	Similarly, $(F_n \rest{Z_1})_n$ is eventually driven by condition \labelcref{item:defn:admissible:at_most_one}. Letting, for each $x \in Z_1$, $n_x \in \N$ be the least index such that $|\Classes[{[x]_F}]{F_n \rest{p}}| = 1$, we witness the smoothness of $F \rest{Z_1}$ by the map $x \mapsto$ the unique $U \in \Classes[V]{F_n \rest{p}}$ with $p(U) \ge 1$.
	
	Next, observe that for each $x \in Z_\Si$, there is $n \in \N$ such that the supremum of 
	$$
	\set{p(U) : U \in \Classes[{[x]_F}]{F_n}}
	$$ 
	is achieved, but only by finitely many $U \in \Classes[{[x]_F}]{F_n}$. Therefore, fixing a Borel linear order $<_B$ on $\FinX$, we witness the $F$-smoothness of $Z_\Si$ by the map
	$$
	x \mapsto U : Z_\Si \mapsto \FinX,
	$$
	where $U$ is the $<_B$-least in $\Classes[{[x]_F}]{F_n}$ that achieves the above supremum and $n \in \N$ is the least for which such $U$ exists.
	
	It remains to show that $Z_2$ is $F$-compressible.
	
	Let $\Psi$ be the collection of all inclusion-minimal sets in $\set{U \in \Classes[Z_2]{F \rest{p}}}$, so $\Psi$ is Borel and the sets in it are pairwise disjoint, so $F' \defeq E_\Psi$ is a Borel \fsr. Moreover, for each $F$-class $V \subseteq Z_2$ and for all large enough $n \in \N$, $\Classes[V]{F_n \rest{p}} \ne \0$, so $\dom(F')$ is an $F$-complete set for $Z_2$. Fix a Borel selector $s : \FinX \to X$ and put $Y \defeq s(\Psi)$, so $Y$ is still an $F$-complete set for $Z_2$. We will create infinitely many disjoint equidecomposable copies of $Y$, thus witnessing the $F$-compressibility of $Z_2$.
	
	For each $n \ge 1$, define $\ga_n : Y \to Z_2$ as follows. Set $\ga_n(y) \defeq y$ if there is no $V \in \Classes{F_n}$ with $[y]_{F'} \subseteq V$; otherwise, let $V$ denote the unique $F_n$-class containing $[y]_{F'}$. If $|\Classes[V]{F_{n-1} \rest{p}}| \le 1$, again put $\ga_n(y) \defeq y$. 
	
	Finally, suppose that $|\Classes[V]{F_{n-1} \rest{p}}| \ge 2$. Because $V$ is $p$-admissible for $F_{n-1}$, it must satisfy \labelcref{item:defn:admissible:many}, namely:
	\begin{equation}\label{eq:enough_to_fit}
		\sum_{U \in \Classes[V]{F_{n-1} \rest{p}}} \floor{p(U)} \le |V'|,
	\end{equation}
	where $V' \defeq V \setminus \dom(F_{n-1} \rest{p})$. Furthermore, because $|\Classes[V]{F_{n-1} \rest{p}}| \ge 2$, $V \notin \Psi$, so each $F'$-class $W \subseteq V$ is contained in some $U \in \Classes[V]{F_{n-1}}$. Thus,
	\begin{align*}
		|Y \cap V|
		= 
		|\Classes[V]{F'}| 
		&\le 
		\sum_{W \in \Classes[V]{F'}} \floor{p(W)}
		\\
		&= 
		\sum_{U \in \Classes[V]{F_{n-1}}} \sum_{W \in \Classes[U]{F'}} \floor{p(W)}
		\\
		&\le 
		\sum_{U \in \Classes[V]{F_{n-1}}} \left\lfloor\sum_{W \in \Classes[U]{F'}} p(W)\right\rfloor
		\\
		&\le
		\sum_{U \in \Classes[V]{F_{n-1}}} \floor{p(U)},
	\end{align*}
	where the last two inequalities are due to the superadditivity of the floor function and $p$, respectively. Combined with \labelcref{eq:enough_to_fit}, this gives
	$$
	|Y \cap V| \le |V'|.
	$$
	Assuming that $y$ is the $i^\text{th}$ least element of $Y \cap V$ (in some a priori fixed Borel linear order on $X$), we let $\ga_n(y)$ be the $i^\text{th}$ least element of $V'$.
	
	It is straightforward to check that 
	\begin{enumenv}
		\item \label{ga_n:1-1} for any $(n, y) \ne (m, y')$ in $\N \times Y$, $\ga_n(y) \ne \ga_m(y')$ unless $y=y'$ and $\ga_n(y) = y$; in particular, each $\ga_n$ is one-to-one;
		
		\item \label{ga_n:infinite_travel} for every $y \in Y$, there are infinitely many $n$ with $\ga_n(y) \ne y$.
	\end{enumenv}
	For each $n \ge 0$, inductively define $\tau_n : Y \to X$ by $y \mapsto \ga_m(y)$ where $m$ is the least in $\N$ such that $\ga_m(y) \notin \set{\tau_k(y)}_{k < n}$; such $m$ exist by \labelcref{ga_n:1-1,ga_n:infinite_travel}. Moreover, \labelcref{ga_n:1-1} implies that each $\tau_n$ is one-to-one, so the sets $\tau_n(Y)$ are pairwise disjoint by definition. This gives a compression
	$$
	\tau_0(Y) \to \tau_1(Y) \to \hdots \to \tau_n(Y) \to \hdots
	$$
	where the function $\tau_n(Y) \to \tau_{n+1}(Y)$ is defined by $x \mapsto \tau_{n+1} \circ \tau_n^{-1}(x)$.
\end{proof}

\begin{defn}
	For a collection $\Phi \subseteq \FinX$ and a superadditive function $p : \FinX \to [0,\w)$, call an \fsr $F$ \emph{$p$-packed within $\Phi$} if $F$ is within $F$ and there is no $V \in \Phi \setminus \Classes{F}$ that is $p$-admissible for $F$.
\end{defn}

\begin{obs}
	For any $\Phi \subseteq \FinX$ and any superadditive function $p : \FinX \to [0,\w)$, any \fsr that is $p$-packed within $\Phi$ is, in particular, saturated within $\Phi$.
\end{obs} 

\begin{cor}\label{existence-of-packed-fsr}
	For any countable Borel equivalence relation $E$ on $X$, superadditive Borel function $p : \FinE \to [0, \w)$, and Borel $\Phi \subseteq \FinE$, there is a Borel \fsr $F \subseteq E$ that is $p$-packed within $\Phi$ modulo $E$-compressible.
\end{cor}
\begin{proof}
	Follows from \cref{winning_PackG} below, where we have Player 1 play $\Phi$ all the time.
\end{proof}

As with saturation, we formulate a slightly stronger version of the last statement in the language of games.

\begin{defn}[Packing game]\label{defn:PackG}
	Let $E$ be a countable Borel equivalence relation on $X$ and $p : \FinX \to [0, \w)$ be a superadditive Borel function. The \emph{packing game $\PackG(E, p)$} is as follows:
	
	\medskip
	$
	\Jue
	$
	\medskip
	
	\noindent where $\Phi_n \subseteq \FinE$ is Borel for each $n$ and the $F_n$ are $p$-admissibly increasing Borel \fsrs with each $F_n$ being within $\bigcup_{k < n} \Phi_k$ (so $F_0 = \0$). We say that \emph{Player 1 wins modulo $E$-compressible} if, modulo $E$-compressible, $F \defeq \bigcup_{n \in \N} F_n$ is $p$-packed within $\Phi \defeq \bigcup_{n \in \N} \Phi_n$.
\end{defn}

\begin{theorem}\label{winning_PackG}
	For any countable Borel equivalence relation $E$ on $X$ and any superadditive Borel function $p : \FinX \to [0, \w)$, Player 1 has a strategy to win $\PackG(E, p)$ modulo $E$-compressible.
\end{theorem}
\begin{proof}
	The proof is verbatim that of \cref{winning_SatG}, replacing \qts{injective} with \qts{$p$-admissible} and \cref{injectively_increasing_is_smooth} with \cref{admissible_sequence_stabilizes_mod-compressible}.
\end{proof}


\section{Hyperfinite equivalence relations and graphs}\label{sec:hyperfinite_basics}

Let $X$ be a standard Borel space.

\subsection{Definitions}

Let $E$ be an equivalence relation on $X$. Call $E$ \emph{hyperfinite} if it is a countable union $\bigcup_n E_n$ of an increasing sequence $(E_n)_n$ of finite Borel equivalence relations. We will call such a sequence $(E_n)_n$ a \emph{witness to the hyperfiniteness} of $E$. In particular, hyperfinite equivalence relations are Borel\footnote{In a more general setting, hyperfinite equivalence relations are defined to be increasing unions of finite \emph{analytic} equivalence relations, so they may not be Borel. However, this does not concern us in this paper.}.

Call a Borel graph $G$ on a standard Borel space $X$ \emph{hyperfinite} if $E_G$ is hyperfinite. 

Now let $(X,\mu)$ be a standard measure space. A Borel equivalence relation $E$ on $X$ is called \emph{$\mu$-hyperfinite} if $E \rest{X'}$ is hyperfinite for some $\mu$-conull set $X'$. Similarly, call a Borel graph $G$ on $X$ \emph{$\mu$-hyperfinite} if $E_{G \rest{X'}}$ is $\mu$-hyperfinite for some $\mu$-conull set $X' \subseteq X$. Note that this is different from $E_G \rest{X'}$ being hyperfinite; however the two coincide when $\mu$ is $E_G$-\emph{quasi-invariant}, i.e. the $E_G$-saturation of every $\mu$-null set is still $\mu$-null.


\subsection{Basic facts}

\begin{lemma}\label{witness_to_hyperfin_starts_anywhere}
	For any hyperfinite Borel equivalence relation $E$ on a standard Borel space $X$ and any finite Borel subequivalence relation $F \subseteq E$, there is a witness $(F_n)_n$ to the hyperfiniteness of $E$ with $F_0 = F$.
\end{lemma}
\begin{proof}
	Let $s : X \to X$ be a Borel selector for $F$ and let $(E_n)_n$ be a witness to the hyperfiniteness of $E$. Put $Y \defeq s(X)$ and define equivalence relations $E_n'$ by $E_n' \rest{Y^c} \defeq \Id(Y^c)$ and $E_n' \rest{Y} \defeq E_n \rest{Y}$. Now put $F_n \defeq E_n' \vee F$. Then it is not hard to see that $(F_n)_n$ is as desired.
\end{proof}

We recall the following proposition from \cite{Kechris-Miller}*{Remark 6.10}.

\begin{prop}\label{bounded_hyperfiniteness_witness}
	Any hyperfinite Borel equivalence relation $E$ on $X$ admits a witness $(E_n)_n$ of bounded equivalence relations. In fact, given a bounded Borel subequivalence relation $F \subseteq E$, we can take $E_0 \defeq F$.
\end{prop}

Call a sequence $(G_n)_n$ of subgraphs of $G$ \emph{exhaustive} if $G = \bigcup_n G_n$.

\begin{prop}\label{graph-hyperfinite}
	A Borel graph $H$ on $X$ is hyperfinite if and only if it admits an increasing exhaustive sequence $(H_n)_n$ of component-finite bounded Borel subgraphs. In fact, $H_0$ can be taken to be any component-finite bounded Borel subgraph of $H$.
\end{prop}
\begin{proof}
	Using \cref{bounded_hyperfiniteness_witness}, let $(E_n)_n$ be a witness to the hyperfiniteness of $E_H$ such that each $E_n$ is bounded and $E_0 = E_{H_0}$. Take $H_n \defeq H \cap E_n$.
\end{proof}

\begin{lemma}\label{finite_spanning_subtrees}
	Every component-finite Borel graph $H$ on $X$ admits an acyclic Borel spanning subgraph $T$. In fact, we can ensure that $T \supseteq T_0$ for any given acyclic Borel subgraph $T_0 \subseteq H$.
\end{lemma}
\begin{proof}
	Clearly, each $H$-connected component $U$ admits a spanning subtree of $H \rest{U}$ extending $T_0 \rest{U}$. Because there are only finitely many such subtrees available for each $U$, we can choose one in a Borel fashion.
\end{proof}

\begin{prop}\label{hyperfinite_spanning_subtrees}
	Every hyperfinite Borel graph $H$ on $X$ admits an acyclic Borel spanning subgraph.
\end{prop}
\begin{proof}
	Write $H$ as an increasing union of component-finite Borel subgraphs $(H_n)_n$. Iterative applications of \cref{finite_spanning_subtrees} give an increasing sequence of acyclic Borel subgraphs $(T_n)_n$ such that each $T_n$ is a spanning subgraph of $H_n$. Therefore, $T \defeq \bigcup_n T_n$ is acyclic and it spans $H$.
\end{proof}

\subsection{Sliding hyperfinite trees into lines}

\begin{defn}\label{defn:line_forest}
	Let $X$ be a standard Borel space. We call a graph $L$ on $X$ a \emph{line forest} if it is acyclic and each vertex in it has degree at most $2$. Call a directing of such an $L$ \emph{proper} if each vertex has at most one incoming and at most one outgoing edge.
\end{defn}

\begin{lemma}\label{sliding_finite_trees_into_lines}
	Let $L \subseteq T$ be Borel component-finite acyclic graphs on $X$, where $L$ is a line forest. There is a Borel $L$-based well-iterated one-to-one edge sliding $\si$ of $T$ such that $L' \defeq \si(T)$ is a line forest. Moreover, given a proper Borel directing $\dir{L}$ of $L$, there is a proper Borel directing $\dir{L'}$ of $L'$ extending $\dir{L}$.
\end{lemma}
\begin{proof}
	Clearly, each $T$-connected component has such a well-iterated edge sliding, as well as a desired directing, and the component-finiteness of $T$ allows us to choose these in a Borel fashion uniformly for all $T$-connected components at once.
\end{proof}

\begin{prop}\label{sliding_hyperfinite_trees_into_lines}
	Every acyclic hyperfinite Borel graph $T$ on $X$ admits a Borel one-to-one well-iterated edge slide $L$ that is a line forest and has a proper Borel directing $\dir{L}$.
\end{prop}
\begin{proof}
	This follows by recursive applications of \cref{sliding_finite_trees_into_lines}, put together by \cref{connecting_unions_using_images}.
\end{proof}

\begin{cor}\label{hyperfinite_subgraphs_admit_sublines}
	Every hyperfinite Borel graph $H$ on $X$ admits a one-to-one Borel well-iterated edge slide $\tH$ containing a Borel spanning line subforest $L$ with a proper Borel directing $\dir{L}$.
\end{cor}
\begin{proof}
	Apply \cref{hyperfinite_spanning_subtrees} in tandem with \cref{sliding_hyperfinite_trees_into_lines}.
\end{proof}


\section{Aperiodic and ergodic hyperfinite decompositions}\label{sec:decomps}

Let $X$ be a standard Borel space and $E$ a countable Borel equivalence relation. 

Define the natural extension $\extE$ of $E$ to $\XE$ as the pullback of $E$ under the projection map $\proj : \XE \to X$. For $A \subseteq \XE$, put
\begin{align*}
	\Vrts(A) &\defeq A \cap X 
	\\
	\Edges(A) &\defeq A \cap E.
\end{align*}

\subsection{Constructing an aperiodic hyperfinite factor}\label{subsec:aperiodic-hyperfinite_factor}

\begin{lemma}[Nontrivial finite factor]\label{finite_factor}
	Let $E_0 \subseteq E$ be countable Borel equivalence relations on $X$. For any Borel graphing $G$ of $E$, there is a finite Borel subequivalence relation $F \subseteq E$ transverse to $E_0$ such that
	\begin{enumref}{i}{finite_factor}
		\item \emph{(built from $G$)} $F$ is graphed by a Borel $E_0$-based well-iterated edge sliding of $G$ into $F$,
		
		\item \emph{(exhaustive or nontrivial)} modulo $E$-compressible, $X$ admits a partition $X_0 \disjU X'$ into $E$-invariant Borel sets such that $E \rest{X_0} = E_0 \rest{X_0} \transvee F \rest{X_0}$ and each $F$-class in $X'$ has more than one element.
	\end{enumref}
\end{lemma}
\begin{proof}	
	We will play the saturation game $\SatG(\extE)$ on $\XE$ (see \cref{defn:SatG}) as follows: assuming that Player 2 has last played an \fsr $\extF \subseteq \extE$, we put $F \defeq \extF \cap X^2$ and define the response $\Phi(\extF)$ of Player 1 as the collection of all $U \in \Finw{\XG}_{\extE}$ satisfying the following conditions:
	\begin{enumenv}
		\item \label{item:finite_factor:pf:grows} $|\Vrts(U)| \ge 1$,
		
		\item \label{item:finite_factor:pf:edge-disjoint}  the graph $G_U \defeq \Edges(U) \setminus \dom(\extF)$ is disjoint from (transverse to) $E_0 \vee F$,
		
		\item \label{item:finite_factor:pf:minimal} $|G_U| = |\Classes[U]{F}| - 1$,
		
		\item \label{item:finite_factor:pf:connects} there is an edge sliding $\si$ along $E_0 \vee F$ that moves $G_U$ into $\Vrts(U)^2$ so that $\si(G_U) \cup F$ connects $\Vrts(U)$.	
	\end{enumenv}
	
	Let Player 1 play $\SatG(\extE)$ according to the strategy provided by \cref{winning_SatG}, and, assuming Player 1 has played $\extF_n$ in her $n^\text{th}$ move, we let Player 2 play $\Phi_n \defeq \Phi(\extF_n)$. Thus, the $\extF_n$ are injectively increasing and, modulo $\extE$-compressible, $\extF_\w \defeq \bigcup_{n \in \N} \extF_n$ is finite and saturated within $\Phi_\w \defeq \bigcup_{n \in \N} \Phi_n$.
	
	For each $n \le \w$, put $\extD_n \defeq \dom(\extF_n)$, $F_n \defeq \extF_n \cap X^2$, and $E_n \defeq E_0 \vee F_n$.
	
	\begin{claim+}
		$F_{n+1} \trans_{F_n} E_n$.
	\end{claim+}
	\begin{pf}
		Follows from conditions \labelcref{item:finite_factor:pf:edge-disjoint,item:finite_factor:pf:connects} together with \labelcref{item:sliding_preserves:transversality}.
	\end{pf}
	
	\begin{claim+}
		$Z \defeq \ext{Z} \cap X$ is $E$-compressible, where $\ext{Z} \subseteq \XE$ is the aperiodic part of $\extF_\w$.
	\end{claim+}
	\begin{pf}
		Noting that $(F_n)_{n \in \N}$ is clearly injectively increasing, it is enough to show that $Z$ is $F_\w$-aperiodic because then \cref{injectively_increasing_is_smooth,compressible=supset_aperiodic+smooth} imply that it is $E$-compressible. But the $F_\w$-aperiodicity of $Z$ easily follows from \labelcref{item:finite_factor:pf:minimal} because each $U \in \Classes{F_n}$ contains more vertices than edges; in fact, it is easy to show by induction on $n$ that $|\Edges(U)| = |U| - 1$.
	\end{pf}
	
	Hence, throwing out $[Z]_E$ from $X$, and hence also $\ext{Z}$ from $\XE$, we may assume that $\extF_\w$ is finite and is saturated within $\Phi_\w$.
	
	\medskip
	
	\begin{claim+}\label{claim:taken_edges_in_Ew}
		$\Edges(\extD_\w) \subseteq E_\w \cap G$.
	\end{claim+}
	\begin{pf}
		Due to \labelcref{item:finite_factor:pf:connects}, any edge $e \defeq \edge{x,y} \in \extD_\w$ slides along $E_n$, for some $n \in \N$, into an edge, whose endpoints are $E_{n+1}$-equivalent, so \labelcref{item:sliding_preserves:connectivity} implies that $e \in E_{n+1}$.
	\end{pf}
	
	Put $X' \defeq \big(\Vrts(\extD_\w)\big)_E$ and $X_0 \defeq X \setminus X'$.
	
	\begin{claim+}\label{claim:class_in_domain_unless_connects}
		$E \rest{X_0} = E_\w \rest{X_0}$.
	\end{claim+}
	\begin{pf}
		Let $C \subseteq X_0$ be an $E$-class, so there is $x \in C \setminus \extD_\w$, and suppose towards a contradiction that $C$ is not one $E_\w$-class. Because $G$ graphs $E$, there is an edge $e \defeq \gen{x', y} \in G$ such that $[x]_{E_\w} = [x']_{E_\w} \ne [y]_{E_\w}$, so, by \cref{claim:taken_edges_in_Ew}, $e \notin \extD_\w$. Thus, $U \defeq \set{x} \cup \set{e} \cup [y]_{\extF_\w}$ is injective over $\extF_\w$. Let $n$ be large enough so that $[y]_{\extF_n} = [y]_{\extF_\w}$ and $x' \in [x]_{E_n}$. It is easy to check that $U \in \Phi(\extF_n) = \Phi_{n+1}$, contradicting the saturation of $\extF_\w$ within $\Phi_\w$.
	\end{pf}
	
	Now we perform the conceived well-iterated edge sliding. Recalling that $F_0 = \0$, fix $n \ge 0$. Put $G_n \defeq \Edges(\extD_{n+1}) \setminus \Edges(\extD_n)$. For each $U \in \Classes{\extF_{n+1}}$, let $\si_U : U \cap G_n \to [X]^2$ denote the restriction of an edge sliding along $E_n$ that witnesses condition \labelcref{item:finite_factor:pf:connects} for $U$; because there are only finitely many choices for $\si_U$, we can ensure that the map $U \mapsto \si_U$ is Borel. Thus,
	\[
	\si_n \defeq \bigdisjU_{U \in \Classes{\extF_{n+1}}} \si_U \disjU \id \rest{G_n^c}
	\]
	defines a Borel $E_n$-based edge sliding of $G_n$ into $F_{n+1}$ that graphs $F_{n+1}$ over $F_n$; in fact, by \cref{replace_base_any_spanning_subgraph}, $\si_n$ is an $(E_0 \cup F_n)$-based edge sliding of $G_n$. Furthermore, by the definition of $G_n$ and \cref{claim:taken_edges_in_Ew}, $E_n \trans G_n \subseteq E_{n+1}$, so the $G_n$ are disjoint. Thus, \cref{connecting_unions_with_disjoint} applies, so the composition $\si$ of the $\si_n$ is a Borel $E_0$-based well-iterated edge sliding of $\bigdisjU_{n \in \N} G_n$ into $F_\w$ that graphs $F \defeq F_\w \cup \Id_X$.
\end{proof}

\begin{lemma}[Aperiodic hyperfinite factor]\label{hyperfinite_factor}
	Let $E_0 \subseteq E$ be countable Borel equivalence relations on $X$. For any Borel graphing $G$ of $E$, there is a hyperfinite Borel subequivalence relation $F \subseteq E$ transverse to $E_0$ such that
	\begin{enumref}{i}{hyperfinite_factor}
		\item \label{item:hyperfinite_factor:edge_sliding} \emph{(built from $G$)} $F$ is graphed by a Borel $E_0$-based well-iterated edge sliding of $G$ into $F$,
		
		\item \label{item:hyperfinite_factor:exhaust-or-aper} \emph{(exhaustive or aperiodic)} modulo $E$-compressible, $X$ admits a partition $X_0 \disjU X'$ into $E$-invariant Borel sets such that $F \rest{X_0}$ is finite and $E \rest{X_0} = E_0 \rest{X_0} \transvee F \rest{X_0}$, and $F \rest{X'}$ is aperiodic.
	\end{enumref}
\end{lemma}
\begin{proof}
	Starting with $F_0 = \Id_X$, recursive applications of \cref{finite_factor} produce an increasing sequence $(F_n)_{n \in \N}$ of finite Borel equivalence relations and a sequence $(\si_n)_{n \in \N}$ of Borel edge-operators on $X$, where $F_{n+1} \trans_{F_n} E_0$ and $\si_n$ is an $E_0$-based well-iterated edge sliding of $\bsi_n(G)$ into $F_{n+1}$ that graphs $F_{n+1}$ over $F_n$. More precisely, supposing that $(F_k)_{k < n}$  and $\bsi_n$ are defined, apply \cref{finite_factor} to the (concrete) quotients $\Xmod{F_n}$, $\Emod{F_n}$, $\Gmod[R]{F_n}$, and $\Gmod[\bsi_n(G)]{F_n}$ to get $F_{n+1}'$ and $\si_n'$, respectively; take $F_{n+1} \defeq F_{n+1}' \vee F_n$ and let $\si_n$ be the natural lift of $\si_n'$. By \cref{connecting_unions_using_images}, the composition $\si$ of $(\si_n)_{n \in \N}$ satisfies \labelcref{item:hyperfinite_factor:edge_sliding}.
\end{proof}

\subsection{Aperiodic hyperfinite decomposition}\label{subsec:aperiodic-hyperfinite_decomposition}

\begin{theorem}[Aperiodic hyperfinite decomposition for equivalence relations]\label{decomp:aper:eq_rel}
	For any countable Borel equivalence relations $E_0 \subseteq E$ on a standard Borel space $X$, any Borel graphing $G$ of $E$ admits a Borel $E_0$-based well-iterated edge slide $\tG$, for which there is
	
	\begin{enumref}{i}{decomp:aper:eq_rel}
		\item an increasingly transverse decomposition $E = \bigtransvee_{n \in \N} E_n$, where, for each $n \in \N^+$, $E_n$ is hyperfinite and is supergraphed by $\tG$,
		
		\item a partition $X = \disjU_{N \in \N \cup \set{\w}} X_N$ into $E$-invariant (possibly empty) Borel sets, where, $X_0$ is compressible, and, for each $N \in \N^+ \cup \set{\w}$ and $n \in \N^+$,
		
		\medskip
		\begin{tabular}{rcl}
			$n < N$ &$\imp$& $E_n \rest{X_N}$ is aperiodic,
			\\
			$n = N$ &$\imp$& $E_n \rest{X_N}$ is finite,
			\\
			$n > N$ &$\imp$& $E_n \rest{X_N} = \Id_{X_N}$.
		\end{tabular}
		
		\medskip
		
		\noindent In particular, $E \rest{X_N} = \bigtransvee_{n = 0}^N E_n \rest{X_N}$ for each $N \in \N^+ \cup \set{\w}$.
	\end{enumref}
	
	\noindent Moreover, given any Borel edge-coloring $(G_k)_{k = 1}^\w$ of $G$, we can ensure that $G_n \subseteq \bigtransvee_{k \le n} E_k$.
\end{theorem}
\begin{proof}
	This is easily obtained by recursive applications of \cref{hyperfinite_factor} applied to the (concrete) quotients by $E_{G_n}$, with $\bigtransvee_{k < n} E_k$ in place of $E_0$ for each $n \ge 1$, and put together by \cref{connecting_unions_using_images}. The fact that $G_n \subseteq \bigtransvee_{k \le n} E_k$ ensures that $\bigtransvee_{n < \w} E_n$ exhausts $E$.
\end{proof}

We also state a version of the last theorem for graphs.

\begin{cor}[Aperiodic hyperfinite decomposition for graphs]\label{decomp:aper:graphs}
	For any locally countable Borel graphs $G_0 \subseteq G$ on a standard Borel space $X$, $G$ admits a Borel $G_0$-based well-iterated edge slide $\tG$, for which, modulo $E_G$-compressible, there is
	\begin{enumref}{i}{decomp:aper:graphs}
		\item an increasingly transverse spanning partition $\bigtransvee_{n \in \N} G_n$ of $\tG$, where, for each $n \in \N^+$, $G_n$ is a hyperfinite Borel graph,
		
		\item a partition $X = \bigdisjU_{N \in \N^+ \cup \set{\w}} X_N$ into $E_G$-invariant (possibly empty) Borel sets, where, for each $N \in \N^+ \cup \set{\w}$ and $n \in \N^+$,
		
		\medskip
		\begin{tabular}{rcl}
			$n < N$ &$\imp$& $G_n \rest{X_N}$ is component-infinite,
			\\
			$n = N$ &$\imp$& $G_n \rest{X_N}$ is component-finite,
			\\
			$n > N$ &$\imp$& $G_n \rest{X_N} = \0$.
		\end{tabular}
	\end{enumref}
	
	\noindent In particular, if $G$ is acyclic, then $\tG$ is also acyclic, and hence, $\tG = \bigdisjU_{n < \om} G_n$. Moreover, each $G_n$, for $n \in \N^+$, is actually a line forest that admits a proper Borel directing. 
\end{cor}
\begin{proof}
	The main part is just a rephrasing of \cref{decomp:aper:eq_rel} applied to $E_0 \defeq E_{G_0}$ and $G$: indeed, take $G_n \defeq \tG \cap E_n$. The part about spanning line forests is provided by \cref{hyperfinite_subgraphs_admit_sublines}.
\end{proof}

An immediate corollary of this is Hjorth's lemma for cost attained (see \cite{Hjorth:cost_lemma} or \cite{Kechris-Miller}*{Theorems 28.2 and 28.3}), which we restate below. We denote by $[E]$ (resp. $\llbracket E \rrbracket$) the group (resp. groupoid) of (resp. partial) Borel automorphisms of $E$.

\begin{defn}
	Let $E$ be a Borel equivalence relation on a standard Borel space $X$ and let $N \in \N \cup \set{\w}$. A set $\set{\ga_n}_{n = 0}^N \subseteq \llbracket E \rrbracket$ is said to be \emph{free}, if, for every nonempty word $w_0 w_1 \hdots w_l \in \Fin[N]$, $\ga_{w_0} \ga_{w_1} \hdots \ga_{w_l} x \ne x$ for every $x \in \dom(\ga_{w_0} \ga_{w_1} \hdots \ga_{w_l})$. Furthermore, $\set{\ga_n}_{n = 0}^N$ is said to be \emph{$E$-generating} if $E = \bigvee_{n=0}^N E_{\ga_n}$.
\end{defn}

\begin{cor}[Hjorth]\label{Gregs_lemma}
	Let $E$ be a countable Borel ergodic measure-preserving equivalence relation on a standard probability space $(X,\mu)$ and let $N \in \N \cup \set{\w}, 0 \le \de < 1$. The following are equivalent:
	\begin{enumerate}[(1)]
		\item \label{item:Gregs_lemma:treeable} $E$ is treeable a.e. and $c_\mu(E) = N + \de$.
		
		\item \label{item:Gregs_lemma:free_dec} $E$ admits a free decomposition \cite{Kechris-Miller}*{Section 27}
		$
		E = \bigfreeprod_{n = 0}^N E_n \text{ a.e.},
		$
		where, for each $1 \le n < N+1$, $E_n$ is aperiodic hyperfinite, and $E_0$ is finite with $c_\mu(E_0) = \de$ if $N < \w$ and $c_\mu(E_0) = 0$ if $N = \w$.
		
		\item \label{item:Gregs_lemma:free_group} There is an $E$-generating a.e. free system $\set{\ga_n}_{n = 0}^N$, where, for each $1 \le n < N+1$, $\ga_n \in [E]$ is aperiodic, and $\ga_0 \in \llbracket E \rrbracket$ with $\mu(\dom(\ga)) = \de$ if $N < \w$ and $\mu(\dom(\ga)) = 0$ if $N = \w$.
	\end{enumerate}
\end{cor}
\begin{proof}
	\labelcref{item:Gregs_lemma:treeable}$\imp$\labelcref{item:Gregs_lemma:free_dec} follows by applying \cref{decomp:aper:graphs} to an a.e. treeing $G$ of $E$ and observing that $X_0$ is null because $E$ is p.m.p., and $X_n$ is non-null for exactly one $n \in \N \cup \set{\w}$ because $E$ is ergodic. \labelcref{item:Gregs_lemma:free_dec}$\imp$\labelcref{item:Gregs_lemma:free_group} is by applying \cref{hyperfinite_subgraphs_admit_sublines} to each $E_n$, and \labelcref{item:Gregs_lemma:free_group}$\imp$\labelcref{item:Gregs_lemma:treeable} follows by taking $E_n \defeq E_{\ga_n}$.
\end{proof}

\begin{remark}
	\cref{decomp:aper:graphs} also easily implies a related theorem of Hjorth--Kechris \cite{Kechris-Miller}*{Lemma 27.7}, whose original proof, however, is perhaps shorter.
\end{remark}

\subsection{Ergodic hyperfinite decomposition}

Here we strengthen \cref{decomp:aper:eq_rel} by getting every factor to be ergodic provided the ambient equivalence relation itself is ergodic. 

\begin{theorem}[Ergodic hyperfinite decomposition over an ergodic base]\label{decomp:erg_over_base}
	For any countable Borel measure-preserving ergodic equivalence relations $E_0 \subseteq E$ on a standard probability space $(X,\mu)$, any Borel graphing $G$ of $E$ that also supergraphs $E_0$ admits a Borel $E_0$-based well-iterated edge slide $\tG$ for which there is $N \in \N \cup \set{\w}$ with $N \le C_\mu(G)$ and an increasingly transverse decomposition 
	\[
	E = \bigtransvee_{n = 0}^N E_n \text{ a.e.}
	\] 
	into Borel equivalence relations supergraphed by $\tG$, where, for each $1 \le n < N$, $E_n$ is ergodic hyperfinite and, if $N < \w$, $E_N$ is finite.
\end{theorem}
\begin{proof}
	Identical to that of \cref{decomp:aper:eq_rel} once \cref{hyperfinite_factor} is replaced with \cref{hyperfinite_factor:ergodic} (proved below). The bound on $N$ is due to \labelcref{item:iterated_edge_sliding_preserves:cost} and the fact that aperiodic hyperfinite equivalence relations have cost $\mu(X) = 1$ by Levitt's lemma (see \cite{Levitt} or \cite{Kechris-Miller}*{Lemma 22.1}).
\end{proof}

\cref{decomp:erg_over_base} combined with \cref{ergodic_hyp_slid-subgraph} gives the main result of the current paper.

\begin{theorem}[Ergodic hyperfinite decomposition]\label{decomp:erg}
	For any countable Borel measure-preserving ergodic equivalence relation $E$ on a standard probability space $(X,\mu)$, any Borel graphing $G$ of $E$ admits a Borel well-iterated edge slide $\tG$ for which there is $N \in \N \cup \set{\w}$ with $N \le C_\mu(G)$ and an increasingly transverse decomposition 
	\[
	E = \bigtransvee_{n = 0}^N E_n \text{ a.e.}
	\] 
	into Borel equivalence relations supergraphed by $\tG$, where, for each $0 \le n < N$, $E_n$ is ergodic hyperfinite and, if $N < \w$, $E_N$ is finite.
\end{theorem}

As before, because a Borel well-iterated edge slide of a treeing is again a treeing by \labelcref{item:iterated_edge_sliding_preserves:acyclicity,item:well-iterated_properties:connectivity_preserving}, we obtain a significant strengthening of Hjorth's lemma (\cref{Gregs_lemma}), where each occurrence of ``aperiodic'' is replaced with ``ergodic''.

\begin{cor}[Action with ergodic generators]\label{ergodic_Gregs_lemma}
	Let $E$ be a countable Borel ergodic measure-preserving equivalence relation on a standard probability space $(X,\mu)$ and let $N \in \N \cup \set{\w}, 0 \le \de < 1$. The following are equivalent:
	\begin{enumerate}[(1)]
		\item \label{item:ergodic_Gregs_lemma:treeable} $E$ is treeable a.e. and $c_\mu(E) = N + \de$.
		
		\item \label{item:ergodic_Gregs_lemma:free_dec} $E$ admits a free decomposition \cite{Kechris-Miller}*{Section 27}
		$
		E = \bigfreeprod_{n = 0}^N E_n \text{ a.e.},
		$
		where, for each $1 \le n < N+1$, $E_n$ is ergodic hyperfinite, and $E_0$ is finite with $c_\mu(E_0) = \de$ if $N < \w$ and $c_\mu(E_0) = 0$ if $N = \w$.
		
		\item \label{item:ergodic_Gregs_lemma:free_group} There is an $E$-generating a.e. free system $\set{\ga_n}_{n = 0}^N$, where, for each $1 \le n < N+1$, $\ga_n \in [E]$ is ergodic, and $\ga_0 \in \llbracket E \rrbracket$ with $\mu(\dom(\ga)) = \de$ if $N < \w$ and $\mu(\dom(\ga)) = 0$ if $N = \w$.
	\end{enumerate}
\end{cor}

\subsection{Constructing an ergodic hyperfinite factor}

We now prove an analogue of \cref{finite_factor} for building ergodic factors and, for the rest of this subsection, we let $(X, \mu)$ denote a standard measure space with finite nonzero $\mu$.

\begin{notation}
	For a Borel function $f : X \to \R$, the following are the \emph{essential infimum}, \emph{supremum}, and \emph{oscillation} of $f$ with respect to $\mu$:
	\begin{align*}
		\infmu(f) &\defeq \sup \set{r \in \R : \mu(\set{x \in X : f(x) < r}) = 0}
		\\
		\supmu(f) &\defeq \inf \set{r \in \R : \mu(\set{x \in X : f(x) > r}) = 0}
		\\
		\oscmu(f) &\defeq \supmu(f) - \infmu(f),
	\end{align*}
	where we use the convention that $\sup \0 \defeq -\w$ and $\inf \0 \defeq +\w$.
\end{notation}

\begin{notation}\label{notation:rho-size_mean-of-f}
	Let $f : X \to \R, w : X \to [0, \w)$ be Borel functions. For a set $U \in \FinX$, put
	\[
	\meanwf{U}
	\defeq 
	\begin{cases}
	\frac{\sum_{x \in U} f(x) w(x)}{|U|_w} & \text{if $|U|_w \ne 0$}
	\\
	0 & \text{otherwise},
	\end{cases}
	\]
	Furthermore, for a finite Borel equivalence relation on $X$, define $\meanwf{F} : X \to \R$ by $\meanwf{F}[x] \defeq \meanwf{[x]_F}$. We omit $w$ from the superscripts if $w \equiv 1$.
\end{notation}

Below, let $E$ denote a measure-preserving countable Borel equivalence relation on $(X,\mu)$, $G$ a Borel graphing of $E$, $E_0 \subseteq E$ a Borel ergodic subequivalence relation.

\begin{lemma}[Approximately ergodic finite factor]\label{finite_factor:averaging}
	For any $\N^+$-valued Borel (weight) function $w$ on $X$ and any $\Q$-valued bounded Borel function $f$ on $X$, there is a finite Borel subequivalence relation $F \subseteq E$ transverse to $E_0$ such that:
	
	\begin{enumref}{i}{finite_factor:averaging}
		\item \label{item:finite_factor:averaging:edge_sliding} \emph{(built from $G$)} $F$ is graphed by a Borel $E_0$-based well-iterated edge sliding of $G$ into $F$,
		
		\item \label{item:finite_factor:averaging:Cauchy-or-fin_index} \emph{(exhaustive or decreases oscillation)} either $E = E_0 \transvee F$ modulo $\mu$-null, or 
		\[
		\oscmu\big(\meanwf{F}\big) \le \frac{2}{3} \oscmu(f).
		\]
	\end{enumref}
\end{lemma}
\begin{proof}
	Let $I$ be the middle third closed subinterval of $[\infmu(f), \supmu(f)]$. Call a set $V \in \FinX$ \emph{negative} (resp. \emph{central}, \emph{positive}) if $\meanwf{V}$ is in $I_-$ (resp. $I$, $I_+$); moreover, say that it is \emph{of type $(q,w)$} if $\meanwf{V} = q$ and $|V|_w = w$. For a set $Y \subseteq X$, we say that \emph{points of $Y$ have the same sign} if a.e. point $x \in Y$ is non-negative or a.e. point $x \in Y$ is non-positive.
	
	Recalling that $\extE$ is the natural lift of $E$ to $X \disjU E$, let $\Phi$ be the collection of all $U \in \Finw{X \disjU (G \setminus E_0)}_{\extE}$ such that
	\begin{enumenv}
		\item \label{item:finite_factor:averaging:pf:central} $\Vrts(U)$ is central,
		
		\item $|\Edges(U)| = |\Vrts(U)| - 1$,
		
		\item \label{item:finite_factor:averaging:pf:connects} there is an edge sliding $\si$ along $E_0$ that moves $\Edges(U)$ into $\Vrts(U)^2$ so that $\si\big(\Edges(U)\big)$ connects $\Vrts(U)$.
	\end{enumenv}
	Let $\extF_0 \subseteq \extE$ be a $\Phi$-maximal Borel \fsr given by \cref{existence-of-maximal-fsr} and put $F_0 \defeq \extF_0 \cap X^2$. For each $U \in \Classes{\extF_0}$, let $\si_U : \Edges(U) \to [X]^2$ denote the restriction of an edge sliding along $E_0$ that witnesses condition \labelcref{item:finite_factor:averaging:pf:connects} for $U$. The finiteness of the choices of $s_U$ guarantees that a uniformly Borel choice is possible, so putting these $s_U$ together defines a Borel $E_0$-based edge sliding of $G$ that graphs $F_0$.
	
	If, modulo $\mu$-null, $G' \defeq G \setminus \big(E_0 \cup \dom(\extF_0)\big) = \0$, then $E_0 \transvee F_0 = E$ and we are done, so suppose otherwise. The $E_0$-classes incident to some, but only finitely many, edges from $G'$ form an $E_0$-smooth Borel set, which is hence $\mu$-null, so we may assume that each $E_0$-class is incident to either none or infinitely many edges from $G'$. By ergodicity and our assumption, the latter case must hold for a.e. $E_0$-class.
	
	If points of $Y \defeq X \setminus \dom(F_0)$ have the same sign modulo $\mu$-null, then clearly $\oscmu\big(\meanwf{F_0}\big) \le \frac{2}{3} |I|$ a.e. and we are done, so suppose otherwise.
	
	\begin{claim+}\label{claim:inf-many_types}
		There are weights $w_-, w_+ \in \N^+$ and rational numbers $q_- \in I_-, q_+ \in I_+$, such that a.e. $E_0$-class contains infinitely many points from $Y$ of each of the types $(q_-, w_-)$ and $(q_+, w_+)$.
	\end{claim+}
	\begin{pf}
		By ergodicity, the union of $E_0$-classes that only have positive (resp. negative) points lying in $Y$ is either null or conull, so it must be null by the assumption on $Y$. Next, the $E_0$-classes that contain some but only finitely many positive (resp. negative) points lying in $Y$ form an $E_0$-smooth Borel set, which is therefore null as well. Thus, we may assume that each $E_0$-class contains infinitely many points of $Y$ of each sign: positive and negative. Moreover, for a point $x \in Y$, there are only countably many possible values for $f(x)$ and $w(x)$, so the ergodicity of $E_0$ (this is the only place where we really use ergodicity) yields $w_-, w_+$ and $q_-, q_+$ as desired.
	\end{pf}
	
	Let $E'$ be the equivalence relation induced by $E_0 \cup G'$.
	\begin{claim+}
		$[E' : E_0] < \w$.
	\end{claim+}
	\begin{pf}
		Otherwise, we contradict the $\Phi$-maximality of $\extF_0$ as follows. Fix $k_-, k_+ \in \N^+$ such that 
		\[
		\frac{k_- w_-}{k_- w_- + k_+ w_+} q_- + \frac{k_+ w_+}{k_- w_- + k_+ w_+} q_+ \in I_0
		\]
		and let $C_1^-, \hdots, C_{k_-}^-, C_1^+, \hdots, C_{k_+}^+$ be distinct $E_0$-classes, whose union is $(E_0 \cup G')$-connected. Taking one point of type $(q_-, w_-)$ from each set $C_i^- \cap Y$ and one point of type $(q_+, w_+)$ from each set $C_j^+ \cap Y$, it is clear how to form a set $U \in \Finw{Y \disjU G'}$ with $U \in \Phi$.
	\end{pf}
	
	If $G$ were acyclic and supergraphed $E_0$, we would be done because the assumption that a.e. $E_0$-class is incident to infinitely many $G'$-edges contradicts the last claim, so it must be that the points of $Y$ have the same sign.
	
	For general $G$ though, we need the following extra step. Using \cref{edge_sliding_into_complete-sect}, we may assume without loss of generality that $G' \subseteq Y^2$, so we can apply \cref{hyperfinite_factor} to $E_0 \rest{Y}$, $E \rest{Y}$, and $G'$, and get a hyperfinite Borel subequivalence relation $F' \subseteq E \rest{Y}$ on $Y$ satisfying the conclusion of \cref{hyperfinite_factor}. By the last claim, this $F'$ must be finite, so $E \rest{Y} = E_0 \rest{Y} \transvee F'$ by \labelcref{item:hyperfinite_factor:exhaust-or-aper} modulo $E$-compressible, and hence modulo $\mu$-null. Therefore, $F \defeq F_0 \cup F'$ is a finite equivalence relation satisfying $E = E_0 \transvee F$. $F$ also satisfies \labelcref{item:finite_factor:averaging:edge_sliding} by composing the three well-iterated edge slidings mentioned above as it is clear that they form an $E_0$-based $G$-conservative sequence.
\end{proof}

\begin{lemma}[Ergodic hyperfinite factor]\label{hyperfinite_factor:ergodic}
	There is a hyperfinite Borel subequivalence relation $F \subseteq E$ transverse to $E_0$ such that 
	\begin{enumref}{i}{hyperfinite_factor:ergodic}
		\item \label{item:hyperfinite_factor:ergodic:edge_sliding} \emph{(built from $G$)} $F$ is graphed by a Borel $E_0$-based well-iterated edge sliding of $G$ into $F$,
		
		\item \label{item:hyperfinite_factor:ergodic:F_ergodic_if_possible} \emph{(exhaustive or ergodic)} either $F$ is finite and $E = E_0 \transvee F$ modulo $\mu$-null, or $F$ is ergodic.
	\end{enumref}
\end{lemma}
\begin{proof}
	Let $\DC$ be a countable dense family in $L^1(X,\mu)$ of $\Q$-valued bounded Borel functions and let $(f_n)_{n \in \N}$ be an enumeration of $\DC$, where each $f \in \DC$ appears infinitely many times.
	
	Starting with $F_0 = \Id_X$ and taking $w \equiv 1$, recursive applications of \cref{finite_factor:averaging} produce an increasing sequence $(F_n)_{n \in \N}$ of finite Borel equivalence relations and a sequence $(\si_n)_{n \in \N}$ of Borel edge-operators on $X$, where $F_{n+1} \trans_{F_n} E_0$ and $\si_n$ is an $E_0$-based well-iterated edge sliding that graphs $F_{n+1}$ over $F_{n-1}$. More precisely, supposing that $F_n$ and $\bsi_n$ are defined, apply \cref{finite_factor:averaging} to $(\Xmod{F_n}, \mu \rest{\Xmod{F_n}})$, $\Emod{F_n}$, $\Emod[(E_0)]{F_n}$, $\Gmod[\bsi_n(G)]{F_n}$, $\mean{F_n}(f_n)$, and $\wmod{F_n}$, and get $F_{n+1}'$ and $\si_n'$, respectively; take $F_{n+1} \defeq F_{n+1}' \vee F_n$ and let $\si_n$ be the natural lift of $\si_n'$. By \cref{connecting_unions_using_images}, the composition $\si$ of $(\si_n)_{n \ge 1}$ satisfies \labelcref{item:hyperfinite_factor:ergodic:edge_sliding}. 
	
	As for \labelcref{item:hyperfinite_factor:ergodic:F_ergodic_if_possible}, if there was $n \in \N$ such that
	$
	E = E_0 \vee F_n \text{ a.e.},
	$
	we would be done, so suppose otherwise. Then, for each $n \in \N$,
	\[
	\oscmu\big(\mean{F_{n+1}}(f_n)\big) \le \frac{2}{3} \oscmu\big(\mean{F_n}(f_n)\big).
	\]
	
	\begin{claim+}
		For each $f \in \DC$, $\lim_n \oscmu\big(\meanf{F_n}\big) = 0$.
	\end{claim+}
	\begin{pf}
		Indeed, there is an infinite subsequence $(f_{n_k})_k$ with each $f_{n_k}$ equal to $f$, so
		\[
		\oscmu\big(\meanf{F_{n_{k + 1}}}\big) 
		\le
		\oscmu\big(\meanf{F_{n_k + 1}}\big) 
		\le 
		\frac{2}{3} \oscmu\big(\meanf{F_{n_k}}\big)
		\le
		\hdots
		\le
		\left(\frac{2}{3}\right)^{k+1} \oscmu\big(\meanf{F_{n_0}}\big),
		\]
		and the conclusion follows by the fact that $\oscmu\big(\meanf{F_n}\big)$ is nonincreasing in $n$.
	\end{pf}
	
	This claim implies that for each $f \in \DC$, $\lim_n \meanf{F_n}$ exists a.e. and is a constant function. Because $\DC$ is dense, it follows that $F$ is ergodic.
\end{proof}


\section{Hyperfinite asymptotic means and ergodicity}\label{sec:hyp_asymp_means}

Throughout this section, we fix a standard measure space $(X,\mu)$ equipped with a Borel weight function $w : X \to [0, \w)$ and use \cref{notation:rho-size_mean-of-f}.

We also denote by $\mu_w$ the measure on $X$ defined by $d \mu_w = w d \mu$. Note that for a smooth $\mu$-preserving Borel equivalence relation $F$ on $X$, the quotient measure $(\mu_w) / F$ on $\Xmod{F}$ is equal to the measure $(\mu \rest{\Xmod{F}})_{(\wmod{F})}$, and we simply write $\mu_{\wmod{F}}$ for the latter.

\subsection{Finite means}

\begin{lemma}[Convexity of mean]\label{convexity_of_mean}
	Let $U,V \in \FinX$ be disjoint with $|U|_w, |V|_w > 0$.
	
	\begin{enumref}{a}{convexity_of_mean}
		\item \label{item:convexity_of_mean:disjU_of_means}
		$\meanwf{U \cup V} = \frac{|U|_w}{|U|_w + |V|_w}\meanwf{U} + \frac{|V|_w}{|U|_w + |V|_w}\meanwf{V}$. 
		
		\item \label{item:convexity_of_mean:increment_bound}
		$|\meanwf{U \cup V} - \meanwf{U}| 
		\le 
		2 \, \Linf{f} \frac{|V|_w}{|U|_w + |V|_w}
		\le 
		2 \, \Linf{f} \frac{|V|_w}{|U|_w}$.
	\end{enumref}
\end{lemma}
\begin{proof}
	One verifies \labelcref{item:convexity_of_mean:disjU_of_means} directly, and \labelcref{item:convexity_of_mean:increment_bound} follows from \labelcref{item:convexity_of_mean:disjU_of_means} and the triangle inequality.
\end{proof}

\begin{prop}\label{fin_means}
	Let $F$ be a finite measure-preserving Borel equivalence relation on $(X,\mu)$ and let $f \in L^1(X,\mu_w)$.
	
	\begin{enumref}{a}{fin_means}
		\item \label{item:fin_means:equality} $\displaystyle\int_X f d\mu_w = \int_{\Xmod{F}} \meanwf{F} \, d \mu_{\wmod{F}} = \int_X \meanwf{F} \, d \mu_w$.
		
		\item \label{item:fin_means:Lone} $\Lonew{\meanwf{F}} \le \Lonew{f}$.
	\end{enumref}
\end{prop}
\begin{proof}
	\labelcref{item:fin_means:Lone} is immediate from \labelcref{item:fin_means:equality} and the fact that $|\meanwf{F}| \le \meanw{F}(|f|)$. As for \labelcref{item:fin_means:equality}, for each $n \in \N$, restrict to the part where each $F$-class has size $n$, take an automorphism $T$ that induces $F$, recall that $S \defeq \Xmod{F}$ is just a Borel transversal for $F$, and use the invariance of $\mu$ to deduce
	\[
	\displaystyle
	\int_X f d \mu_w
	=
	\sum_{i < n} \int_{T^i(S)} f w \, d \mu
	= 
	\int_S \sum_{i < n} (f \circ T^i) (w \circ T^i) d \mu
	=
	\int_S \meanwf{F} \cdot \wmod{F} \, d \mu
	\]
	and, conversely, also using the $T$-invariance of $\meanwf{F}$,
	\[
	\displaystyle
	\int_S \meanwf{F} \cdot \wmod{F} \, d \mu
	=
	\int_S \sum_{i < n} \big(\meanwf{F} \circ T^i\big) \big(w \circ T^i\big) d \mu(x)
	= 
	\sum_{i < n} \int_{T^i(S)} \meanwf{F} \cdot w \, d \mu
	=
	\int_X \meanwf{F} d \mu_w
	.
	\qedhere
	\]
\end{proof}

\subsection{Hyperfinite means: a pointwise ergodic theorem} \label{subsec:ptwise_erg_for_hyperfinite}

The following is a folklore theorem among descriptive set theorists as it easily follows from the pointwise ergodic theorem for $\Z$-actions. We give a direct proof of it here.

\begin{theorem}[Pointwise ergodic theorem for hyperfinite equivalence relations]\label{hyperfin_means}
	Let $E$ be a Borel measure-preserving hyperfinite equivalence relation on $(X,\mu)$ and let $w \in L^1(X,\mu)$ be a nonnegative Borel function. For any $f \in L^1(X,\mu_w)$ and any witness $(E_n)_n$ to the hyperfiniteness of $E$, the pointwise limit
	$$
	\meanwf{E} \defeq \lim_{n \to \w} \meanwf{E_n}
	$$
	exists a.e. and is independent of the choice of the witness to the hyperfiniteness of $E$ modulo $\mu$-null, i.e. for two different witnesses, the corresponding limits are equal a.e. Furthermore,
	\begin{enumref}{a}{hyperfin_means}
		\item \label{item:hyperfin_means:equality} $\int_Y \meanwf{E} \,d \mu_w = \int_Y f \,d \mu_w$ for any $E$-invariant Borel $Y \subseteq X$.
		
		\item \label{item:hyperfin_means:Lone} $\Lonew{\meanwf{E}} \le \Lonew{f}$.
	\end{enumref}
\end{theorem}

We call $\meanwf{E}$ the (\emph{$w$-weighted}) \emph{mean of $f$ over $E$}.

\begin{proof}
	Granted that the pointwise limit exists a.e., we first deduce the rest. \labelcref{item:fin_means:Lone} enables the use of Dominated Convergence Theorem, which then implies \labelcref{item:hyperfin_means:Lone}, as well as \labelcref{item:hyperfin_means:equality,item:hyperfin_means:Lone}. The independence of the witness follows immediately from \labelcref{item:hyperfin_means:equality}.
	
	\smallskip
	
	Turning now to the existence of the pointwise limit, let
	$
	\overf \defeq \limsup_{n \to \w} \meanwf{E_n}
	$
	and
	$
	\underf \defeq \liminf_{n \to \w} \meanwf{E_n},
	$
	and suppose towards a contradiction that for some $a < b$, the set $X'$ of all $x \in X$, for which $\underf{E}(x) < a < b < \overf(x)$ is $\mu$-positive. This set $X'$ is $E$-invariant and Borel, so we may assume without loss of generality that $X' = X$. 
	
	\smallskip
	
	Fix $\e > 0$ such that $(b - a) \mu(X) > \e (2 + |a| + |b|)$.
	
	\begin{claim+}
		There is a finite Borel equivalence relation $F \subseteq E$ such that the set 
		\[
		Z \defeq \set{x \in X : \meanwf{F} > a}
		\]
		carries less than $\e$ of the total $\mu_w$-weight of $1$ and $f$, i.e. $\mu_w(Z) + \Lonew{f \cdot \1_Z} < \e$.
	\end{claim+}
	\begin{pf}
		Because $\underf \le a$ and $\Lonew{f} < \w$, there is $N \in \N$ such that the set
		\[
		Z \defeq \set{x \in X : (\forall n < N)\; \meanwf{E_n}(x) > a}
		\]
		satisfies $\Lonew{\1_Z} + \Lonew{f \cdot \1_Z} < \e$. Assuming, as we may, that $E_0$ is just the equality relation on $X$, define a function $k : X \to \N$ as follows: if $x \in Y$, then let $k(x)$ be largest natural number less than $N$ with $\meanwf{E_n}(x) \le a$. Noting that $k^{-1}(n)$ is $E_n$-invariant, we see that the sets $[x]_{E_{k(x)}}$ are pairwise disjoint when distinct, so we let $F$ be the equivalence relation whose classes are exactly the sets $[x]_{E_{k(x)}}$. Therefore, for each $x \in X \setminus Z$, $\meanwf{F}[x] \le a$.
	\end{pf}
	
	Using \labelcref{item:fin_means:equality}, we compute
	\[
	\int_X f d\mu_w \approx_\e \int_{X \setminus Z} f d\mu_w = \int_{X \setminus Z} \meanwf{F} d\mu_w \le a \mu_w(X \setminus Z) \approx_{|a| \e} a \mu_w(X),
	\]
	so $\int_X f d\mu_w \le a \mu_w(X) + \e (1 + |a|)$. An analogous argument for $\overf$ and $b$ gives $\int_X f d \mu_w \ge b \mu_w(X) - \e (1 + |b|)$, so $(b - a) \mu(X) \le \e (2 + |a| + |b|)$, contradicting the choice of $\e$.
\end{proof}

Below, we will omit $w$ from the notation $\meanwf{E}$ if $w \equiv 1$.

\subsection{The Cauchy property}\label{subsec:Cauchy}

Throughout this subsection, suppose that $\mu$ is a finite measure and let $E$ be a Borel equivalence relation on $X$ (not necessarily countable).

For a Borel $f : X \to \R$, put $\osc(f) \defeq \sup f - \inf f$ and define $\oscE(f) : X \to \R$ by
$$
\oscE(f)[x] \defeq \osc(f \rest{[x]_E}).
$$
Note that $\oscE$ is an $E$-invariant universally measurable function and it is Borel if $E$ is a countable equivalence relation.

\begin{defn}[Cauchy property]\label{defn:Cauchy}
	Let $f \in L^1(X,\mu)$.
	\begin{itemize}
		\item Say that a finite Borel subequivalence relation $F \subseteq E$ \emph{$\e$-ties $f$ within $E$}, if for some $F$-invariant $\mu$-co-$\e$ Borel set $X' \subseteq X$,
		\begin{equation}\label{eq:eps-ties}
			\osc_{E \rest{X'}}(\meanf{F} \rest{X'}) < \e.
		\end{equation}
		
		\item Say that an increasing sequence $(F_n)_n$ of finite Borel subequivalence relations of $E$ is \emph{$f$-Cauchy within $E$} if for every $\e > 0$ there is $n \in \N$ such that $F_n$ $\e$-ties $f$ within $E$.
		
		\item We also call a Borel subequivalence relation $F \subseteq E$ \emph{$f$-Cauchy within $E$} if, for every $\e > 0$, every finite Borel subequivalence relation $F_0 \subseteq F$ admits a finite Borel extension to a subequivalence relation $F_1 \subseteq F$ that $\e$-ties $f$ within $E$.
	\end{itemize}
\end{defn}

Because the difference of averages is the average of differences, we have the following.

\begin{obs}\label{var_decreases}
	For a set $Y$, any function $f : Y \to \R$ and any finite Borel equivalence relations $F_0 \subseteq F_1$ on $Y$,
	$
	\sup_{y \in Y} \osc(\meanf{F_1}) \le \sup_{y \in Y} \osc(\meanf{F_0}).
	$
\end{obs}

\begin{prop}[Characterization of $f$-Cauchy for hyperfinite]\label{char_of_Cauchy}
	Let $F \subseteq E$ be a hyperfinite Borel $\mu$-preserving subequivalence relation and let $(F_n)_n$ be a witness to the hyperfiniteness of $F$. For any Borel function $f : X \to \R$, the following are equivalent:
	\begin{enumerate}[(1)]
		\item \label{item:char_of_Cauchy:defn} $(F_n)_n$ is $f$-Cauchy within $E$.
		
		\item \label{item:char_of_Cauchy:defn_large_enough_n} For every $\e > 0$, $F_n$ $\e$-ties $f$ within $E$ for all large enough $n$.
		
		\item \label{item:char_of_Cauchy:union_is_Cauchy} $F$ is $f$-Cauchy within $E$.
		
		\item \label{item:char_of_Cauchy:meanf_invariant} The hyperfinite mean function $\meanf{F}$ is $E$-invariant modulo $\mu$-null.
	\end{enumerate}
\end{prop}
\begin{proof}
	\noindent \labelcref{item:char_of_Cauchy:defn}$\imp$\labelcref{item:char_of_Cauchy:defn_large_enough_n}: Follows from \cref{var_decreases}.
	
	\medskip
	
	\labelcref{item:char_of_Cauchy:defn_large_enough_n}$\imp$\labelcref{item:char_of_Cauchy:union_is_Cauchy}: Let $F' \subseteq F$ be a finite Borel subequivalence relation and let $\e > 0$. Then, by the finiteness of $\mu$, the set $X_n \defeq \set{x \in X : [x]_{F'} \subseteq [x]_{F_n}}$ is $\mu$-co-$\frac{\e}{2}$ for all large enough $n$, which, combined with \labelcref{item:char_of_Cauchy:defn_large_enough_n}, implies \labelcref{item:char_of_Cauchy:union_is_Cauchy}.
	
	\medskip
	
	\labelcref{item:char_of_Cauchy:union_is_Cauchy}$\imp$\labelcref{item:char_of_Cauchy:defn}: Fix $\e > 0$ and let $F' \subseteq F$ be a finite Borel subequivalence relation that $\frac{\e}{2}$-ties $f$ within $E$. If $n$ is large enough so that $F' \subseteq F_n$ on a $\mu$-co-$\frac{\e}{2}$ set, then $F_n$ $\e$-ties $f$ within $E$ due to \cref{var_decreases}.
	
	\medskip
	
	\labelcref{item:char_of_Cauchy:defn_large_enough_n}$\imp$\labelcref{item:char_of_Cauchy:meanf_invariant}: Letting $(\e_k)_k$ be a summable sequence of positive reals, the Cauchy condition gives a subsequence $(F_{n_k})_k$ and, for each $k$, an $F_{n_k}$-invariant $\mu$-co-$\e_k$ Borel set $X_k$ such that
	$$
	\osc_E(\meanf{F_{n_k}} \rest{X_k}) < \e_k.
	$$
	The sequence $(F_{n_k})_k$ is a witness to the hyperfiniteness of $F$, so, modulo $\mu$-null,
	$$
	\meanf{F} = \lim_{k \to \w} \meanf{F_{n_k}},
	$$ 
	and we assume, as we may, that this holds everywhere.
	
	By the Borel--Cantelli lemma, the set of points $x \in X$ that don't make it into $X_k$ for arbitrarily large $k$ is $\mu$-null. Throwing these points out, we may assume that every point $x \in X$ is in $X_k$ for all large enough $k$. Thus, for any $\e > 0$ and any two $E$-equivalent points $x,y \in X$,
	$
	\meanf{F_{n_k}}(x) \approx_\e \meanf{F_{n_k}}(x)
	$
	for large enough $k$. On the other hand, for large enough $k$, we also have
	$$
	\meanf{F}(x) \approx_\e \meanf{F_{n_k}}(x) \text{ and } \meanf{F}(y) \approx_\e \meanf{F_{n_k}}(y),
	$$
	so $\meanf{F}(x) \approx_{3\e} \meanf{F}(y)$, so $\meanf{F}$ is $E$-invariant since $\e$ is arbitrary.
	
	\medskip
	
	\labelcref{item:char_of_Cauchy:meanf_invariant}$\imp$\labelcref{item:char_of_Cauchy:defn}: Throwing out a $\mu$-null set, we may assume that $\meanf{F} = \lim_n \meanf{F_n}$ everywhere and $\meanf{F}(x) = \meanf{F}(y)$ for any two $E$-equivalent points $x, y \in X$.
	
	Fix $\e > 0$. We know that for every $x \in X$ there is $n \in \N$ such that $\meanf{F_n}(x) \approx_\e \meanf{F}(x)$. Switching the quantifiers using the finiteness of $\mu$, we obtain a $\mu$-co-$\e$ set $X' \subseteq X$ such that, for all $x \in X'$, $\meanf{F_n}(x) \approx_\e \meanf{F}(x)$ and, replacing $X'$ by $[X']_{F_n}$, we may assume that $X'$ is $F_n$-invariant. Now for any two $E$-invariant $x,y \in X'$, we have
	$$
	\meanf{F_n}(x) \approx_\e \meanf{F}(x) = \meanf{F}(y) \approx_\e \meanf{F_n}(y),
	$$
	so $\meanf{F_n}(x) \approx_{2\e} \meanf{F_n}(y)$. Thus, $F_n$ $(2 \e)$-ties $f$ within $E$.
\end{proof}

\subsection{Relative ergodicity}\label{subsec:char_of_rel_erg}

Recall that a Borel equivalence relation $E$ on $X$ is called \emph{$\mu$-ergodic} if every $E$-invariant Borel set is $\mu$-null or $\mu$-conull. Here we relativize this definition for subequivalence relations of a given ambient equivalence relation.

\begin{defn}
	For Borel equivalence relations $F \subseteq E$ on $(X,\mu)$, say that $F$ is \emph{$\mu$-ergodic within} (or \emph{relative to}) \emph{$E$} if every $F$-invariant Borel set is also $E$-invariant modulo $\mu$-null\footnote{This $\mu$-null set does not have to be $E$-invariant.}.
\end{defn}

As with ergodicity, the density of simple functions in $L^1$ yields:

\begin{obs}\label{relative_erg:functions}
	$F$ is $\mu$-ergodic within $E$ if and only if every $F$-invariant function $f \in L^1(X,\mu)$ is also $E$-invariant modulo $\mu$-null.
\end{obs}

Note that if $E$ itself is $\mu$-ergodic, then $F$ being $\mu$-ergodic within $E$ is equivalent to $F$ being $\mu$-ergodic. In particular, $F$ being $\mu$-ergodic is equivalent to $F$ being $\mu$-ergodic within the trivial equivalence relation $E \defeq X^2$.

Lastly, for measure-preserving Borel graphs $G,H$ on $(X,\mu)$ with $E_H \subseteq E_G$, say that $H$ is \emph{$\mu$-ergodic within} (or \emph{relative to}) \emph{$G$} if $E_H$ is $\mu$-ergodic within $E_G$.

\medskip

We characterize relative ergodicity via the Cauchy property.

\begin{theorem}[Characterization of relative ergodicity for hyperfinite]\label{char_of_ergodicity}
	Let $F \subseteq E$ be a $\mu$-preserving hyperfinite Borel subequivalence relation. For any dense family $\DC \subseteq L^1(X,\mu)$ of bounded Borel functions, the following are equivalent:
	\begin{enumerate}[(1)]
		\item \label{item:char_of_ergodicity:defn} $F$ is $\mu$-ergodic within $E$.
		
		\item \label{item:char_of_ergodicity:meanf_invariant} For any $f \in \DC$, the hyperfinite mean function $\meanf{F}$ is $E$-invariant $\mu$-a.e.
		
		\item\label{item:char_of_ergodicity:any_witness_Cauchy} For any $f \in \DC$, any witness to the hyperfiniteness of $F$ is $f$-Cauchy within $E$.
		
		\item \label{item:char_of_ergodicity:F_Cauchy} For any $f \in \DC$, $F$ is $f$-Cauchy within $E$.
		
		\item \label{item:char_of_ergodicity:some_witness_Cauchy} For any $f \in \DC$, some witness to the hyperfiniteness of $F$ is $f$-Cauchy within $E$.
		
		\item \label{item:char_of_ergodicity:some_subwitness_Cauchy} There is an increasing sequence of finite Borel subequivalence relations $(F_n)_n$ of $F$ that is $f$-Cauchy within $E$ for all $f \in \DC$.
	\end{enumerate}
\end{theorem}
\begin{proof}
	\noindent \labelcref{item:char_of_ergodicity:defn}$\imp$\labelcref{item:char_of_ergodicity:meanf_invariant}: Follows from \cref{relative_erg:functions} and that the function $\meanf{F}$ is in $L^1(X,\mu)$ by \labelcref{item:hyperfin_means:Lone}.
	
	\medskip
	
	\noindent \labelcref{item:char_of_ergodicity:meanf_invariant}$\imp$\labelcref{item:char_of_ergodicity:defn}: Fix any $F$-invariant Borel function $g \in L^\w(X,\mu)$ in order to show that it is $E$-invariant a.e. Also fix an $\e > 0$ and let $f \in \DC$ such that $\Lone{g - f} < \e$. Then, for any finite Borel subequivalence relation $F' \subseteq F$, \labelcref{item:fin_means:Lone} gives:
	\[
	\Lone{\meanf{F'} - g} = \Lone{\meanf{F'} - \mean{F'}(g)} = \Lone{\mean{F'}(f - g)} \le \Lone{f - g} < \e.
	\]
	So, it follows from the definition of $\meanf{F}$ (\cref{hyperfin_means}) and the dominated convergence theorem that $\Lone{\meanf{F} - g} \le \e$. Since $\e$ is arbitrary, this means that $g$ is an $L^1$-limit of $E$-invariant functions, which implies that $g$ itself must be $E$-invariant a.e.
	
	Because $\DC$ is dense, it is enough to show that every $F$-invariant $f \in \DC$ is also $E$-invariant modulo $\mu$-null. But, because $f$ is $F$-invariant, $\meanf{F} = f$ by definition, so \labelcref{item:char_of_ergodicity:meanf_invariant} implies that $f$ too is $E$-invariant.
	
	\medskip
	
	\noindent \labelcref{item:char_of_ergodicity:meanf_invariant}$\shortiff$\labelcref{item:char_of_ergodicity:any_witness_Cauchy}$\shortiff$\labelcref{item:char_of_ergodicity:F_Cauchy}$\shortiff$\labelcref{item:char_of_ergodicity:some_witness_Cauchy}: By \cref{char_of_Cauchy}.
	
	\medskip
	
	\noindent \labelcref{item:char_of_ergodicity:some_witness_Cauchy}$\imp$\labelcref{item:char_of_ergodicity:some_subwitness_Cauchy}: Trivial. 
	
	\medskip
	
	\noindent \labelcref{item:char_of_ergodicity:some_subwitness_Cauchy}$\imp$\labelcref{item:char_of_ergodicity:meanf_invariant}: Letting $F' \defeq \bigcup_n F_n$, the implication \labelcref{item:char_of_ergodicity:some_witness_Cauchy}$\imp$\labelcref{item:char_of_ergodicity:defn} applied to $F'$ in lieu of $F$ yields the ergodicity of $F'$ within $E$, and hence also the ergodicity of $F$ within $E$.
\end{proof}


\section{Asymptotic means along graphs}\label{sec:asymptotic_means}

Throughout this section, we fix a set $X$ and a weight function $w : X \to [0, \w)$. Below the term \emph{$w$-large} used in \emph{arbitrarily $w$-large} and \emph{$w$-large enough} $U \in \FinX$ refers to $|U|_w$ being arbitrarily large and large enough, respectively.

\subsection{Asymptotic means along discrete graphs}

For this subsection, let $G$ be a locally countable graph on $X$ such that for each $G$-connected component $C \subseteq X$, $|C|_w = \w$.  We also let $f : X \to \R$ be a bounded function. 

\begin{lemma}\label{approx_invariance_of_means}
	For $r \in [0,1]$, $U_0 \in \FinX$, and $\e > 0$, if there are arbitrarily $w$-large $U \in \FinG$ containing $U_0$ and satisfying $\meanwf{U} \approx_\e r$, then, for any $U_1 \in \FinG$ within $[U_0]_G$, there are arbitrarily $w$-large $V \in \FinG$ containing $U_0 \cup U_1$ and satisfying $\meanwf{V} \approx_{2\e} r$.
\end{lemma}
\begin{proof}
	Replacing $U_1$ with any set $U_1' \in \FinG$ containing $U_0 \cup U_1$, we may assume that $U_0 \subseteq U_1$ to begin with. Take a $w$-large enough $U \in \FinG$ with $x \in U$ and $\meanwf{U} \approx_\e r$ so that 
	\[
	\frac{|U_1|_w}{|U|_w} < \frac{\e}{2 \Linf{f}}.
	\]
	Observe that $V \defeq U_1 \cup U \in \FinG$ and, by \labelcref{item:convexity_of_mean:increment_bound},
	\[
	|\meanwf{V} - \meanwf{U}| \le 2 \, \Linf{f} \frac{|U_1|_w}{|U|_w} < \e,
	\]
	so $\meanwf{V} \approx_{2\e} r$.
\end{proof}

\begin{defn}\label{defn:asymptotic_means}
	Call $r \in \R$ an \emph{asymptotic mean of $f$ at $U_0 \in \FinG$ along $G$} if for every $\e > 0$ there is $U \in \FinG$ with $U_0 \subseteq U$, $|U|_w > \e^{-1}$, and $|r - \meanwf{U}| < \e$. We denote by $\MSwGf[U_0]$ the \emph{set of asymptotic means} of $f$ at $U_0$ along $G$, and we simply write $\MSwGf[x]$ when $U_0 = \set{x}$.
\end{defn}

Note that, by the compactness of $[\inf f, \sup f]$, $\MSwGf[x] = \0$ if and only if $|[x]_G|_w < \w$.

\begin{prop}[Invariance of the means]\label{invariance_of_means}
	$\MSwGf$ is constant on each $G$-connected component $C$, i.e. for any $U_0, U_1 \in \FinG[C]$, $\MSwGf[U_0] = \MSwGf[U_1]$.
\end{prop}
\begin{proof}
	Immediate from \cref{approx_invariance_of_means}.
\end{proof}

\begin{lemma}[Intermediate Value Property]\label{intermediate_value_property}
	Let $U,V \in \FinG$ such that $U \subseteq V$ and $|U|_w > 0$, and put
	\[
	\De \defeq \frac{\Linf{f \rest{V \setminus U}}\Linf{w \rest{V \setminus U}}}{|U|_w}.
	\]
	For every $r \in [\meanwf{U}, \meanwf{V}]$, there is $W \in \FinG$ with $U \subseteq W \subseteq V$ and $\meanwf{W} \approx_\De r$.
	
\end{lemma}
\begin{proof}
	By replacing $f$ with $-f$ if necessary, we may assume that $\meanwf{U} < \meanwf{V}$. Fix $r \in \big(\meanwf{U}, \meanwf{V}\big)$ and let $W$ be an inclusion-maximal set in $\FinG$ with $U \subseteq W \subseteq V$ and $\meanwf{W} < r$.
	
	Now if $\meanwf{W} + \De \ge r$, we are done, so suppose $\meanwf{W} + \De < r$. Because $\meanwf{W} < r < \meanwf{V}$, $W \ne V$, so there is $W' \in \FinG$ with $W \subseteq W' \subseteq V$ such that $|W' \setminus W| = 1$. By the maximality of $W$, it must be that $\meanwf{W'} \ge r$, so $\meanwf{W} + \De < r \le \meanwf{W'}$. On the other hand, \labelcref{item:convexity_of_mean:increment_bound} implies that $\meanwf{W'} - \meanwf{W} \le 2 \De$, so $\meanwf{W'} < r + \De$ and hence $\meanwf{W'} \approx_{\De} r$.
\end{proof}

\begin{prop}\label{means_form_interval}
	For any $x \in X$, $\MSwGf[x]$ is a closed. Moreover, if $w$ is bounded, then, $\MSwGf[x]$ is an interval.
\end{prop}
\begin{proof}
	It is clear from the asymptotic nature of its definition that $\MSwGf[x]$ is closed. 
	
	Assume $w$ is bounded. To show that $\MSwGf[x]$ is an interval, suppose towards a contradiction that there is a gap in between, i.e. there are $a,b \in [\inf f, \sup f]$ such that $\min \MSwGf[x] < a < b < \max \MSwGf[x]$ and $[a, b] \cap \MSwGf[x] = \0$. Whence, there is $N \in \N$ such that
	\begin{equation}\label{eq:gap}
		\text{for any } W \in \FinG \text{ with } x \in W \text{ and } |W|_w \ge N, \; \meanwf{W} \notin [a,b].
	\end{equation}
	We may take $N$ large enough so that $\frac{\Linf{f} \Linf{w}}{N} < \frac{b-a}{2}$.
	
	Because $\min \MSwGf[x] < a$, there is $U \in \FinG$ with $x \in U$, $|U|_w \ge N$, and $\meanwf{U} < a$. On the other hand, due to $\max \MSwGf[x] > b$ and \cref{invariance_of_means}, there is $V \in \FinG$ with $U \subseteq V$ and $\meanwf{V} > b$. But then, Intermediate Value Property \cref{intermediate_value_property} gives $W \in \FinG$ with $U \subseteq W \subseteq V$ and $\meanwf{W} \approx_\De \frac{a + b}{2}$, where
	\[
	\De \defeq \frac{\Linf{f} \Linf{w}}{|U|_w} \le \frac{\Linf{f} \Linf{w}}{N} < \frac{b-a}{2},
	\]
	so $\meanwf{W} \in [a,b]$, contradicting \labelcref{eq:gap}.
\end{proof}

Lastly, we observe that taking finite quotients can only shrink the set $\MSwGf$.

\begin{prop}\label{quotienting_shrinks_MS}
	For a finite equivalence relation $F \subseteq E_G$, identifying $\Xmod{F}$ with a transversal for $E_G$ and letting $\Gmod{F}$ be the quotient graph on $\Xmod{F}$, $\MC^{\wmod{F}}_{\Gmod{F}}(\meanwf{F} \rest{\Xmod{F}}) \subseteq \MSwGf$.
\end{prop}
\begin{proof}
	Follows from the fact that any $\Gmod{F}$-connected set lifts to a $G$-connected $F$-invariant set.
\end{proof}

\subsection{Asymptotic means along Borel and measurable graphs}

Equipping $X$ with a standard Borel structure, suppose that $G$ is a locally countable Borel graph on $X$ and that $w : X \to [0,\w)$ is a Borel weight-function such that $|C|_w = \w$ for every $G$-connected component $C \subseteq X$. Also, fix a bounded Borel function $f \in L^1(X, \mu_w)$.

By \cref{means_form_interval}, $\MSwGf[x]$ is a closed subset of $I_f \defeq [\inf f, \sup f]$ for every $x \in X$, so the assignment $x \mapsto \MSwGf[x]$ defines a map from $X$ to the hyperspace $\KC(I_f)$ of compact subsets of $I_f$. It is easy to check that this map is Borel using the Luzin--Novikov theorem, and it is also $E_G$-invariant due to \cref{invariance_of_means}.

\medskip

For an interval $I \subseteq \R$ and $\e > 0$, put $I \pm \e \defeq (\inf I - \e, \sup I + \e)$.

\begin{prop}\label{entire_fsr_with_correct_calc}
	For any $\e > 0$, there is a finite Borel $G$-connected subequivalence relation $F \subseteq E_G$ such that $\meanwf{F}[x] \in \MSwGf[x] \pm \e$ for every $x \in X$ modulo $E_G$-compressible.
\end{prop}
\begin{proof}
	Let $\Phi$ be the collection of all $U \in \FinG$ such that for any $U' \in \FinG$ with $U \subseteq U'$,
	\[
	\meanwf{U'} \in \MSwGf[U'] \pm \e.
	\]
	By \cref{existence_of_saturated_fsr}, throwing out an $E_G$-compressible set, we get a $\Phi$-saturated Borel \fsr $F$. Due to \cref{rich-saturated_is_entire}, to show that $F$ is entire, it is enough to show that $\Phi$ is rich.
	
	To this end, note that for every $x \in X$, there is $N_x \in \N$ such that for all $U \in \FinG$ with $x \in U$ and $|U|_w \ge N_x$, $\meanwf{U} \in \MSwGf[U] \pm \e$. Thus, $\bigcup \Phi = X$. 
	
	Fixing $\Psi \subseteq \Phi$ and $x \notin D \defeq \bigcup \Psi$, it is only worth considering the case when $x \in [D]_G$. Then, there is $U \in \FinG$ such that $x \in U$ and $U_0 \defeq U \cap D \in \Psi$. By the virtue of $U_0 \in \Phi$ and $U \supseteq U_0$, we have $\meanwf{U'} \in \MSwGf[U'] \pm \e$ for any $U' \in \FinG$ containing $U$. Whence, $U \in \Phi$ witnessing the richness of $\Phi$.
\end{proof}

Now let $\mu$ be an $E_G$-invariant Borel probability measure on $X$.

\begin{prop}
	If $G$ is hyperfinite, then $\meanwf{E_G}[x] \in \MSwGf[x]$ for $\mu$-a.e. $x \in X$.
\end{prop}
\begin{proof}
	This is simply due to \cref{graph-hyperfinite}.
\end{proof}

\begin{prop}\label{between_min_mix}
	If $G$ is $\mu$-ergodic and $\mu_w(X) < \w$, then $\inf \MSwGf \le \frac{1}{\mu_w(X)} \int_X f d \mu_w \le \sup \MSwGf$. In particular, $\frac{1}{\mu_w(X)} \int_X f d \mu_w \in \MSwGf$.
\end{prop}
\begin{proof}
	By the ergodicity of $E_G$ and $E_G$-invariance of the map $x \mapsto \MSwGf[x]$, it must be constant a.e., so we may assume that $\MSwGf$ is some closed subset of $I_f$. We only show that $\int_X f d \mu_w \le \mu_w(X) \cdot \sup \MSwGf$ as the other inequality is proven analogously. 
	
	Take an arbitrary $\e > 0$ and, applying \cref{entire_fsr_with_correct_calc}, get a finite Borel $G$-connected subequivalence relation $F \subseteq E_G$ with $\meanwf{F}[x] < \sup \MSwGf + \e$ for $\mu$-a.e. But then, \labelcref{item:fin_means:equality} implies that $\int_X f \, d \mu_w = \int_X \meanwf{F} \, d \mu_w \le \mu_w(X) \cdot (\sup \MSwGf + \e)$, so we are done because $\e$ is arbitrary.
\end{proof}


\section{Finitizing cuts and hyperfiniteness}\label{sec:finitizing_cuts}

\subsection{Vanishing sequences of finitizing cuts}

Let $X$ be a standard Borel space and $G$ a locally countable Borel graph on it.

\begin{defn}
	Call a subset $C \subseteq X$ (resp. $H \subseteq X^2$) a \emph{finitizing vertex-cut} (resp.  \emph{edge-cut}) \emph{for $G$} if $G \rest{X \setminus C}$ (resp. $G \setminus H$) is component-finite. 
\end{defn}

Call a sequence of sets \emph{vanishing} if it is decreasing and has empty intersection.

\begin{lemma}\label{vanishing_vertex-cuts=edge-cuts}
	$G$ admits a vanishing sequence of finitizing Borel vertex-cuts if and only if it admits a vanishing sequence of finitizing Borel edge-cuts.
\end{lemma}
\begin{proof}
	For a vertex-cut $C$, $\EdgsBtw{C}{X}$ is an edge-cut; and conversely, for an edge-cut $H$, $\dom(H)$ is a vertex-cut.
\end{proof}

\begin{prop}\label{vanishing_vertex-cuts->hyperfinite}
	If $G$ admits a vanishing sequence $(C_n)_n$ of Borel finitizing vertex-cuts, then it is hyperfinite.
\end{prop}
\begin{proof}
	For each $n$, define an equivalence relation $F_n$ on $X$ by
	$$
	x F_n y \defequiv \text{$x$ and $y$ are $G \rest{X \setminus C_n}$-connected}.
	$$
	Clearly the $F_n$ are finite and increasing. Moreover, because the $C_n$ are vanishing, the $F_n$ union up to $E_G$.
\end{proof}

The converse of the last proposition is not true in general, see \cref{examples:hyperfinite_doesnot_imply_fvp=0}, but it is true for locally finite graphs.

\begin{prop}\label{hyperfinite->vanishing_cuts}
	Suppose that $G$ is locally finite. If $G$ is hyperfinite, then it admits a vanishing sequence of Borel finitizing vertex-cuts, as well as a vanishing sequence of Borel finitizing edge-cuts.
\end{prop}
\begin{proof}
	For any witness $(F_n)_n$ to the hyperfiniteness of $E_G$, the sets
	$$
	C_n \defeq \set{x \in X : N_G(x) \nsubseteq [x]_{F_n}}
	$$
	form a vanishing sequence of finitizing vertex-cuts. We get edge-cuts from \cref{vanishing_vertex-cuts=edge-cuts}.
\end{proof}

We have obtained the following.

\begin{cor}\label{char_of_hyperfiniteness_via_vanishing_cuts}
	For a locally finite Borel graph $G$, the following are equivalent:
	\begin{enumerate}[(1)]
		\item \label{item:char_of_hyperfiniteness_via_vanishing_cuts:defn} $G$ is hyperfinite.
		
		\item \label{item:char_of_hyperfiniteness_via_vanishing_cuts:vanishing_vertex-cuts} $G$ admits a vanishing sequence of Borel vertex-cuts.
		
		\item \label{item:char_of_hyperfiniteness_via_vanishing_cuts:vanishing_edge-cuts} $G$ admits a vanishing sequence of Borel edge-cuts.
	\end{enumerate}
\end{cor}

\subsection{Finitizing cut price}

Let $(X,\mu)$ be a standard measure space and let $G$ be a locally countable measure-preserving Borel graph on it. 

\begin{defn}
	The \emph{finitizing vertex-cut price} and \emph{the finitizing edge-cut price} of $G$ are the following quantities, respectively:
	\begin{align*}
		\fvp(G) 
		&\defeq 
		\inf \set{\mu(C) : C \subseteq X \text{ is a Borel finitizing vertex-cut for $G$}},
		\\
		\fep(G) 
		&\defeq 
		\inf \set{C_\mu(H) : H \subseteq G \text{ is Borel finitizing edge-cut for $G$}}.
	\end{align*}
\end{defn}

A sequence $(C_n)_n$ of subsets of $X$ is said to be \emph{$\mu$-vanishing} if it is decreasing and its intersection is $\mu$-null. Similarly, a sequence $(H_n)_n$ of subsets of $X^2$ is said to be \emph{$\mu$-vanishing} if it is decreasing and its intersection is of $\mu$-cost $0$.

\begin{lemma}\label{fcp=0->vanishing_cuts}
	If $\fvp(G) = 0$, then there is a $\mu$-vanishing sequence $C_n \subseteq X$ of Borel finitizing vertex-cuts for $G$. Similarly, if $\fep(G) = 0$, then there is a $\mu$-vanishing sequence $(H_n)_n$ of Borel finitizing edge-cuts for $G$.
\end{lemma}
\begin{proof}
	Let $(C_n)_n$ be a sequence of Borel finitizing vertex-cuts with $\mu(C_n) < 2^{-n}$. Put $D_n \defeq \bigcup_{i > n} C_i$, so the $D_n$ are decreasing, and each $D_n$ is a finitizing vertex-cut with $\mu(D_n) < 2^{-n}$; in particular, the $D_n$ are $\mu$-vanishing. The part about $\fep(G)$ is proven analogously. 
\end{proof}

\begin{lemma}\label{fvp_vs_fep}
	Let $G$ be a locally countable measure-preserving graph on $(X,\mu)$.
	\begin{enumref}{a}{fvp_vs_fep}
		\item \label{part:fvp_vs_fep:fvp<2fep} $\fvp(G) \le 2 \fep(G)$.
		
		\item \label{part:fvp_vs_fep:fep<deg_fvp} $\fep(G) \le \deg(G) \fvp(G)$.
		
		\item \label{part:fvp_vs_fep:fvp=0->fep=0} If $C_\mu(G) < \w$, then $\fvp(G) = 0$ implies $\fep(G) = 0$.
	\end{enumref}
\end{lemma}
\begin{proof}
	Part \labelcref{part:fvp_vs_fep:fvp<2fep} follows from the fact that for a Borel $H \subseteq X^2$, the set $C \subseteq X$ of all vertices incident to $H$ has measure at most $2 C_\mu(H)$. Conversely, the set $H \subseteq G$ of edges incident to a Borel set $C \subseteq X$ has cost at most $\deg(G) \mu(C)$, whence part \labelcref{part:fvp_vs_fep:fep<deg_fvp} follows. Lastly, \labelcref{part:fvp_vs_fep:fvp=0->fep=0} follows from \cref{fcp=0->vanishing_cuts,vanishing_vertex-cuts=edge-cuts}.
\end{proof}

\begin{cor}\label{fep=0->fvp=0->hyperfinite}
	For any locally countable measure-preserving Borel graph $G$ on a standard measure space $(X,\mu)$, the following implications hold:
	$$
	\fep(G) = 0 \implies \fvp(G) = 0 \implies \text{$G$ is $\mu$-hyperfinite}.
	$$
\end{cor}
\begin{proof}
	The first implication follows from \labelcref{part:fvp_vs_fep:fvp<2fep} and the second one from \cref{vanishing_vertex-cuts->hyperfinite,fcp=0->vanishing_cuts}.
\end{proof}

\begin{remark}
	Neither of the implications in \cref{fep=0->fvp=0->hyperfinite} can be reversed in general. Indeed, for the first implication, observe that every finitizing edge-cut should contain all but finitely-many edges incident to each vertex, so it has infinite cost if the degree of every vertex is infinite. As for the second implication, its converse is false even for bipartite non-locally-finite Borel graphs as the following example shows.
\end{remark}


\begin{example}\label{examples:hyperfinite_doesnot_imply_fvp=0}
	Let $E$ be a hyperfinite Borel equivalence relation on a standard Borel space $X$ equipped with an $E$-quasi-invariant nonzero measure $\mu$. Take a Borel set $A \subseteq X$ such that both $A$ and $X \setminus A$ meet every $E$-class in infinitely-many points (such sets exist by, for example, the marker lemma \cite{Kechris-Miller}*{Lemma 6.7}). Let $G$ be the complete bipartite graph within $E$ with partitions $A$ and $X \setminus A$, i.e. $G \defeq \set{(x,y) \in E : x \in A \nLeftrightarrow y \in A}$. It is clear that any Borel finitizing vertex-cut has to fully contain at least one of the partitions, so $\fvp(G) \ge \min\set{\mu(A), \mu(X \setminus A)} > 0$.
\end{example}

The next proposition is rather crucial (yet straightforward) as it turns nonhyperfiniteness (equivalently, nonamenability) into a positive (existential) statement. Although the present authors could not find an explicit statement of it in the literature, it has implicitly appeared in a number of places, for example, in \cite{Aldous-Lyons:processes_on_unim_random_networks}*{Theorem 8.5}, \cite{Connes-Feldman-Weiss}, \cite{Elek:fin_graphs_and_amenability}, and \cite{Kaimanovich:amenability_isoperimetric}.

\begin{prop}[Characterization of $\mu$-hyperfiniteness for finite-cost graphs]\label{char_of_mu_hyperfiniteness_via_cuts}
	Let $G$ be a Borel locally countable measure-preserving graph on a standard probability space $(X,\mu)$. If $C_\mu(G) < \w$, then the following are equivalent:
	\begin{enumerate}[(1)]
		\item \label{item:char_of_mu_hyperfiniteness_via_cuts:defn} $G$ is $\mu$-hyperfinite.
		
		\item \label{item:char_of_mu_hyperfiniteness_via_cuts:fep=0} $\fep(G) = 0$.
		
		\item \label{item:char_of_mu_hyperfiniteness_via_cuts:fvp=0} $\fvp(G) = 0$.
	\end{enumerate}
\end{prop}
\begin{proof}
	By \cref{fep=0->fvp=0->hyperfinite}, it remains to show \labelcref{item:char_of_mu_hyperfiniteness_via_cuts:defn}$\imp$\labelcref{item:char_of_mu_hyperfiniteness_via_cuts:fep=0}. The finiteness of $C_\mu(G)$ implies that $G$ is locally finite modulo a $\mu$-null set, which we may throw out. Thus, \cref{hyperfinite->vanishing_cuts} gives a vanishing sequence of Borel edge-cuts, whose cost must converge to $0$, again due to the finiteness of $C_\mu(G)$.
\end{proof}

Below, we will only use the following:

\begin{obs}\label{subtracting_less_than_cut}
	For a Borel $H \subseteq G$, if $C_\mu(H) < \fep(G)$, then $\fep(G \setminus H) > 0$.
\end{obs}


\section{Shortcutting}\label{sec:shortcutting}

\subsection{The shortcutting graph}

Let $G$ be a locally countable graph on $X$ and fix a set $S \subseteq X$.

\begin{defn}
	Call an edge $(x,y) \in X^2$ an \emph{$S$-shortcut of $G$} if there is a path $x = x_0, x_1, ..., x_k = y$ in $G$ with all intermediate vertices in $S$, i.e. $x_i \in S$ for all $0 < i < k$. Define a graph
	$$
	\GS \defeq G \cup \set{e \in X^2 \setminus S^2 : \text{$e$ is an $S$-shortcut of $G$}}
	$$
	and call it the \emph{$S$-shortcutting of $G$}. Note that we do not add to $\GS$ any new edges between two elements of $S$.
\end{defn}

\begin{obss}\label{shortcutting:properties}
	Let $X, G, S$ be as above.
	\begin{enumref}{a}{shortcutting:properties}
		\item \label{item:shortcutting:properties:connected_comp_UcupS} Any $\GS$-connected set $U \subseteq X$ is contained in a connected component of the graph $G \rest{U \cup S}$. In particular, $E_G = E_{\GS}$.
		
		\item \label{item:shortcutting:properties:erasing_S_from_path} Let $P : u = x_0, x_1, \hdots, x_n = v$ be a vertex-path in $G$ from $u$ to $v$ where $v \notin S$ and let $P' : u = y_0, y_1, \hdots, y_m = v$ be the sequence of vertices obtained from $P$ by erasing all vertices $x_i, i \ge 1$ that are in $S$. $P'$ is a path in $\GS$ from $u$ to $v$.
	\end{enumref}
\end{obss}

\labelcref{item:shortcutting:properties:erasing_S_from_path} immediately implies the following key properties of $\GS$, which are our main reasons for defining $\GS$.

\begin{cor}[Key properties of shortcutting]\label{shortcutting:connecting_over_S}
	Let $X,G,S$ be as above and let $V$ be a $\GS$-connected set. 
	\begin{enumref}{a}{shortcutting:connecting_over_S}
		\item \label{item:shortcutting:connected_outside_of_S} $V \setminus S$ is $\GS$-connected.
		
		\item \label{item:shortcutting:extending_outside_of_S} For any subset $U \subseteq V$ with $U \nsupseteq V \setminus S$, there is $v \in V \setminus S$ that is $\HS$-adjacent to $V$.
		
	\end{enumref}
\end{cor}

\subsection{Building shortcuts via edge-sliding}

\begin{lemma}\label{shortcutting:connecting_set}
	Let $U \in \FinX$ be a $\GS$-connected set such that $U \cap S$ is $G$-connected (possibly empty). There is an edge sliding $\si$ along a finite subset $R_U \subseteq \EdgsBtw{S}{S \cup U}$ that moves a subgraph of $\EdgsBtw{U \setminus S}{S \setminus U}$ into $\GS \cap U^2$ so that $\si(G)$ connects $U$.
\end{lemma}
\begin{proof}
	We prove by induction on $|U|$, so suppose the statement is true for the sets of size smaller than $|U|$. If $|U| = 1$ or $U \subseteq S$, then $U$ is already $G$-connected and there is nothing to prove, so suppose that $|U| \ge 2$ and $U \nsubseteq S$. Whence, there is $u \in U \setminus S$ such that $U' \defeq U \setminus \set{u}$ is still $\GS$-connected. Also, $U' \cap S = U \cap S$ is still $G$-connected, so, by induction, there is an edge sliding $\si'$ satisfying the requirements above written for $U'$; in particular, $\si'$ is an edge sliding of $G \rest{U \cup S}$. But by \labelcref{item:shortcutting:properties:connected_comp_UcupS}, $U$ is in a connected component of $G \rest{U \cup S}$, so it follows from \labelcref{item:edge_sliding:preserves_connectivity} that $U$ is in a connected component of $\si'(G \rest{U \cup S})$. 
	
	Let $u = x_0, x_1, \hdots, x_n, x_{n+1} = u'$ be a shortest path in $\si'(G \rest{U \cup S})$ with $u' \in U'$ and denote by $P$ the set of edges on this path. Because it is the shortest path, $x_1, \hdots, x_n \notin U'$ and hence must be in $S$, so none of the edges in $P$ are in $\Mv(\si') \cup \IMv(\si')$. In other words, $P \subseteq \EdgsBtw{S \cup U} {U}$. The edge $e_0 \defeq (u, x_1)$ slides into $e_1 \defeq (u,v) \in \GS \rest{U}$ along $P$, so the function $\si : X^2 \to X^2$ defined by
	$$
	\si(e) \defeq 
	\begin{cases}
	e_1 & \text{if $e = e_0$}
	\\
	-e_1 & \text{if $e = -e_0$}
	\\
	\si'(e) & \text{otherwise}
	\end{cases}
	$$
	is a desired edge sliding.
\end{proof}

Now, instead of just a single set $U$, we will simultaneously connect every set in a disjoint collection.

\begin{lemma}\label{shortcutting:graphing_fsr}
	Let $G$ be a locally countable Borel graph on a standard Borel space $X$, $S \subseteq X$ be a Borel set, and $F \subseteq E_G$ be a Borel \fsr such that for each $F$-class $U$, 
	\begin{enumref}{i}{shortcutting:graphing_fsr}
		\item $U$ is $\GS$-connected and
		
		\item $U \cap S$ is $G$-connected. 
	\end{enumref}	
	There is a Borel edge sliding $\si$ along a subgraph $R \subseteq \EdgsBtw{S}{S \cup \dom(F)}$ that moves a subgraph of $\EdgsBtw{\dom(F) \setminus S}{S} \setminus F$ into $\GS \cap F$ so that $\si(G)$ supergraphs $F$.
\end{lemma}
\begin{proof}
	For each $F$-class $U$, let $\si_U$ and $R_U$ be given by \cref{shortcutting:connecting_set}. In particular, $\Mv(\si_U) \subseteq \EdgsBtw{U \setminus S}{S}$, so the sets $\Mv(\si_U)$ are disjoint for distinct $F$-classes $U$. Therefore,
	$$
	\si \defeq \bigcup_{U \in \Classes{F}} \si_U
	$$
	defines an edge-operator $\si : X^2 \to X^2$. Observe that the set $\EdgsBtw{U \cup S}{S} \supseteq R_U$ is disjoint from $\EdgsBtw{V \setminus S}{S} \supseteq \Mv(\si_V)$ for any two distinct $F$-classes $U,V$, so
	$$
	R \defeq \bigcup_{U \in \Classes{F}} R_U
	$$
	is pointwise fixed by $\si$, and thus, $\si$ is an edge sliding along $R$. It is now obvious that $\si$ satisfies the remaining conditions, except perhaps being Borel. For the latter, note that we can define $\si$ such that it is Borel since the choice of $\si_U$ for each $U$ can be made in a uniformly Borel fashion (using \cref{relative_enum_of_each_class}) as the choice is made among only those edge-operators whose set of non-fixed points is a finite subset of $[U]_{E_G}$.
\end{proof}

\begin{prop}\label{self-shortcutting}
	Let $G$ be a Borel graph on $X$, let $(F_n)_n$ be an increasing sequence of Borel \fsrs of $E_G$, where $F_0 = \0$, and put $F_\w \defeq \bigcup_{n \in \N} F_n$, $S_n \defeq \dom(F_n)$ for each $n \le \w$. Suppose that for each $n \ge 0$ and each $F_{n+1}$-class $U$,
	\begin{enumref}{i}{self-shortcutting}
		\item \label{item:self-shortcutting:connects_Fn+1} $U$ is $\GS[S_n]$-connected and
		
		\item \label{item:self-shortcutting:U_cap_Sn_is_connected} $U \cap S_n$ is $(G \cup F_n)$-connected.
	\end{enumref}
	Then there is a Borel well-iterated edge sliding $\si$ of $\EdgsBtw{S_\w}{S_\w}$ into $F_\w$ graphing every $F_n$.
\end{prop}
\begin{proof}
	Putting $H_n \defeq G \cup F_n$, observe that 
	\[
	\GS[S_n] = \GpS[S_n]{H_n} \text{ and } \EdgsBtw[H_n]{S_{n+1} \setminus S_n}{S_n} = \EdgsBtw{S_{n+1} \setminus S_n}{S_n}.
	\]
	Thus, for each $n \in \N$, \cref{shortcutting:graphing_fsr} applies to $H_n$, $S_n$, and $F_{n+1}$, yielding a Borel edge sliding $\si_n$ of $\EdgsBtw[H_n]{S_{n+1}}{S_n}$ that moves a subgraph of $\EdgsBtw{S_{n+1} \setminus S_n}{S_n}$ into $F_{n+1}$ so that $\si_n(H_n)$ supergraphs $F_{n+1}$. In particular,
	\begin{equation}\label{eq:self-shortcutting:fixes_Sn}
		\Fx(\si_n) \supseteq S_n^2.
	\end{equation}
	
	\noindent Put $G_n \defeq G \rest{S_n}$.
	\begin{claim+}\label{claim:self-shortcutting:what_is_si_n}
		$\si_n$ is an $(E_{G_n} \vee F_n)$-based edge sliding of $\EdgsBtw{S_{n+1} \setminus S_n}{S_n}$.
	\end{claim+}
	\begin{pf}
		Follows from \labelcref{eq:self-shortcutting:fixes_Sn} and the following calculation: $\EdgsBtw[H_n]{S_{n+1}}{S_n} = H_n \rest{S_n} \cup \EdgsBtw[H_n]{S_{n+1} \setminus S_n}{S_n} = G_n \cup F_n \cup \EdgsBtw{S_{n+1} \setminus S_n}{S_n}$.
	\end{pf}
	
	\begin{claim+}\label{claim:self-shortcutting:bsi_n_graphs_Fn}
		$\bsi_n$ is an edge sliding of $G_n$ into $F_n$ graphing $F_n$.
	\end{claim+}
	\begin{pf}
		Follows by a straightforward induction on $n$ using \cref{claim:self-shortcutting:what_is_si_n}.
	\end{pf}
	
	\begin{claim+}\label{claim:self-shortcutting:si_n_edge_sliding_of_bsi_n}
		$\si_n$ is an edge sliding of $\bsi_n(G)$.
	\end{claim+}
	\begin{pf}
		By \cref{finite_itaration_connectivity_preserving}, $\bsi_n$ is connectivity preserving for $G_n$, so $\bsi_n(G_n)$ is a graphing of $E_{G_n}$; in fact, it is a graphing of $E_{G_n} \vee F_n$ by \cref{claim:self-shortcutting:bsi_n_graphs_Fn}. It now follows from \cref{claim:self-shortcutting:what_is_si_n} and \labelcref{item:railway_any_spanning_subgraph} that $\si_n$ is a $\bsi_n(G_n)$-based edge sliding of $\EdgsBtw{S_{n+1} \setminus S_n}{S_n}$.
	\end{pf}
	
	\begin{claim+}
		$\bsi_{n+1}(G)$ is a supergraphing of $F_{n+1}$.
	\end{claim+}
	\begin{pf}
		Firstly, by \labelcref{eq:self-shortcutting:fixes_Sn}, the set $(G_n \setminus F_n) \cup F_n \cup \si_n(G \setminus G_n)$ is equal to $\si_n(G) \cup F_n = \si_n(H_n)$, and hence is a supergraphing of $F_{n+1}$. Because $\IMv(\bsi_n) \subseteq F_n$, \cref{Mv_cap_im_subset_IMv} implies that $G_n \setminus F_n \subseteq \Fx(\bsi_n)$, so $\bsi_n(G_n) \supseteq G_n \setminus F_n$, which, together with \cref{claim:self-shortcutting:bsi_n_graphs_Fn}, implies that $\bsi_n(G_n) \cup \si_n(G \setminus G_n)$ is a supergraphing of $F_{n+1}$. Lastly, $\bsi_n(G_n) \cup \si_n(G \setminus G_n) = \bsi_{n+1}(G)$ because $\Fx(\bsi_n) \supseteq G \setminus G_n$ by \cref{claim:self-shortcutting:bsi_n_graphs_Fn}.
	\end{pf}
	
	Thus, $\si_n$ is an $F_n$-based edge sliding of $\bsi_n(G)$ graphing $F_{n+1}$, so \cref{connecting_unions_using_images} applies, finishing the proof.
\end{proof}


\section{Ergodic hyperfinite slid-subgraphs}\label{sec:erg_hyp_slid-subgraphs}

The goal of this section is to give a direct proof of the following weaker form of Tucker-Drob's theorem \cite{TuckerDrob:erg_hyp_subgraph}:

\begin{theorem}\label{ergodic_hyp_slid-subgraph}
	Any Borel locally countable p.m.p. ergodic graph $G$ on a standard probability space $(X,\mu)$ admits a Borel well-iterated edge slide $\tG$ that has an ergodic hyperfinite Borel subgraph $\tH$. In fact, given $\e > 0$, we can arrange so that $C_\mu(\tG \symdiff G) < \e$.
\end{theorem}

\begin{remark}
	We have stated the theorem for an ergodic graph $G$, but one can remove this assumption and get in the conclusion that $\tH$ is ergodic relative to $G$. This is done by a standard argument using the Ergodic Decomposition theorem, see \cite{Farrell} and \cite{Varadarajan}.
\end{remark}

\subsection{The step of the iteration}

Throughout, let $(X,\mu)$ be a standard finite measure space and let $G$ a measure-preserving locally countable ergodic Borel graph on it. We denote $E \defeq E_G$.

\begin{lemma}[The induction step]\label{main_lemma:step_of_iteration}
	Suppose we are given
	\begin{itemize}[\scriptsize$\triangleright$]
		\item a Borel graph $R \subseteq G$ such that $G \setminus R$ is component-infinite on an $E$-complete set,
		
		\item bounded Borel functions $w : X \to [1, \w), f : X \to \R$,
		
		\item an error tolerance $\e > 0$.
	\end{itemize}
	Then there are
	\begin{enumref}{i}{main_lemma:step_of_iteration}
		\item \label{item:main_lemma:step_of_iteration:tf} a Borel function $\tf : X \to [\infmu(f), \supmu(f)]$ such that $\tf$ equals $f$ on a Borel $\mu$-co-$\e$ set $\tX \subseteq X$, in particular, $\Lone{f - \tf} < \e \cdot \oscmu(f)$;
		
		\item \label{item:main_lemma:step_of_iteration:F} a bounded Borel \fsr $F \subseteq E$ such that $\tX = \dom(F)$ and $\osc_E(\meanw{F}(\tf)) \le \frac{2}{3} \osc_E(f)$, in fact, $\osc_E(\meanw{F}(\tf)) \le \frac{2}{3} |\MSwGf|$;
		
		\item \label{item:main_lemma:step_of_iteration:R'} a Borel subgraph $R' \subseteq G \setminus F$ of cost less than $\e$;
		
		\item \label{item:main_lemma:step_of_iteration:si} a Borel $R'$-based well-iterated edge sliding $\si$ of $G$ that graphs $F$ and satisfies $R \subseteq \Fx(\si)$ and $C_\mu(\Mv(\si)) < \frac{\e}{2}$; in particular, $C_\mu(\si(G) \symdiff G) < \e$.
	\end{enumref}
\end{lemma}

Putting $\de \defeq \frac{\e}{8}$, we devote the rest of this subsection to the proof of this lemma. 

In the course of the proof, we will construct several $G$-connected Borel \fsrs $F_0 \subseteq F_1 \subseteq \hdots \subseteq F_n \subseteq E$ and the last one, $F_n$, will be the output $F$ of the lemma. For the sake of simplifying the exposition, after constructing each $F_i$, we will \emph{pass to a concrete quotient by $F_i$}, by which we mean that we

\begin{center}
	\begin{tabular}{lccccc}
		replace & $X$, & $\mu$, & $G$, & $w$, & $f$ 
		\\
		by & $\Xmod{F_i}$, & $\mu \rest{\Xmod{F_i}}$, & $\Gmod{F_i}$, & $\wmod{F_i}$, & $\meanwf{F_i}$.
	\end{tabular}
\end{center}
When we successfully construct a desired subgraph $R'$ on a concrete quotient of $G$, making sure that its cost is less than $\de$, \cref{smooth_quotient_has_inverse} will grant a Borel lift of it to a subgraph of $G$ of the same cost, thus satisfying \cref{item:main_lemma:step_of_iteration:R'}. Similarly, \cref{lift_of_sliding} will take care of lifting the well-iterated edge sliding that we will also define on a quotient of $G$.

\subsubsection{Obtaining a component-infinite hyperfinite subgraph}

Let $D$ be the union of all infinite connected components of the graph $G \setminus R$, an $E$-complete set by the hypothesis. \cref{connecting_with_compl_sect} now gives a Borel \fsr $F_0 \subseteq E$ with $E$-cocompressible domain such that $G$ connects $F_0$ and each $F_0$-class intersects $D$. This $F_0$ may be unbounded, but removing a $\mu$-co-$\de$ $F_0$-invariant set from the domain of $F_0$, we may assume that it is bounded and its domain is $\mu$-co-$\de$; moreover, adding $\Diag(D)$ to it, we may assume that $X_0 \defeq \dom(F_0) \supseteq D$. For any infinite $(G \setminus R)$-connected component $C$, the set $[C]_{F_0}$ is connected by the graph $(G \setminus R) \cup (G \cap F_0)$ and these sets partition $X_0$; in other words, $\Gmod[(G \setminus R)]{F_0}$ is component-infinite on $\Xmod[(X_0)]{F_0}$. Passing to a concrete quotient by $F_0$, we may assume that $G \setminus R$ was component-infinite on $X_0$ to begin with, where $X_0$ is an $E$-invariant $\mu$-co-$\de$ Borel set.

Applying \cite{Kechris-Miller}*{Remark 23.3} to $G \rest{X_0}$, we get a component-infinite acyclic hyperfinite Borel subgraph $H \subseteq G \rest{X_0} \subseteq G \setminus R$. By Levitt's lemma (see \cite{Levitt} or \cite{Kechris-Miller}*{Lemma 22.1}), $C_\mu(H) = \mu(X_0) < \w$, so it follows from \cref{graph-hyperfinite} that there is a Borel $H$-connected bounded \fsr $F_1 \subseteq E_H$ with $\dom(F_1) = X_0$ such that $C_\mu(H \setminus F_1) < \de$. Thus, passing to a concrete quotient by $F_1$, we may assume that $C_\mu(H) < \de$ to begin with, where $H$ is an acyclic component-infinite hyperfinite Borel subgraph of $G \rest{X_0}$. In particular, $\dom(H) = X_0$ is $\mu$-co-$\de$.

The to-be-defined well-iterated edge sliding $\si$ will be that of $H$ and we will take as $R'$ a Borel subset of $H$, thus guaranteeing that $C_\mu(R') < \de < \e$.

\subsubsection{Packing with central sets}\label{subsubsec:packing_with_central_sets}

For a closed bounded interval $I$, call a set $V \in \FinX$ \emph{negative} (resp. \emph{central}, \emph{positive}) if $\meanwf{V}$ is in $I_-$ (resp. $I$, $I_+$).

Because the Borel function $x \mapsto \MSwGf[x]$ is $E$-invariant, the ergodicity of $E$ implies that it is constant modulo a $\mu$-null set, which we disregard. Thus, we simply write $\MSwGf$ for that constant interval and we let $I$ be the closed middle third of that interval, so $|I| = |\MSwGf| / 3$. 

Define $p : \FinX[X] \to [1, \w)$ by $U \mapsto |U|$. We will play the packing game $\PackG(E, p)$ (see \cref{defn:PackG}) as follows: assuming that Player 2 has last played an \fsr $F \subseteq E$, we put $S \defeq \dom(F)$, and define the response $\Phi(F)$ of Player 1 as the collection of all $U \in \FinE$ satisfying the following conditions:
\begin{enumenv}
	\item \label{item:step:packing:central} $U$ is central,
	
	\item \label{item:step:packing:F-invariant} $U$ is $F$-invariant,
	
	\item \label{item:step:packing:G-connected_over_H} $G \cup E_H$ connects $U$,
	
	\item \label{item:step:packing:HS-connected} each $E_H \rest{U}$-class $C \subseteq U$ is $\HS$-connected, and
	
	\item \label{item:step:packing:H-connected_on_S} for each $E_H \rest{U}$-class $C \subseteq U$, $C \cap S$ is $(H \cup F)$-connected.
\end{enumenv}

Let Player 1 play $\PackG(E, p)$ according to the strategy provided by \cref{winning_PackG}, and, assuming Player 1 has played $F_n$ in her $n^\text{th}$ move, we let Player 2 play $\Phi_n \defeq \Phi(F_n)$. Thus, the $F_n$ are increasing and $F_\w \defeq \bigcup_{n \in \N} F_n$ is finite and $p$-packed within $\Phi_\w \defeq \bigcup_{n \in \N} \Phi_n$, modulo an $E$-compressible, hence $\mu$-null, set, which we throw out. Whence, $\Phi_\w = \Phi(F_\w)$, so $F_\w$ is $p$-packed within $\Phi(F_\w)$.

Putting $F_n' \defeq F_n \cap E_H$ and $S_n \defeq \dom(F_n')$ for each $n \in \N$, as well as $S \defeq \dom(F_\w) = \bigcup_n S_n$, it is clear that \cref{self-shortcutting} applies to $(F_n')_n$ and $H$, yielding a Borel well-iterated edge sliding $\si_0$ of $\EdgsBtw[H]{S}{S}$ into $F_\w \cap E_H$ graphing $F_\w \cap E_H$. Because $G \cup E_H$ supergraphs $F_\w$, $\si_0(G)$ also supergraphs $F_\w$.

\subsubsection{All $E_H \rest{X \setminus S}$-classes have the same sign}\label{subsubsection:same_sign}

\begin{claim+}
	For each $E_H$-class $C \subseteq \dom(H)$ modulo $E$-compressible, $C \setminus S$ is either empty or infinite.
\end{claim+}
\begin{pf}
	The $E_H$-classes $C \subseteq \dom(H)$ with $C \setminus S$ finite form an $E_H$-smooth set $A$. Because the sets $C$ are infinite, $A$ is $E$-compressible.
\end{pf}

Thus, throwing out an $E$-compressible set, we assume that the last claim holds everywhere.

For each $E_H$-class $C \subseteq \dom(H)$, denote $C' \defeq C \setminus S$ and note that $C'$ is $\HS$-connected by \labelcref{item:shortcutting:connected_outside_of_S}. Whence, the map $x \mapsto \MS^w_{\HS \rest{C'}}(f)[x]$ is constant on $C'$, by \cref{invariance_of_means}, and we denote its value (a closed interval) by $I_{C'}$. 

For any interval $J$, we denote its closure by $\cl{J}$ and interior by $\Int(J)$.

\begin{claim+}[The key claim]
	There is a sign $* \in \set{\mathord{+}, \mathord{-}}$ such that for every $E_H$-class $C$, $I_{C'} \subseteq \cl{I_*}$.
\end{claim+}
\begin{pf}
	First note that for every $E_H$-class $C \subseteq \dom(H)$, no $U \in \Finw{C'}_{\HS}$ is central because that would contradict the maximality of $F_\w$ within $\Phi(F_\w)$. In particular, because $I_{C'}$ is an interval, $I_{C'} \subseteq \cl{I_-}$ or $I_{C'} \subseteq \cl{I_+}$.
	
	Suppose towards a contradiction that there are $E_H$-classes $C,D \subseteq \dom(H)$ with $I_{C'} \subseteq \cl(I_-)$ and $I_{D'} \subseteq \cl(I_+)$; in particular $C \ne D$.
	
	Because $E_H \subseteq E_G$, we can choose $C,D$ such that there is an edge $(c,d) \in G$ with $c \in C$ and $d \in D$. Completing $F_\w$ to an entire equivalence relation $\bF_\w \defeq F_\w \cup \Diag(X)$, put $U_0 \defeq [c]_{\bF_\w}$. If $U_0 \cap D \ne \0$, put $U_1 \defeq U_0$ and forget about $d$; otherwise, put $U_1 \defeq U_0 \disjU [d]_{\bF_\w}$. Either way, $U_1$ meets both $C$ and $D$, and satisfies conditions \crefrange{item:step:packing:F-invariant}{item:step:packing:H-connected_on_S}.
	
	By \labelcref{item:shortcutting:extending_outside_of_S}, there is a vertex $c' \in C'$ $\HS$-adjacent to $U_1$ and hence, there are arbitrarily large $U' \in \Finw{C'}_{\HS}$ disjoint from $U_1$ and containing $c'$ such that $U \defeq U_1 \cup U'$ is negative; this is because $I_{C'} \subseteq \cl{I_-}$ and $\Finw{C'_{\HS}}$ has no central sets. In fact, we take this $U'$ large enough so that $|U'| \ge |U_1|$ and 
	$$
	\De \defeq \frac{\Linf{f}\Linf{w}}{|U|_w} < \frac{|I|}{100}.
	$$
	Analogously, we also get $V' \in \Finw{D'}_{\HS}$ disjoint from $U$ but $\HS$-adjacent to it such that $V \defeq U \disjU V'$ is positive. Applying \cref{intermediate_value_property} to $U$ and $V$ (with respect to the graph $G \cup \HS$) and $r \defeq$ the midpoint of $I$, we get a central $W' \in \Finw{V'}_{\HS}$ such that $W'$ is $\HS$-adjacent to $U$ and $W \defeq U \disjU W'$ is central. $W$ is $p$-admissible for $F_\w$ because
	$$
	|W \setminus \dom(F_\w)| \ge |W \setminus U_1| \ge |U'| \ge |U_1| = p(U_1).
	$$
	Also, by construction, $W \in \Phi(F_\w)$, contradicting the $p$-packedness of $F_\w$ within $\Phi(F_\w)$.
\end{pf}

\subsubsection{Final packaging and sliding}

Putting $X' \defeq \dom(H) \setminus S$, recall that all $E_H \rest{X'}$-classes are infinite. Thus, applying \cref{entire_fsr_with_correct_calc} to $\HS \rest{X'}$ and $\e \defeq \frac{|I|}{2}$, we obtain, after throwing out an $E_H$-compressible hence $\mu$-null set, a Borel \fsr $F' \subseteq E_H \rest{X'}$ with $\dom(F') = X'$ such that each $F'$-class $U$ is $\HS$-connected and $\meanwf{U} \in I \cup I_*$.

Lastly, we apply \cref{shortcutting:graphing_fsr} to $\si_0(H)$, $S$, and $F'$, and obtain a Borel edge sliding $\si_1$ along a subgraph of $\si_0(H)$ that moves a subgraph of $\EdgsBtw[\si_0(H)]{\dom(F') \setminus S}{S} = \EdgsBtw[H]{\dom(F') \setminus S}{S}$ into $\HS \cap F'$ so that $\si_1(\si_0(H))$ supergraphs $F'$. Hence, putting $\tF \defeq F_\w \disjU F'$ and $\si \defeq \si_1 \circ \si_0$, we see that $\si$ is a well-iterated edge sliding of $H$, $\tH \defeq \si(H)$ supergraphs $\tF \cap E_H$, and hence, $\tG \defeq \si(G)$ supergraphs $\tF$. Because $H \subseteq G \setminus R$, $\si$ fixes $R$ pointwise. Moreover, because $\si$ moves a subgraph of $H$ into $\tF$, it fixes $R' \defeq H \setminus \tF$ pointwise, so $\si$ is an $R'$-based well-iterated edge sliding of $H$ graphing $\tF$ and fixing $R$ pointwise. Finally, $C_\mu(H) < \de$ and $\Mv(\si), R' \subseteq H$, so the costs of $\Mv(\si)$ and $R'$ are also less than delta, so all of the requirements of \labelcref{item:main_lemma:step_of_iteration:R',item:main_lemma:step_of_iteration:si} are met.

Noting that $\dom(\tF) = \dom(H)$ is $\mu$-co-$\de$ and that for each $\tF$-class $U$, $\meanwf{U} \subseteq I \cup I_*$, the \fsr $\tF$ is almost ready to be the output $F$ of the lemma with the minor wrinkle that it may not be bounded. But destroying no more than a measure $\de$ set of $\tF$-classes, we obtain an $\tF$-invariant $\mu$-co-$2\de$ Borel subset $X' \subseteq \dom(\tF)$ such that $F \defeq \tF \rest{X'}$ is bounded. Now $F$ indeed satisfies \labelcref{item:main_lemma:step_of_iteration:F}.

It remains to define $\tf : X \to [\infmu(f), \supmu(f)]$ by $\tf \rest{X'} \defeq f \rest{X'}$ and $\tf \rest{X \setminus X'} \defeq $ the midpoint of $I$, so $\tf$ clearly satisfies \labelcref{item:main_lemma:step_of_iteration:tf}, concluding the proof of \cref{main_lemma:step_of_iteration}.

\subsection{Proof of \cref{ergodic_hyp_slid-subgraph}}\label{subsection:proof_of_slid_hyp_subgraphs}

\subsubsection{Setup}

Assume that $G$ itself is not hyperfinite since there is nothing to prove otherwise. Thus, $\fep(G) > 0$ by \cref{fep=0->fvp=0->hyperfinite}.

Given $\e > 0$, our goal is to construct a Borel well-iterated edge slide $\tG$ of $G$ with $C_\mu(G \symdiff \tG) < \e$ and a hyperfinite Borel subgraph $H \subseteq \tG$ that is $\mu$-ergodic within $\tG$. By \cref{char_of_ergodicity}, it is enough to ensure that $E_H$ is $g$-Cauchy within $E \defeq E_G$ for every function $g$ in a dense collection $\mathcal{D} \subseteq L^1(X,\mu)$ of bounded Borel functions.

\bigskip

Fix a bijection $\gen{\cdot, \cdot} : \N \times \N \bij \N$ such that $\gen{r,c} < \gen{r, c'}$ for all natural numbers $r$ and $c < c'$. Define \emph{row} and \emph{column} decoding functions $\br,\bc : \N \to \N$ such that $\gen{\br(n), \bc(n)} = n$. Fix an enumeration of $\DC$ such that each function $g \in \DC$ appears infinitely many times.

Lastly, fix a $\de > 0$ that is less than $\frac{\e}{2^{10}}$ and $\frac{\fep(G)}{2^{10}}$, and put $\e_n \defeq \de \cdot 2^{-n}$ for each $n \in \N$.

\subsubsection{Iterating the induction step}

Using recursive applications of \cref{main_lemma:step_of_iteration} to concrete quotients, we obtain:
\begin{itemize}
	\item an increasing sequence $(F_n)_n$ of Borel \fsrs, where $F_0 \defeq \0$ and each $\dom(F_{n+1})$ is $\mu$-co-$\e_n$;
	
	\item a sequence of bounded Borel functions $(f_{r,c})_{r,c \in \N}$, where $f_{c,0} = g_c$ and 
	\begin{equation}\label{eq:frc:range}
		f_{r,c+1} : X \to [\infmu(f_{r,c}), \supmu(f_{r,c})]
	\end{equation}
	such that $f_{r,c+1}$ equals $f_{r,c}$ on $\dom(F_{\gen{r,c}})$; moreover,	
	\begin{equation}\label{eq:frc:var}
		\osc_E(\mean{F_{n+1}}(f_{r,c+1})) \le \frac{2}{3} |\MS_G(\mean{F_n}(f_{r,c}))|;
	\end{equation}
	
	\item a sequence $(R_n)_{n \ge 0}$ of Borel subgraphs of $G$, where $R_0 \defeq \0$, $R_n \cap F_{n+1} = \0$ and $C_\mu(R_n) < \e_n$, and we put $\bR_n \defeq \bigcup_{k < n} R_n$;
	
	\item a sequence $(\si_n)_{n \ge 0}$ of Borel edge-operators on $X$, where $\si_n$ is an $(\bR_n \cup F_{n-1})$-based well-iterated edge sliding of $\bsi_n(G)$ that graphs $F_n$ over $F_{n-1}$.
\end{itemize}
In addition,
\begin{equation}\label{fixing_Rn}
	R_n \subseteq \Fx(\si_k) \text{ for all $n,k \ge 0$}.
\end{equation}
Indeed, suppose that $F_n$, $f_{\bc(n), \br(n)}$, $\bR_n$, and $\bsi_n$ are defined. Taking a concrete quotient by $F_n$, apply \cref{main_lemma:step_of_iteration} to $(\Xmod{F_n}, \mu \rest{\Xmod{F_n}})$, $w \defeq \wmod[1]{F_n}$, $\Gmod[\bsi_n(G)]{F_n}$, $\Gmod[(\bR_n)]{F_n}$, $\mean{F_n}(f_{\bc(n), \br(n)})$, and $\e_n$, and lift the outputs back to the space $X$ obtaining $F_{n+1} \supseteq F_n$, $R_n \subseteq \bsi_n(G) \setminus F_{n+1}$, $f_{\bc(n), \br(n) + 1}$, and a well-iterated edge sliding $\si_n$. For each $k < n$, it now follows by \cref{Mv_cap_im_subset_IMv} from the facts that $R_n \cap F_n = \0$ and $\IMv(\si_k) \subseteq F_{k+1}$ that $R_n \subseteq \Fx(\si_k)$, which implies $R_n \subseteq \Fx(\bsi_n) \cap \bsi_n(G) \subseteq G$.

Thus, the entire $R \defeq \bigcup_{k \in \N} R_k$ is fixed pointwise by every $\si_n$, so, in fact, $\si_n$ is an $(R \cup F_{n-1})$-based well-iterated edge sliding that graphs $F_{n+1}$ over $F_n$. Hence, \cref{connecting_unions_using_images} applies yielding a Borel $R$-based well-iterated edge sliding $\si$ of $G$ that graphs $F \defeq \bigcup_{n \in \N} F_n$.

Denote by $\tG \defeq \si(G)$ and $\tH \defeq \si(G) \cap F$, so $\tH$ is a graphing of $F$. It remains to show that $(F_n)_n$ is $f$-Cauchy within $E$.

\subsubsection{The Cauchyness of $(F_n)_n$}

By \cref{char_of_ergodicity}, it is enough to check that $(F_n)_n$ is $f$-Cauchy for every $f$ in a dense subset of $L^1(X,\mu)$.

Using \labelcref{eq:frc:range}, \labelcref{eq:frc:var}, and the summability of the $\e_n$, the Borel-Cantelli lemma implies that for each fixed $c \in \N$, the sequence $(f_{r,c})_c$ converges pointwise a.e., as well as in $L^1$, and we let $f_{r,\w}$ denote its limit. In fact,

\begin{claim+}\label{frc_approx_frw}
	For each $r,c \in \N$, putting $n_c = \gen{r,c}$, there is an $F_{n_c}$-invariant $\mu$-co-$2\e_{n_c}$ Borel set $X_{n_c}$ such that
	$
	f_{r,\w} \rest{X_{n_c}} = f_{r,c} \rest{X_{n_c}} 
	\text{ and }
	\Linf{f_{r,\w} \rest{X \setminus X_{n_c}} - f_{r,c} \rest{X \setminus X_{n_c}}} \le 2 \e_n.
	$
	In particular, $\Lone{f_{r,\w} - f_{r,c}} \le 2 \e_n \oscmu(g_c)$.
\end{claim+}

\begin{claim+}
	$\DC' \defeq \set{f_{r,\w} : r \in \N}$ is dense in $\DC$ and hence in $L^1(X,\mu)$.
\end{claim+}
\begin{pf}
	By \cref{frc_approx_frw}, $\Lone{g_r - f_{r,\w}} \le 2 \cdot \e_n$ where $n \defeq \gen{r,0}$. But for each $g \in \DC$, we chose our enumeration such that $g = g_{r_k}$ for some subsequence $(r_k)_k$. Hence $f_{r_k, \w} \to_{L_1} g$ as $k \to \w$.
\end{pf}

It remains to show that $(F_n)_n$ is $f_{r,\w}$-Cauchy for every $r \in \N$. Fixing $r \in \N$ and $\e' > 0$, we need to find $n$ such that $F_n$ $\e'$-ties $f_{r,\w}$ within $E$. 

It follows from \labelcref{eq:frc:var} that $\osc_E(f_{r,c}) \to 0$ as $c \to \w$, so we take $c \in \N$ large enough such that $\e_n < \e'$, $n \defeq \gen{r,c}$ and $\osc_E(\mean{F_n}(f_{c,r})) < e'$. By \cref{frc_approx_frw}, $f_{r,\w}$ coincides with $f_{r,c}$ on an $F_{n_c}$-invariant $\mu$-co-$\e'$ Borel set $X'$. Thus, it is enough to check the $\e'$-tying condition \labelcref{eq:eps-ties} for $f_{r,c}$ and $F_n$ on $X'$, which is already ensured by the choice of $c$. \qed(\cref{ergodic_hyp_slid-subgraph})


\end{document}